\documentclass[12pt,oneside,english]{amsart}
\usepackage[T1]{fontenc}
\usepackage[latin9]{inputenc}
\usepackage{geometry}
\usepackage{color}
\definecolor{note_fontcolor}{rgb}{0.800781, 0.800781, 0.800781}
\usepackage{babel}
\usepackage{verbatim}
\usepackage{mathrsfs}
\usepackage{mathtools}
\usepackage{amstext}
\usepackage{amsthm}
\usepackage{amssymb}
\usepackage{graphicx}
\usepackage[unicode=true,pdfusetitle,
 bookmarks=true,bookmarksnumbered=false,bookmarksopen=false,
 breaklinks=false,pdfborder={0 0 1},backref=false,colorlinks=false]
 {hyperref}

\makeatletter

\newenvironment{lyxgreyedout}
  {\textcolor{note_fontcolor}\bgroup\ignorespaces}
  {\ignorespacesafterend\egroup}

\numberwithin{equation}{section}
\numberwithin{figure}{section}
\theoremstyle{plain}
\newtheorem{thm}{\protect\theoremname}[section]
\theoremstyle{definition}
\newtheorem{defn}[thm]{\protect\definitionname}
\theoremstyle{plain}
\newtheorem{cor}[thm]{\protect\corollaryname}
\theoremstyle{plain}
\newtheorem{prop}[thm]{\protect\propositionname}
\theoremstyle{remark}
\newtheorem{rem}[thm]{\protect\remarkname}
\theoremstyle{plain}
\newtheorem{lem}[thm]{\protect\lemmaname}
\theoremstyle{remark}
\newtheorem*{claim*}{\protect\claimname}
\theoremstyle{plain}
\newtheorem*{lem*}{\protect\lemmaname}
\theoremstyle{definition}
\newtheorem{example}[thm]{\protect\examplename}

\usepackage{environ}
\RenewEnviron{lyxgreyedout}{}

\@ifundefined{showcaptionsetup}{}{%
 \PassOptionsToPackage{caption=false}{subfig}}
\usepackage{subfig}
\makeatother

\providecommand{\claimname}{Claim}
\providecommand{\corollaryname}{Corollary}
\providecommand{\definitionname}{Definition}
\providecommand{\examplename}{Example}
\providecommand{\lemmaname}{Lemma}
\providecommand{\propositionname}{Proposition}
\providecommand{\remarkname}{Remark}
\providecommand{\theoremname}{Theorem}

\begin{document}
\title{Diophantine approximations on random fractals}
\author{yiftach dayan}
\date{Dec 19, 2019}
\begin{abstract}
We show that fractal percolation sets in $\mathbb{R}^{d}$ almost
surely intersect every hyperplane absolutely winning (HAW) set with
full Hausdorff dimension. In particular, if $E\subset\mathbb{R}^{d}$
is a realization of a fractal percolation process, then almost surely
(conditioned on $E\neq\emptyset$), for every countable collection
$\left(f_{i}\right)_{i\in\mathbb{N}}$ of $C^{1}$ diffeomorphisms
of $\mathbb{R}^{d}$, $\dim_{H}\left(E\cap\left(\bigcap_{i\in\mathbb{N}}f_{i}\left(\text{BA}_{d}\right)\right)\right)=\dim_{H}\left(E\right)$,
where $\text{BA}_{d}$ is the set of badly approximable vectors in
$\mathbb{R}^{d}$. We show this by proving that $E$ almost surely
contains hyperplane diffuse subsets which are Ahlfors-regular with
dimensions arbitrarily close to $\dim_{H}\left(E\right)$. 

We achieve this by analyzing Galton-Watson trees and showing that
they almost surely contain appropriate subtrees whose projections
to $\mathbb{R}^{d}$ yield the aforementioned subsets of $E$. This
method allows us to obtain a more general result by projecting the
Galton-Watson trees against any similarity IFS whose attractor is
not contained in a single affine hyperplane. Thus our general result
relates to a broader class of random fractals than fractal percolation.
\end{abstract}

\maketitle

\section{Introduction\label{sec:Introduction}}

\subsection{The set $\text{BA}_{d}$. }

The field of Diophantine approximations deals with approximations
of real numbers and vectors by rationals, where the idea is to keep
the denominators as small as possible. A theorem by Dirichlet implies
that for every $v\in\mathbb{R}^{d}$, there exist infinitely many
$\left(P,q\right)\in\mathbb{Z}^{d}\times\mathbb{N}$, such that

\[
\left\Vert v-\frac{P}{q}\right\Vert _{\infty}<\frac{1}{q^{1+\frac{1}{d}}}.
\]
This result leads to one of the key definitions in the field - the
badly approximable vectors.
\begin{defn}
\label{def:BA}A vector $v\in\mathbb{R}^{d}$ is called \textit{badly
approximable} if there exists some $c>0$, s.t. for every $\left(P,q\right)\in\mathbb{Z}^{d}\times\mathbb{N}$,
\[
\left\Vert v-\frac{P}{q}\right\Vert \geq\frac{c}{q^{1+\frac{1}{d}}}.
\]
The set of all badly approximable vectors in $\mathbb{R}^{d}$ is
denoted by $\text{BA}_{d}$.
\end{defn}

Throughout this paper, $\left\Vert \cdot\right\Vert $ is the Euclidean
norm, which is the only norm on $\mathbb{R}^{d}$ to be considered
from this point forward. Given $x\in\mathbb{R}^{d}$ and $r>0$, $B_{r}\left(x\right)=\left\{ y\in\mathbb{R}^{d}:\,\left\Vert x-y\right\Vert <r\right\} $,
and finally, given a set $S\subseteq\mathbb{R}^{d}$, and $\varepsilon>0$,
$S^{\left(\varepsilon\right)}$ is the \emph{$\varepsilon$-neighborhood}
of $S$ defined by $S^{\left(\varepsilon\right)}=\bigcup\limits _{x\in S}B_{\varepsilon}\left(x\right)$.
Note that using any other norm in Definition \ref{def:BA} would result
in an equivalent definition.

The set $\text{BA}_{d}$ is one of the most intensively investigated
sets in the field of Diophantine approximations. It is well known
that $\text{BA}_{d}$ has Lebesgue measure 0. On the other hand, it
has Hausdorff dimension $d$, which makes it reasonable to surmise
that it intersects various kinds of fractal sets. Indeed, in recent
years there has been a lot of interest, and many results, about the
intersection of $\text{BA}_{d}$ with fractal sets. A key result in
this line of research is due to Broderick, Fishman, Kleinbock, Reich
and Weiss \cite{Broderick2012319}, which deals with the intersection
of $\text{BA}_{d}$ with a certain kind of fractals called \emph{hyperplane
diffuse}. 
\begin{defn}
\label{def:diffuse}Given $\beta>0$, a closed set $K\subseteq\mathbb{R}^{d}$
is called \emph{hyperplane $\beta$ - diffuse} if the following holds: 

$\exists\xi_{0}>0,\,\forall\xi\in\left(0,\,\xi_{0}\right),\,\forall x\in K,\,\forall\mathcal{L}\subset\mathbb{R}^{d}$
affine hyperplane, 
\[
K\cap B_{\xi}\left(x\right)\setminus\mathcal{L}^{\left(\beta\xi\right)}\neq\emptyset.
\]
A set is called \emph{hyperplane diffuse} if it is hyperplane $\beta$
- diffuse for some $\beta$.
\end{defn}

This turns out to be a quite natural property for fractals and many
interesting fractals are known to be hyperplane diffuse, especially
when they have some self similarity. see Theorems 1.3 - 1.5 in \cite{Das2016b}
for some examples. 

In \cite{Broderick2012319} it was shown that if $K\subset\mathbb{R}^{d}$
is hyperplane diffuse, then $\dim_{H}\left(K\cap\text{BA}_{d}\right)>0$.
Moreover, if $K$ is also Ahlfors-regular (defined below) then $\dim_{H}\left(K\cap\text{BA}_{d}\right)=\dim_{H}\left(K\right)$. 
\begin{defn}
For any $\delta>0$, a measure $\mu$ on $\mathbb{R}^{d}$ is called
$\delta$-Ahlfors-regular , if $\exists c_{1},c_{2}>0$ s.t. $\forall\rho\in\left(0,1\right)$,
$\forall x\in\text{supp}\left(\mu\right)$, 
\[
c_{1}\rho^{\delta}\leq\mu\left(B_{\rho}\left(x\right)\right)\leq c_{2}\rho^{\delta}
\]
and Ahlfors-regular if it is $\delta$-Ahlfors-regular for some $\delta>0$.
A set $K\subset\mathbb{R}^{d}$ is called Ahlfors-regular (resp. $\delta$-Ahlfors-regular)
if there exists an Ahlfors-regular (resp. $\delta$-Ahlfors-regular)
measure $\mu$ on $\mathbb{R}^{d}$ s.t. $\text{supp}\left(\mu\right)=K$. 
\end{defn}

This result became a main tool for studying intersections of $\text{BA}_{d}$
with fractals. The above statement is in fact a corollary of a more
general theorem, where the set $\text{BA}_{d}$ is replaced by an
arbitrary hyperplane absolute winning (HAW) set. These are sets which
are winning in a certain game called the hyperplane absolute game,
which we shall describe in section \ref{sec:Hyperplane-absolute-game}.
Note that the set $\text{BA}_{d}$ is HAW \cite{Broderick2012319}.
Thus, the more general theorem is the following.
\begin{thm}[\cite{Broderick2012319}]
\label{thm:HAW intersects diffuse sets}Let $K\subset\mathbb{R}^{d}$
be hyperplane diffuse. Then there exists a constant $C>0$, s.t. $\forall S\subseteq\mathbb{R}^{d}$
HAW, $\dim_{H}\left(K\cap S\right)>C$. Moreover, if $K$ is Ahlfors-regular
then $\dim_{H}\left(K\cap S\right)=\dim_{H}\left(K\right)$.
\end{thm}

Two important properties of HAW sets are the following:
\begin{thm}
\emph{\label{thm:HAW is robust}\cite[Proposition 2.3]{Broderick2012319}}
\end{thm}

\begin{enumerate}
\item \emph{Any countable intersection of HAW sets is HAW.}
\item \emph{Any image of a HAW set under a $C^{1}$ diffeomorphism of $\mathbb{R}^{d}$
is HAW.}
\end{enumerate}
Theorem \ref{thm:HAW is robust} implies for example that if $K\subseteq\mathbb{R}^{d}$
is hyperplane diffuse, then for every sequence $\left(f_{n}\right)_{n\in\mathbb{N}}$
of \emph{$C^{1}$ }diffeomorphisms of \emph{$\mathbb{R}^{d}$, }the
intersection\emph{ $K\cap\left(\bigcap_{i\in\mathbb{N}}f_{i}\left(\text{BA}_{d}\right)\right)$
}has positive Hausdorff dimension, and if $K$ is also Ahlfors-regular
then the Hausdorff dimension of the intersection is maximal, i.e.,
equal to $\dim_{H}\left(K\right)$.

\subsection{Random fractals}

In this paper we deal with a natural model of random fractals which
we will refer to as Galton-Watson fractals. This model may be described
as follows. Suppose we are given a finite IFS $\Phi=\left\{ \varphi_{i}:\mathbb{R}^{d}\to\mathbb{R}^{d}\right\} _{i\in\Lambda}$
of contracting similarity maps  with attractor $K$ (these notions
will be explained in more detail in subsection \ref{subsec:IFSs-and-projections}.
See also \cite{falconer2013fractal} for a good exposition of this
topic). $\Phi$ defines a coding map $\gamma_{\Phi}:\Lambda^{\mathbb{N}}\to\mathbb{R}^{d}$
given by $\gamma_{\Phi}\left(i\right)=\bigcap\limits _{n=1}^{\infty}\varphi_{i_{1}}\circ...\circ\varphi_{i_{n}}\left(K\right)$%
\begin{lyxgreyedout}
$K$ is defined as the attractor which does not involve the coding
map, so why is the definition circular?%
\end{lyxgreyedout}
. Note that $\gamma_{\Phi}\left(\Lambda^{\mathbb{N}}\right)=K$. Let
$W$ be a random variable taking values in the finite set $2^{\Lambda}$.
We construct a Galton-Watson tree by iteratively choosing at random
the children of each element of the tree as realizations of independent
copies of $W$, starting from the root, namely $\emptyset$. By concatenating
each child to its parent, this defines a random subset of the symbolic
space $\Lambda^{\mathbb{N}}$ which we then project using $\gamma_{\Phi}$
to yield a random fractal $E\subset\mathbb{R}^{d}$ (which is contained
in \emph{K}). See Figure \ref{fig:A-Galton-Watson-fractal} for an
illustrative example.

\begin{figure}
\subfloat[The Sierpinski triangle]{\includegraphics[scale=0.5]{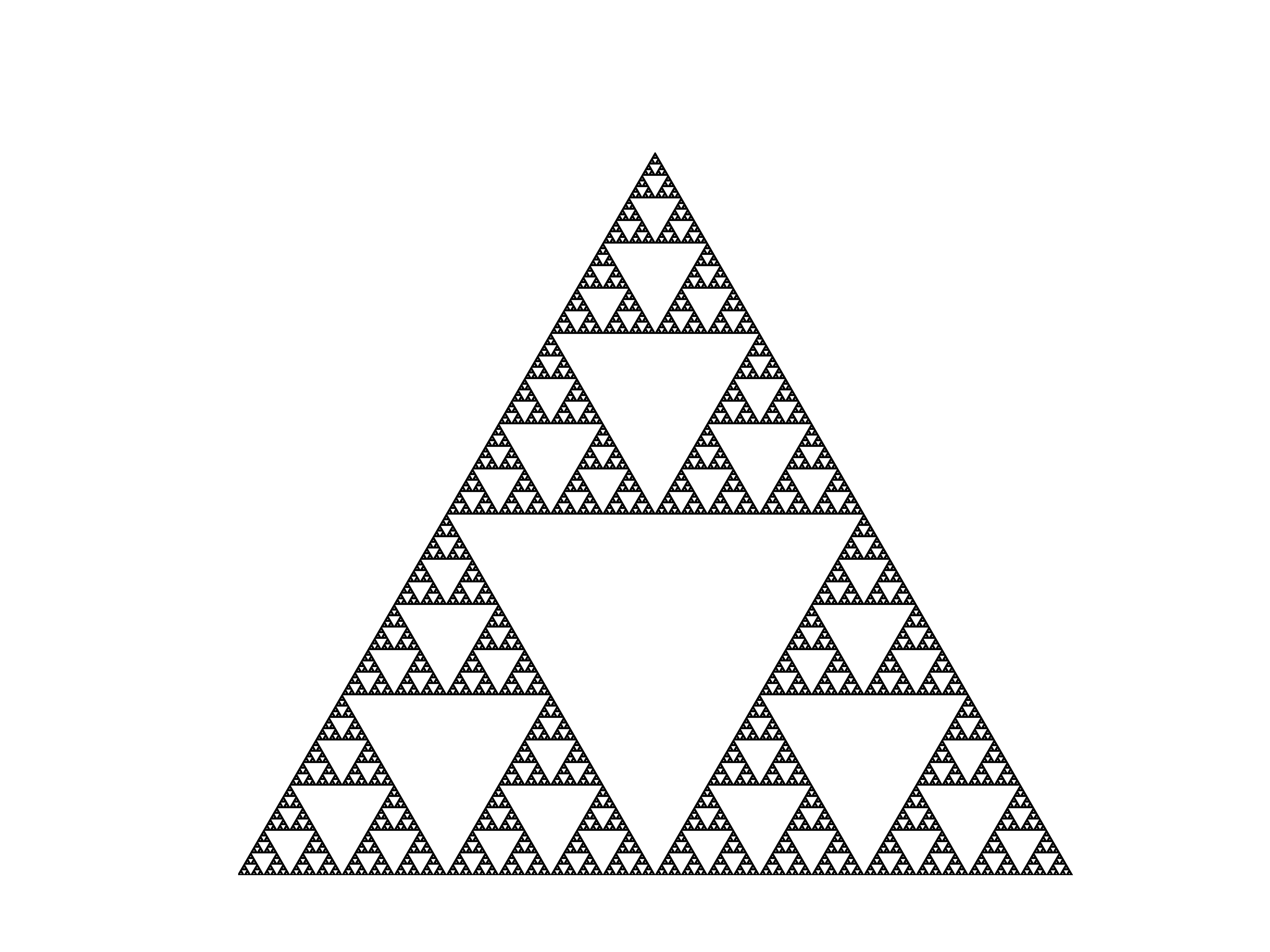}}\subfloat[The Galton-Watson fractal]{\includegraphics[scale=0.5]{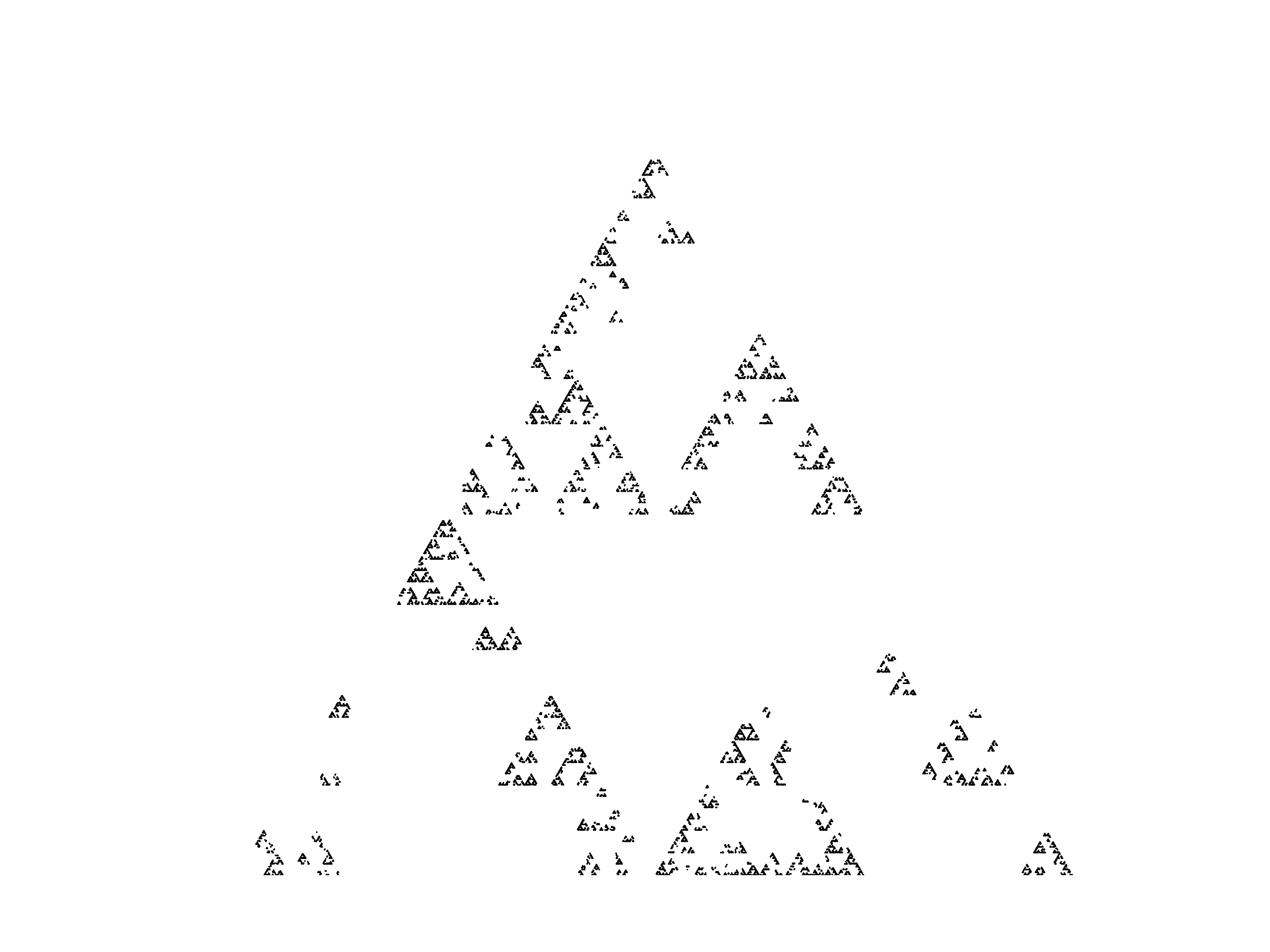}}\caption{\label{fig:A-Galton-Watson-fractal}An approximate realization of
a Galton-Watson fractal which is constructed using an IFS whose attractor
is the Sierpinski triangle, where $W$ has the distribution of a Bernoulli
process on $\left\{ 0,1\right\} ^{\Lambda}$ with parameter $p=0.8$.}
\end{figure}
 Throughout the paper, we shall always assume that $\forall i\in\Lambda,\,\mathbb{P}\left(i\in W\right)>0$.
Note that it is possible that at some level of the tree, no element
survives and the process dies out. If this occurs we say that the
process is extinct, and the resulting limit set is $E=\emptyset$.
It is a well known fact that unless $\left|W\right|=1$ almost surely,
$\mathbb{E}\left(\left|W\right|\right)\leq1\iff E=\emptyset\text{ a.s.}$
(see e.g. \cite[Proposition 5.1 ]{Lyons2016}). The case $\mathbb{E}\left(\left|W\right|\right)>1$
is called \emph{supercritical} and we shall assume this property throughout
this paper. Another well known fact is that if $\Phi$ satisfies the
open set condition (abbreviated to OSC and will be defined in subsection
\ref{subsec:IFSs-and-projections}), then a.s. conditioned on nonextinction
$\dim_{H}\left(E\right)=\delta$ where $\delta$ is the unique number
satisfying $\mathbb{E}\left(\sum\limits _{i\in W}r_{i}^{\delta}\right)=1$,
and $r_{i}$ is the contraction ratio of $\varphi_{i}$ for each $i\in\Lambda$.

A specific example of Galton-Watson fractals that the reader should
keep in mind is that of fractal percolation (AKA Mandelbrot percolation)
which we now describe. Fix some $p\in[0,1]$ and some integer $b\geq2$.
Let $E_{0}\subseteq\mathbb{R}^{d}$ be the unit cube. Divide $E_{0}$
to $b^{d}$ closed subcubes of equal volume. Now, independently, retain
each subcube with probability $p$ or discard it with probability
$1-p$. Let $E_{1}$ be the union of all surviving subcubes. Next,
for each surviving subcube in $E_{1}$ we follow the same procedure.
The union of all surviving subcubes in this step will be denoted by
$E_{2}$. We continue in the same fashion to produce a nested sequence
$E_{0}\supseteq E_{1}\supseteq E_{2}\supseteq...$ where each set
$E_{i}$ is the union of the surviving subcubes of level $i$ of the
process. Eventually we take $E=\underset{i\in\mathbb{N}}{\bigcap}E_{i}$.
See Figure \ref{fig:fractal percolation} for an example.
\begin{figure}
\subfloat{\includegraphics[scale=0.11]{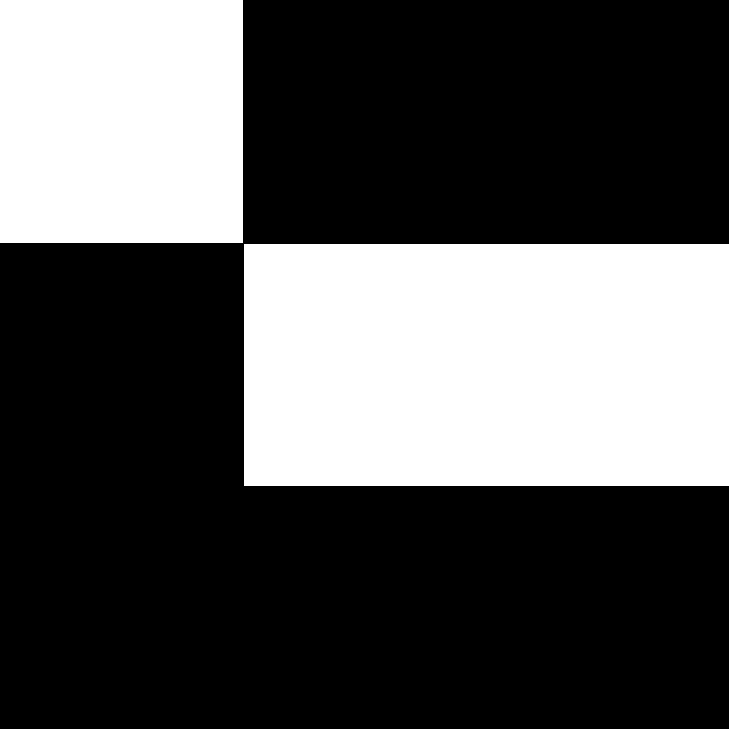}}\hspace{1cm}
\subfloat{\includegraphics[scale=0.11]{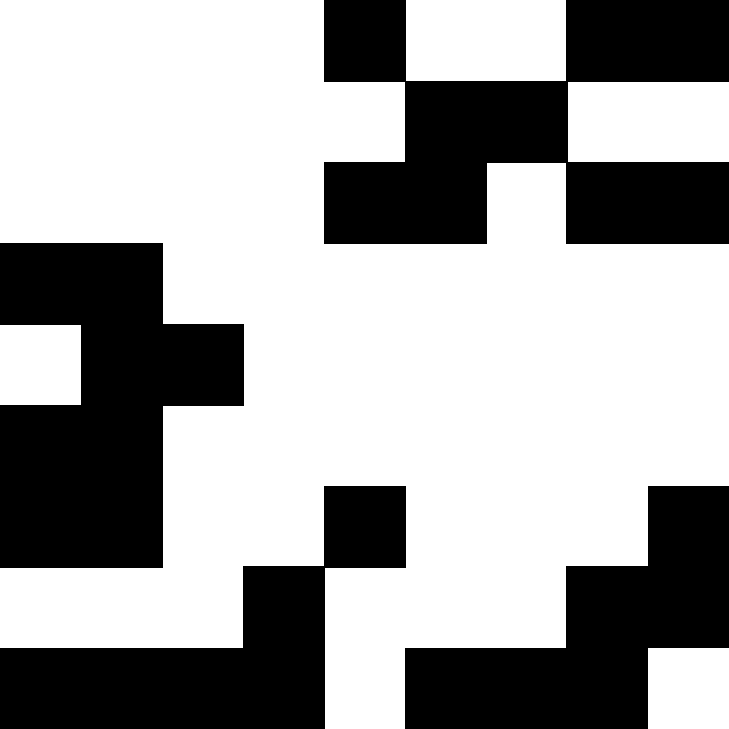}}\hspace{1cm}
\subfloat{\includegraphics[scale=0.11]{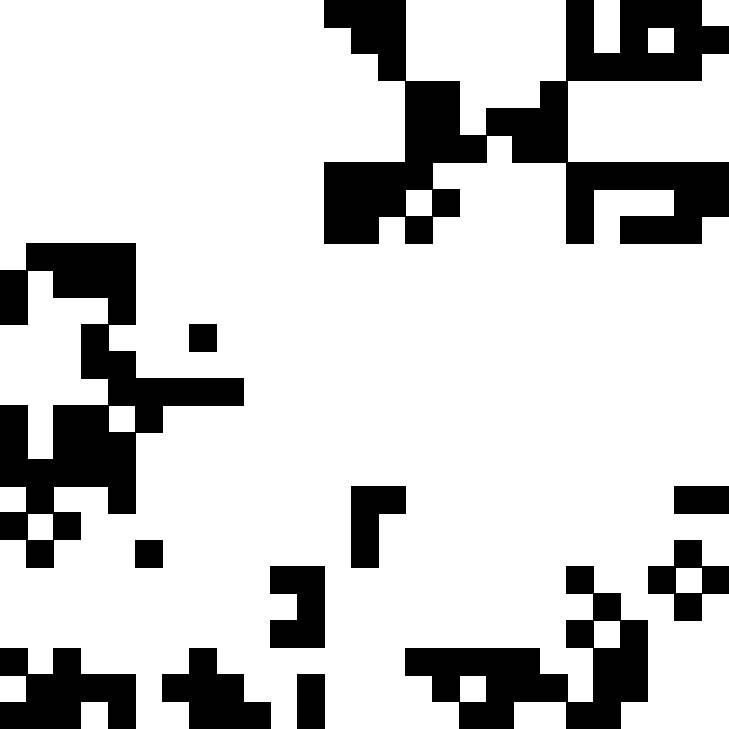}}\hspace{1cm}
\subfloat{\includegraphics[scale=0.11]{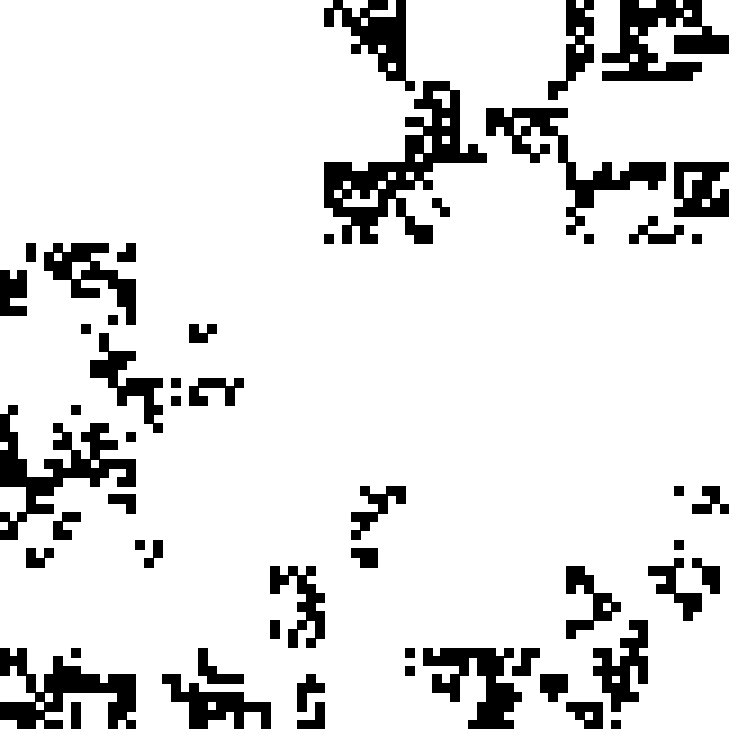}}\caption{\label{fig:fractal percolation}A realization of the first 4 steps
of a fractal percolation process in $\mathbb{R}^{2}$ with $b=3$
and $p=0.6$.}
\end{figure}

Being a relatively simple and natural model, fractal percolation has
been extensively studied over the years%
\begin{lyxgreyedout}
Should I add some references here?%
\end{lyxgreyedout}
.

All the notions raised in this subsection will be defined in a more
formal and detailed manner in section \ref{sec:Labeled-GW Trees}. 

\subsection{Main result and applications}

\subsubsection{Main theorem}

A condition which will recur in this paper is that a Galton-Watson
fractal is not a.s. contained in an affine hyperplane. Such a Galton-Watson
fractal will be referred to as \emph{non-planar}. Since we make the
assumption that $\forall i\in\Lambda,\,\mathbb{P}\left(i\in W\right)>0$,
non-planarity is essentially a property of the underlying IFS. More
precisely, if $E$ is a supercritical Galton-Watson fractal w.r.t.
a similarity IFS $\Phi=\left\{ \varphi_{i}\right\} _{i\in\Lambda}$
and offspring distribution $W$, then $E$ is non-planar iff the attractor
of $\Phi$ is not contained in an affine hyperplane. This fact as
well as some other equivalent conditions to non-planarity are proved
in Proposition \ref{prop: equivalent conditions for not in a hyperplane}.
Note that by definition non-planar Galton-Watson fractals are supercritical.

The main theorem of this paper is the following:
\begin{thm}
\label{thm: main generalized}Let E be a non-planar Galton-Watson
fractal w.r.t. a similarity IFS $\Phi$. Then a.s. conditioned on
nonextinction, 
\[
\exists C>0,\,\forall S\subset\mathbb{R}^{d}\text{ HAW, }\dim_{H}\left(E\cap S\right)>C.
\]
Moreover, if in addition $\Phi$ satisfies the OSC, then a.s. conditioned
on nonextinction, 
\[
\forall S\subset\mathbb{R}^{d}\text{ HAW, }\dim_{H}\left(E\cap S\right)=\dim_{H}\left(E\right).
\]
\end{thm}

The reader should pay special attention to the order of the quantifiers
in Theorem \ref{thm: main generalized} (a.s. $\forall S\subset\mathbb{R}^{d}$...)
which is the stronger form as the collection of HAW sets is uncountable. 

The proof of Theorem \ref{thm: main generalized} is interesting mainly
because in many cases (fractal percolation for example) the Galton-Watson
fractal is a.s. not hyperplane diffuse (see Corollaries \ref{cor:GW fractals are not hyperplane diffuse},
\ref{cor:Fractal percolation is not hyperplane diffuse} ). Therefore
in order to prove Theorem \ref{thm: main generalized} we prove the
following:
\begin{thm}
\label{thm:Main theorem diffuse subses}Let E be a non-planar Galton-Watson
fractal w.r.t. a similarity IFS $\Phi$. Then a.s. conditioned on
nonextinction, $E$ contains a hyperplane diffuse subset. Moreover,
if $\Phi$ satisfies the OSC, then a.s. conditioned on nonextinction,
E contains a sequence of subsets $\left(D_{n}\right)_{n\in\mathbb{N}}$,
s.t. for each $n\in\mathbb{N}$, $D_{n}\subseteq E$ is hyperplane
diffuse and Ahlfors-regular, and $\dim_{H}\left(D_{n}\right)\nearrow\dim_{H}\left(E\right).$
\begin{lyxgreyedout}
Should I mention that this approach was taken before by Simmons et
al?%
\end{lyxgreyedout}
\end{thm}

\subsubsection{Application to $\text{BA}_{d}$}

Applying Theorem \ref{thm: main generalized} to $\text{BA}_{d}$,
together with Theorem \ref{thm:HAW is robust}, yields the following
immediate corollary.
\begin{cor}
Let E be a non-planar Galton-Watson fractal w.r.t. a similarity IFS
$\Phi$. Then a.s. conditioned on nonextinction, there exists a constant
$C>0$ s.t. for every $\left(f_{i}\right)_{i\in\mathbb{N}}$ a sequence
of $C^{1}$ diffeomorphisms of $\mathbb{R}^{d}$,
\[
\dim_{H}\left(E\cap\left(\bigcap_{i\in\mathbb{N}}f_{i}\left(\text{BA}_{d}\right)\right)\right)>C.
\]
Moreover, if in addition $\Phi$ satisfies the OSC, then a.s. conditioned
on nonextinction, for every $\left(f_{i}\right)_{i\in\mathbb{N}}$
a sequence of $C^{1}$ diffeomorphisms of $\mathbb{R}^{d}$,
\[
\dim_{H}\left(E\cap\left(\bigcap_{i\in\mathbb{N}}f_{i}\left(\text{BA}_{d}\right)\right)\right)=\dim_{H}\left(E\right).
\]
\end{cor}

\subsubsection{Absolutely non-normal numbers and a generalization}

Since Theorem \ref{thm: main generalized} deals with any HAW set,
one may consider other interesting sets which are known to be HAW.
One such set is the set of \textit{absolutely non-normal numbers. }
\begin{defn}
Let $2\leq a\in\mathbb{N}$. For $x\in\mathbb{R}$, let $\left(x_{1},x_{2},...\right)$
be the digital expansion of the fractional part of $x$ in base $a$.
Then $x$ is \textit{normal} to base $a$ if $\forall n\in\mathbb{N},$
for every word $\omega\in\left\{ 0,1,...,a-1\right\} ^{n}$,
\[
\lim_{N\rightarrow\infty}\dfrac{1}{N}\left(\#\mbox{ occurrences of }\omega\mbox{ in }x_{1},x_{2},...,x_{N}\right)=a^{-n}
\]
Following \cite{bugeaud2012distribution}, $x$ will be called an
\textit{absolutely non-normal number} if it is normal to no base $2\leq a\in\mathbb{N}$.
By ergodicity of Bernoulli shifts, the set of numbers in the unit
interval which are normal to every integer base has Lebesgue measure
1. However, in \cite[Theorem 2.6]{Broderick2012319} following the
ideas of Schmidt \cite{Schmidt1966}, it was shown that the set of
absolutely non-normal numbers is HAW. In fact, a stronger result was
proved - the set of points whose orbit under multiplication by any
positive integer (mod 1) is not dense is HAW.

A generalization of this for higher dimensions is given by the following.
Let $\mathbb{T}^{d}=\mathbb{R}^{d}/\mathbb{Z}^{d}$ be the $d$ -
dimensional torus, and let $\pi:\mathbb{R}^{d}\rightarrow\mathbb{T}^{d}$
be the projection map. For every matrix $R\in GL_{d}\left(\mathbb{Q}\right)$
with integer entries, and every point $y\in\mathbb{T}^{d}$, we shall
denote 
\[
\mathcal{E}\left(R,y\right)=\left\{ x\in\mathbb{R}^{d}:\,y\notin\overline{\left\{ \pi\left(R^{k}x\right):\,k\in\mathbb{N}\right\} }\right\} 
\]
\end{defn}

\begin{prop}
\cite[Theorem 2.6]{Broderick2012319} For every nonsingular semisimple
matrix with integer entries $R\in GL_{d}\left(\mathbb{Q}\right)$,
and every point $y\in\mathbb{T}^{d}$, $\mathcal{E}\left(R,y\right)$
is HAW.
\end{prop}

In particular, lifting to $\mathbb{R}^{d}$ the set of points whose
orbit under $R$ is not dense in $\mathbb{T}^{d}$, yields a HAW set.
A further generalization of this theorem which relates to lacunary
sequences of matrices may be found in \cite[Theorem 1.3]{Broderick20111095}. 
\begin{cor}
Let E be a non-planar Galton-Watson fractal w.r.t. a similarity IFS
$\Phi$. Then a.s. conditioned on non-extinction, $\exists C>0$ s.t.
for every sequence of nonsingular semisimple matrices with integer
entries $R_{i}\in GL_{d}\left(\mathbb{Q}\right)$, every sequence
of points $y_{i}\in\mathbb{T}^{d}$, and every sequence $\left(f_{i}\right)_{i\in\mathbb{N}}$
of $C^{1}$ diffeomorphisms of $\mathbb{R}^{d}$, 
\[
\dim_{H}\left(E\cap\left(\bigcap_{i\in\mathbb{N}}f_{i}\left(\mathcal{E}\left(R_{i},y_{i}\right)\right)\right)\right)>C.
\]
Moreover, if $\Phi$ satisfies the OSC, then a.s. conditioned on nonextinction,
for every sequences $R_{i}$, $y_{i}$ and $f_{i}$ as above, 
\[
\dim_{H}\left(E\cap\left(\bigcap_{i\in\mathbb{N}}f_{i}\left(\mathcal{E}\left(R_{i},y_{i}\right)\right)\right)\right)=\dim_{H}\left(E\right).
\]
\end{cor}

Note that in the special case of $d=1$, under the above conditions,
a.s. conditioned on nonextinction, the Hausdorff dimension of the
absolutely non-normal numbers in $E$ is bounded from below by some
positive constant, and in case $\Phi$ satisfies the OSC, this dimension
is equal to $\dim_{H}\left(E\right)$. 

\subsection{Known results}

In the special case of fractal percolation, a weaker version of Theorem
\ref{thm: main generalized} may be derived by known results. This
goes through the following theorem by Hawkes \cite{HawkesJohn1981TGba}
(see also \cite[Theorem 9.5]{morters2010brownian}).
\begin{thm}
\label{thm:Hawkes 1981}Let E be a limit set of a supercritical fractal
percolation process with parameters b,p. Let $A\subset\left[0,1\right]^{d}$
be a fixed  set s.t. $\dim_{H}\left(A\right)+\log_{b}p>0$. Then
\[
\text{esssup}\dim_{H}\left(A\cap E\right)=\dim_{H}A+\log_{b}p.
\]
\end{thm}

Using Hawkes' theorem it is not hard to get the following.
\begin{thm}
\label{thm:fractal percolation intersects BA - Haweks}Let E be a
limit set of a supercritical fractal percolation process and let $S\subseteq\mathbb{R}^{d}$
be a HAW set. Then a.s. conditioned on non-extinction, 
\[
\dim_{H}\left(E\cap S\right)=\dim_{H}\left(E\right).
\]
\end{thm}

The proof of Theorem \ref{thm:fractal percolation intersects BA - Haweks}
follows immediately from Theorem \ref{thm:Hawkes 1981} once the following
general observation about HAW sets is made (see Remark \ref{rem:G inv. HAW sets}):
Let $S\subseteq\mathbb{R}^{d}$ be HAW, and consider the set 
\[
\tilde{S}=\bigcap\limits _{\left(r,q\right)\in\mathbb{Q}\times\mathbb{Q}^{d}}rS+q.
\]
$\tilde{S}$ is also HAW, it is invariant under rational scaling and
translations, and is contained in $S$. 

The proof of Theorem \ref{thm:fractal percolation intersects BA - Haweks}
for $S$ follows from the proof for $\tilde{S}$ which is now left
as an exercise for the reader.
\begin{rem}
Note that supercritical fractal percolation processes satisfy the
requirements of Theorem \ref{thm: main generalized}, i.e., the corresponding
IFS satisfies the OSC, and the limit set is a.s. not contained in
an affine hyperplane. Therefore, Theorem \ref{thm: main generalized}
is more general than Theorem \ref{thm:fractal percolation intersects BA - Haweks}. 

Theorem \ref{thm: main generalized} generalizes Theorem \ref{thm:fractal percolation intersects BA - Haweks}
in two aspects. First, the order of the quantifiers in Theorem \ref{thm: main generalized}
is stronger than in Theorem \ref{thm:fractal percolation intersects BA - Haweks}
and provides information about intersections of the random fractal
in question with every HAW set simultaneously. Second, Theorem \ref{thm: main generalized}
allows arbitrary Galton-Watson fractals and is not restricted to fractal
percolation.

\end{rem}

\subsection{Structure of the paper}

The main goal of this paper is to prove Theorem \ref{thm: main generalized}.
It is proved as a corollary of Theorem \ref{thm:Main theorem diffuse subses}
which will be the focus of this paper. We first prove the theorem
for the special case of fractal percolation sets as it includes most
of the ideas of the proof of the general case, but is much cleaner
and contains less complicated notations and definitions. This should
make the ideas of the proof of the general case clearer. After that,
we prove Theorem \ref{thm:Main theorem diffuse subses} in its full
generality (the proof does not depend on the proof for fractal percolation,
so the latter may be skipped).

The structure of the paper is as follows: In section \ref{sec:Labeled-GW Trees}
we define trees as subsets of a symbolic space. Then, we turn to the
random setup and define Galton-Watson trees. We introduce some background
and preliminary results. Then, geometry comes into play and we present
the projection of trees to the Euclidean space. We introduce IFSs
and the special case of fractal percolation. In section \ref{sec:Hyperplane-absolute-game}
we define the hyperplane absolute game and describe some related results.
We then study the hyperplane diffuse property in the context of iterated
function systems and Galton-Watson fractals. In section \ref{sec:the-fractal-percolation}
we prove Theorem \ref{thm:Main theorem diffuse subses} for the special
case of fractal percolation, and in section \ref{sec:The-general-case},
after some required preparations, we prove the theorem in its general
form. Finally, in the appendix, we provide an analysis of the microsets
of Galton-Watson fractals (a notion which will be defined in Section
\ref{sec:Hyperplane-absolute-game}) and show that in many cases (fractal
percolation for example) Galton-Watson fractals are a.s. not hyperplane
diffuse.

\subsection{Acknowledgments}

This work is a part of the author's doctoral thesis written under
the supervision of Prof. Barak Weiss. The author was partially supported
by ISF grant 2095/15 and BSF grant 2016256.

\section{Galton-Watson processes\label{sec:Labeled-GW Trees}}

\subsection{Preliminaries - symbolic spaces and trees}

We shall now fix some notations regarding the symbolic spaces we are
about to use. Let $\mathbb{A}$ be some finite set considered as the
alphabet. Denote $\mathbb{A}^{*}=\left\{ \emptyset\right\} \cup\bigcup\limits _{n\in\mathbb{N}}\mathbb{A}^{n}$,
this is the set of all finite words in the alphabet $\mathbb{A}$,
with $\emptyset$ representing the word of length 0. Given a word
$i\in\mathbb{A}^{*}$, we use subscript indexing to denote the letters
comprising $i$, so that $i=i_{1}...i_{n}$ where $i_{k}\in\mathbb{A}$
for $k=1,...,n$. $\mathbb{A}^{*}$ is considered as a semigroup with
the concatenation operation $\left(i_{1}...i_{n}\right)\cdot\left(j_{1}...j_{m}\right)=\left(i_{1}...i_{n},j_{1}...j_{m}\right)$
and with $\emptyset$ the identity element. The dot notation will
usually be omitted so that the concatenation of two words $i,j\in\mathbb{A}^{*}$
will be denoted simply by $ij$. We will also consider the action
of $\mathbb{A}^{*}$ on $\mathbb{A}^{\mathbb{N}}$ by concatenations
denoted in the same way. We put a partial order on $\mathbb{A}^{*}\cup\mathbb{A}^{\mathbb{N}}$
by defining $\forall i\in\mathbb{A}^{*},\,\forall j\in\mathbb{A}^{*}\cup\mathbb{A}^{\mathbb{N}},\,i\leq j$
iff $\exists k\in\mathbb{A}^{*}\cup\mathbb{A}^{\mathbb{N}}$ with
$ik=j$, that is to say $i\leq j$ iff $i$ is a prefix of $j$. Given
any $i\in\mathbb{A}^{*}$ we shall denote the length of $i$ by $\left|i\right|=n$
where $n$ is the unique integer with the property $i\in\mathbb{A}^{n}$.
Given any $i\in\mathbb{A}^{*}$, the corresponding cylinder set in
$\mathbb{A}^{\mathbb{N}}$ is defined as $\left[i\right]=\left\{ j\in\mathbb{A}^{\mathbb{N}}:\,i<j\right\} $. 
\begin{defn}
A subset $T\subseteq\mathbb{A}^{*}$\emph{ }will be called a\emph{
tree} with alphabet $\mathbb{A}$ if $\emptyset\in T$, and for every
$a\in T$, $\forall b\in\mathbb{A}^{*},\,b\leq a\implies b\in T$.
We shall denote $T_{n}=T\cap\mathbb{A}^{n}$ for every $n\geq1$,
and $T_{0}=\left\{ \emptyset\right\} $ so that $T=\bigcup\limits _{n\geq0}T_{n}$.
We also denote for each $a\in T$, $W_{T}\left(a\right)=\left\{ i\in\mathbb{A}:\,ai\in T\right\} $.
The \emph{boundary} of a tree $T$ is denoted by $\partial T$ and
is given by 
\[
\partial T=\left\{ a\in\mathbb{A}^{\mathbb{N}}:\,\forall n\in\mathbb{N},\,a_{1}...a_{n}\in T_{n}\right\} .
\]
The set of all trees with alphabet $\mathbb{A}$ will be denoted by
$\mathscr{\ensuremath{T}}_{\mathbb{A}}\subset2^{\mathbb{A}^{*}}$.
A \emph{subtree} of $T\in\mathscr{T}_{\mathbb{A}}$ is any tree \emph{$T^{\prime}\in\mathbb{\mathscr{T}_{\mathbb{A}}}$
s.t. $T^{\prime}\subseteq T$}.
\end{defn}

We continue with a few more definitions which will come in handy in
what follows. Given a tree $T\in\mathscr{T}_{\mathbb{A}}$ and $a\in T$
some vertex of $T$, we denote $T^{a}=\left\{ j\in\mathbb{A}^{*}:\,aj\in T\right\} \in\mathscr{T}_{\mathbb{A}}$,
the\emph{ descendants tree of $a$}. The \emph{length} of a tree $T\subseteq\mathbb{A}^{*}$
is defined by $\text{length}\left(T\right)=\sup\left\{ n\in\mathbb{N}\setminus\left\{ 0\right\} :\,T_{n-1}\neq\emptyset\right\} $
and takes values in $\mathbb{N}\cup\left\{ \infty\right\} $. A basic
observation in this context is that $\forall T\in\mathscr{T}_{\mathbb{A}}$
with $\text{length}\left(T\right)=\infty$, $\partial T\neq\emptyset$.
\begin{defn}
A finite set $\Pi\subset\mathbb{A}^{*}$ is called a \emph{section}
if $\bigcup\limits _{i\in\Pi}\left[i\right]=\mathbb{A}^{\mathbb{N}}$
and the union is a disjoint union. Given a tree $T\in\mathscr{T}_{\mathbb{A}}$
and a section $\Pi\subset\mathbb{A}^{*}$ we denote $T_{\Pi}=T\cap\Pi$.

Sections will play an important role in the proof of Theorem \ref{thm:Main theorem diffuse subses}.
In the fractal percolation case, it is enough to consider the sections
$\mathbb{A}^{n}$ for $n\in\mathbb{N}$, therefore we postpone the
discussion on sections and their intersections with random trees to
subsection \ref{subsec:Sections}.%
\end{defn}

\subsection{The random setup - Galton-Watson processes}
\begin{defn}
Let $\mathbb{A}$ be some finite alphabet. Let $W$ be a random variable
with values in $2^{\mathbb{A}}$. Let $\left(W_{a}\right)_{a\in\mathbb{A}^{*}}$
be a (countable) collection of independent copies of $W$. We now
define inductively:
\end{defn}

\begin{itemize}
\item $T_{0}=\left\{ \emptyset\right\} $.
\item For $n\geq1,$ $T_{n}=\bigcup\limits _{a\in T_{n-1}}\left\{ aj:\,j\in W_{a}\right\} \subseteq\mathbb{A}^{n}$ 
\end{itemize}
If at some point $T_{n}=\emptyset$, then for every $l>n$, $T_{l}=\emptyset$
and we shall say that the process dies out or that \emph{extinction}
occurred. Finally we denote $T=\bigcup\limits _{n\geq0}T_{n}$. We
shall call the process $T_{0,}T_{1,}T_{2},...$ (and $T$ as well)
a\emph{ Galton-Watson process} with alphabet $\mathbb{A}$ and \emph{offspring
distribution} $W$. We shall consider $T$ as a tree and refer to
it as a\emph{ Galton-Watson Tree} \emph{(GWT)}. Note that $T$ is
a random variable determined by the random variables $\left(W_{a}\right)_{a\in\mathbb{A}^{*}}$.
As mentioned in the Introduction, we shall make the assumption that
$\forall i\in\mathbb{A},\,\mathbb{P}\left(i\in W\right)>0$ (otherwise
we may take a smaller alphabet without affecting the law of $T$).

By definition $\mathscr{T}_{\mathbb{A}}\subset2^{\mathbb{A}^{*}}$.
As sets, $2^{\mathbb{A}^{*}}\approx\prod\limits _{n=0}^{\infty}2^{\mathbb{A}^{n}}$
with the convention that $A^{0}=\left\{ \emptyset\right\} $, thus
$2^{\mathbb{A}^{*}}$ may be endowed with the product topology of
$\prod\limits _{n=0}^{\infty}2^{\mathbb{A}^{n}}$ which is metrizable,
separable and compact (where each $2^{\mathbb{A}^{n}}$ carries the
discrete topology). With this topology, $\mathscr{T}_{\mathbb{A}}$
is a closed subset of $2^{\mathbb{A}^{*}}$, and from this point forward
$\mathscr{T}_{\mathbb{A}}$ will carry the topology%
\begin{lyxgreyedout}
The definition of the topology is only required for the appendix.
Should it be moved?%
\end{lyxgreyedout}
{} induced by $2^{\mathbb{A}^{*}}$.

Given a finite tree $L\subset\mathbb{A}^{*}$, let $\left[L\right]\subset\mathscr{\ensuremath{T}}_{\mathbb{A}}$
be defined by 
\[
\left[L\right]=\left\{ S\in\mathscr{\ensuremath{T}}_{\mathbb{A}}:\,\forall n\in\mathbb{N},\,L_{n}\neq\emptyset\implies S_{n}=L_{n}\right\} 
\]
These sets form a basis for the topology of $\mathscr{T}_{\mathbb{A}}$
and generate the Borel $\sigma$-algebra on $\mathscr{T}_{\mathbb{A}}$
which we denote by $\mathscr{B}$. By Kolmogorov's extension theorem
the Galton-Watson process yields a unique Borel measure on $\mathscr{T}_{\mathbb{A}}$
which we denote by $\mathscr{GW}$ and is the distribution of the
random variable $T$. The careful reader will notice that all the
events in this paper whose probability is analyzed are in $\mathscr{B}$.
For any measurable property $\mathscr{T}^{\prime}\subseteq\mathscr{T}_{\mathbb{A}}$,
the notation $\mathbb{P}\left(T\in\mathscr{T}^{\prime}\right)$ means
$\mathscr{GW}\left(\mathscr{T}^{\prime}\right)$.%

For each $n\geq0$, we denote $Z_{n}=\left|T_{n}\right|$. We note
that the usual definition of a Galton-Watson process (as defined e.g.
in \cite{Lyons2016}) would be the random process $\left(Z_{n}\right)_{n\geq1}$,
but in our case it is important to keep track of the labels as later
on we are going to project these trees to the Euclidean space (in
the beginning of subsection \ref{subsec:IFSs-and-projections}). Nevertheless,
in some cases where the labels aren't important we shall refer to
the process $\left(Z_{n}\right)_{n\geq1}$ as a Galton-Watson process
as well.

Given a Galton-Watson process, we shall denote $m=\mathbb{E}\left(Z_{1}\right)$.
It is a basic fact that for every $n\geq1$, $\mathbb{E}\left(Z_{n}\right)=m^{n}$.
As mentioned in section \ref{sec:Introduction}, the process is called
supercritical when $m>1$, in which case $\mathbb{P}\left(\text{nonextinction}\right)>0$. 

The following is a basic result in the theory of Galton-Watson processes
(\cite{Kesten1966}, see also \cite{Lyons2016}).
\begin{thm}[Kesten-Stigum]
\label{thm:Kesten - Stigum} Let $\left(Z_{k}\right)_{k=1}^{\infty}$
be a supercritical Galton-Watson process, then $\dfrac{Z_{k}}{m^{k}}$
converges a.s. (as $k\to\infty$) to a random variable $L$, where
$\mathbb{E}\left(L\right)=1$ and $L>0$ a.s. conditioned on nonextinction. 
\end{thm}

An immediate corollary of Theorem \ref{thm:Kesten - Stigum} is the
following.
\begin{cor}
\label{cor:kesten - stigum}Let $\left(Z_{k}\right)_{k=1}^{\infty}$
be a supercritical Galton-Watson process, then the following holds.
\[
\forall\varepsilon>0,\,\exists c>0,\,\exists K_{0}\in\mathbb{N},\,\forall k>K_{0},\,\mathbb{P}\left(\dfrac{Z_{k}}{m^{k}}>c\mid\,\text{nonextinction}\right)>1-\varepsilon.
\]
\end{cor}

The proof of Corollary \ref{cor:kesten - stigum} is standard and
is left as an exercise to reader.

The following proposition is a result of the statistical self similarity
of the Galton-Watson processes, that is, the fact that for every $v\in T$,
the tree $T^{v}$ is itself a GWT with the same offspring distribution
as $T$, and that for $v,w\in T$ which are not descendants of each
other, $T^{v}$ and $T^{w}$ are independent.
\begin{prop}
\label{prop:0-1 law}Let $T$ be a supercritical Galton-Watson tree
with alphabet $\mathbb{A}$ and let $\mathscr{T}^{\prime}\subseteq\mathscr{T}_{\mathbb{A}}$
be a measurable subset. Suppose that $\mathbb{P}\left(T\in\mathscr{T}^{\prime}\right)>0$,
then a.s. conditioned on nonextinction, there exist infinitely many
$v\in T$ s.t. $T^{v}\in\mathscr{T}^{\prime}$.
\end{prop}

\begin{proof}
By Corollary \ref{cor:kesten - stigum}, given some $\varepsilon>0$,
there exists a constant $c>0$ s.t. 
\[
\mathbb{P}\left(\left|T_{k}\right|>cm^{k}\mid\,\text{nonextinction}\right)>1-\varepsilon
\]
whenever $k$ is large enough. Denote $\rho=\mathbb{P}\left(T\in\mathscr{T}^{\prime}\right)$.
Given any $M>0$, 

\[
\begin{array}{l}
\mathbb{P}\left(\left|\left\{ v\in T_{k}:\,T^{v}\in\mathscr{T}^{\prime}\right\} \right|<M\mid\,\text{nonextinction},\,\left|T_{k}\right|>cm^{k}\right)\leq\\
{\displaystyle \mathbb{P}\left(\text{Bin}\left(cm^{k},\,\rho\right)<M\right)\leq}\\
\dfrac{\rho\left(1-\rho\right)cm^{k}}{\left[\rho cm^{k}-M\right]^{2}}\underset{k\rightarrow\infty}{\longrightarrow}0
\end{array}
\]
In the last inequality we used Chebyshev's inequality%
{} assuming that $k$ is large enough so that $\rho cm^{k}>M$.%
\end{proof}

\subsection{A generalization of Pakes-Dekking theorem}

In this subsection we show how to prove the existence of certain subtrees
in GWTs.
\begin{defn}
\label{def:A-tree}Given a nonempty collection of nonempty subsets
of the alphabet $\mathscr{A}\subseteq2^{\mathbb{A}}$, a tree $T\in\mathscr{T}_{\mathbb{A}}$
is called an \emph{$\mathscr{A}$-tree }if for every $a\in T$, $W_{T}\left(a\right)\in\mathscr{A}$.
Note that by definition $\mathscr{A}$-trees are infinite. A finite
tree $T$ will be called an \emph{$\mathscr{A}$-tree of length $n$},
if $\text{length}\left(T\right)=n$ and $\forall a\in T\setminus T_{\text{length}\left(T\right)-1}$,
$W_{T}\left(a\right)\in\mathscr{A}$. Given a tree $S\in\mathscr{T}_{\mathbb{A}}$,
an \emph{$\mathscr{A}$-subtree }(respectively, \emph{$\mathscr{A}$-subtree
of length $n$})\emph{ }of $S$ is a subtree of $S$ which is an \emph{$\mathscr{A}$}-tree\emph{
}(respectively, $\mathscr{A}$-tree of length $n$).
\end{defn}

Given a collection $\mathscr{A}\subseteq2^{\mathbb{A}}$ as above,
we denote $\overline{\mathscr{A}}=\left\{ S\in2^{\mathbb{A}}:\,\exists X\in\mathscr{A},\,X\subseteq S\right\} $.
It is obvious that $T$ has an $\overline{\mathscr{A}}$-subtree iff
$T$ has an $\mathscr{A}$-subtree. $\mathscr{A}$ will be called
\emph{monotonic }if $\overline{\mathscr{A}}=\mathscr{A}$.
\begin{lem}
\label{lem:A-subtree of every length implies infinite A-subtree}Let
$T\in\mathscr{T}_{\mathbb{A}}$ be a tree with alphabet $\mathbb{A}$,
and let $\mathscr{A}\subseteq2^{\mathbb{A}}$ be some collection as
above. Then $T$ has an infinite $\mathscr{A}$-subtree $\iff$ $\forall n\in\mathbb{N}$,
$T$ has an $\mathscr{A}$-subtree of length $n$.
\end{lem}

\begin{proof}
For every $n\in\mathbb{N}$, let $S^{\left(n\right)}$ be an $\mathscr{A}$-subtree
of $T$ of length $n$. Denote $S=\bigcup\limits _{n=1}^{\infty}S^{\left(n\right)}$.
$S$ is a subtree of $T$, and since $\text{length}\left(S\right)=\infty$,
$\partial S\neq\emptyset$%
. Define $S^{\prime}=\left\{ v<w:\,w\in\partial S\right\} $. Then
$S^{\prime}$ is an $\overline{\mathscr{A}}$-subtree of $T$%
, and therefore contains an $\mathscr{A}$-subtree. The other direction
is trivial. 
\end{proof}
The following theorem is the main tool we use to show that certain
$\mathscr{A}$-subtrees exist in GWTs. This theorem is a generalization
of a theorem by Pakes and Dekking (\cite{Pakes1991353}, see also
\cite{Lyons2016}) which deals with the existence of \emph{a}-ary
subtrees (where each element has exactly \emph{a} children) in Galton-Watson
trees. First, we need the following definition.
\begin{defn}
Let $A$ be some finite set, and $p\in\left[0,1\right]$. A random
subset $Y\subseteq A$ is said to have a binomial distribution with
parameter $p$ if 
\[
\forall B\subseteq A,\,\mathbb{P}\left(Y=B\right)=p^{\left|B\right|}\left(1-p\right)^{\left|A\setminus B\right|}.
\]
In this case we denote $Y\sim\text{Bin}\left(A,\,p\right)$. Note
that the notation $\text{Bin}\left(\cdot,\cdot\right)$ will also
be used for the usual binomial distribution as well, where the first
argument will be an integer and not a set. 
\end{defn}

The following notation will recur throughout the paper. Let $A$ be
some fixed finite set, and let $X$ be a random subset of $A$. Given
$s\in\left[0,1\right]$, we denote $X^{\left(s\right)}=X\cap Y$ where
$Y\sim\text{Bin}\left(A,\,1-s\right)$. 

Let $T$ be a GWT with alphabet $\mathbb{A}$ and any offspring distribution
$W$. Let $\mathscr{A}\subseteq2^{\mathbb{A}}$ be some nonempty collection
of nonempty subsets of $\mathbb{A}$. Define the function $g_{\mathscr{A}}:\left[0,1\right]\rightarrow\left[0,1\right]$
by $g_{\mathscr{A}}\left(s\right)=\mathbb{P}\left(W^{\left(s\right)}\notin\overline{\mathscr{A}}\right)$.
Finally, denote $\tau\left(\mathscr{A}\right)=\mathbb{P}\left(\text{\ensuremath{T} has an \ensuremath{\mathscr{A}}-subtree}\right)$.

\begin{thm}
\label{thm:fixed point}With notations as above, $1-\tau\left(\mathscr{A}\right)$
is the smallest fixed point of $g_{\mathscr{A}}$ in $\left[0,1\right]$. 
\end{thm}

\begin{proof}
We follow the Proof given in \cite[chapter 5]{Lyons2016} almost verbatim.
Note that the following properties hold:
\begin{enumerate}
\item $g_{\mathscr{A}}$ is continuous and monotonically increasing. %
{} %
\item $g_{\mathscr{A}}\left(1\right)=1$
\item $g_{\mathscr{A}}\left(0\right)=\mathbb{P}\left(W\notin\overline{\mathscr{A}}\right)$
\end{enumerate}
If $g_{\mathscr{A}}\left(0\right)=0$, then $\mathbb{P}\left(W\in\overline{\mathscr{A}}\right)=1$
which implies the existence of an $\mathscr{A}$ - subtree a.s., i.e.,
$1-\tau\left(\mathscr{A}\right)=0$ and the claim follows. Otherwise,
we assume that $g_{\mathscr{A}}\left(0\right)>0$. Let $q_{n}$ be
the probability that $T$ does not contain an $\mathscr{A}-\text{subtree}$
of length $n$, where $q_{0}=0$. Then $1-q_{n}\searrow\mathbb{P}\left(T\text{ has an \ensuremath{\mathscr{A}}-subtree of every length}\right)$
and by Lemma \ref{lem:A-subtree of every length implies infinite A-subtree}
this is equivalent to $q_{n}\nearrow1-\tau\left(\mathscr{A}\right)$. 
\begin{claim*}
For every $n\geq1$, $q_{n}=g_{\mathscr{A}}\left(q_{n-1}\right)$.
\end{claim*}
\begin{proof}[proof of claim]
 Denote for every $n\geq1$, the following random set:
\[
V_{n}=\left\{ v\in T_{1}:\,T^{v}\,\text{has an \ensuremath{\mathscr{A}}-subtree of length \ensuremath{n}}\right\} .
\]
For each element $v\in\mathbb{A}$, $\mathbb{P}\left(v\in V_{n}|\,v\in T_{1}\right)=1-q_{n}$,
and for every two distinct elements in $\mathbb{A}$ these events
are independent, so $V_{n}\sim W^{\left(q_{n}\right)}$. Now, since
$T$ has an $\mathscr{A}$-subtree of length $n$ iff $V_{n-1}\in\overline{\mathscr{A}}$,
\[
q_{n}=\mathbb{P}\left(V_{n-1}\notin\overline{\mathscr{A}}\right)=\mathbb{P}\left(W^{\left(q_{n-1}\right)}\notin\overline{\mathscr{A}}\right)=g_{\mathscr{A}}\left(q_{n-1}\right).
\]
\end{proof}
Since $g_{\mathscr{A}}$ is increasing and continuous, its smallest
fixed point is $\lim\limits _{n\to\infty}g_{\mathscr{A}}^{n}\left(0\right)$,
where $g_{\mathscr{A}}^{n}$ denotes the composition of $g_{\mathscr{A}}$
with itself $n$ times (this is a general property of increasing and
continuous functions on $\left[0,1\right]$ whose proof is easy and
left to the reader). By the claim above, $\lim\limits _{n\to\infty}g_{\mathscr{A}}^{n}\left(0\right)=\lim\limits _{n\to\infty}q_{n}$
which concludes the proof. %
\end{proof}
\begin{rem}
Given some integer $1\leq a\leq\left|\mathbb{A}\right|$, we may set
$\mathscr{A}=\left\{ S\subseteq\mathbb{A}:\,\left|S\right|=a\right\} $.
In this case $\mathscr{A}$-subtrees are actually \emph{a}-ary subtrees
and Theorem \ref{thm:fixed point} becomes exactly Pakes - Dekking
theorem. 
\end{rem}

Combining Theorem \ref{thm:fixed point} and Proposition \ref{prop:0-1 law}
we obtain the following.
\begin{cor}
\label{cor:Pakes-Dekking result} With notations as above, if $g_{\mathscr{A}}$
has some fixed point $<1$, then almost surely conditioned on nonextinction,
there exist infinitely many $v\in T$ s.t. $T^{v}$ contains an $\mathscr{A}$-subtree.
\end{cor}

\subsection{IFSs and projections to $\mathbb{R}^{d}$\label{subsec:IFSs-and-projections}}

An \emph{iterated function system (IFS)} is a finite collection $\left\{ \varphi_{i}\right\} _{i\in\Lambda}$
of self maps of $\mathbb{R}^{d}$ which are Lipschitz continuous with
Lipschitz constants smaller than 1.\emph{ }It is one of the most basic
results in fractal theory (due to Hutchinson \cite{Hutchinson1981})
that every IFS $\left\{ \varphi_{i}\right\} _{i\in\Lambda}$ gives
rise to a unique nonempty compact set $K\subset\mathbb{R}^{d}$ which
satisfies the equation $K=\bigcup\limits _{i\in\Lambda}\varphi_{i}K$.
The set $K$ is called the \emph{attractor} of the IFS. 

A map $f:\mathbb{R}^{d}\rightarrow\mathbb{R}^{d}$ is called a \emph{contracting
similarity} if there exists a constant $r\in\left(0,1\right)$, referred
to as the \emph{contraction ratio} of $f$, s.t. $\forall x,y\in\mathbb{R}^{d},\,\left\Vert f\left(x\right)-f\left(y\right)\right\Vert =r\left\Vert x-y\right\Vert $,
so that $f$ is a composition of a scaling by factor $r$, an orthogonal
transformation and a translation. In this paper we shall only discuss
IFSs which are formed by contracting similarity maps. Such IFSs shall
be referred to as \emph{similarity IFSs.} 

When analyzing a similarity IFS $\Phi=\left\{ \varphi_{i}\right\} _{i\in\Lambda}$
it is natural to work in the symbolic spaces $\Lambda^{\mathbb{N}}$
and $\Lambda^{*}$. In view of the setup above, in the abstract setting
of trees with alphabet $\mathbb{A}$, we shall often assign weights
to the alphabet $\left\{ r_{i}\right\} _{i\in\mathbb{A}}$. These
weights will correspond to the contraction ratios of similarity maps
and therefore we shall always assume that $r_{i}\in\left(0,1\right)$
for every $i\in\mathbb{A}$. 

Given an IFS $\left\{ \varphi_{i}\right\} _{i\in\Lambda}$, the identification
between the symbolic spaces $\Lambda^{\mathbb{N}},\,\Lambda^{*}$
and the Euclidean space is made via the coding map $\gamma_{\Phi}:\Lambda^{\mathbb{N}}\rightarrow\mathbb{R}^{d}$
which is given by 
\begin{equation}
\gamma_{\Phi}\left(j\right)=\bigcap\limits _{n=1}^{\infty}\varphi_{j_{1}...j_{n}}\left(K\right)\label{eq:coding map}
\end{equation}
where $\varphi_{j_{1}...j_{n}}=\varphi_{j_{1}}\circ...\circ\varphi_{j_{n}}$.
It may be easily seen that $K=\gamma_{\Phi}\left(\Lambda^{\mathbb{N}}\right)$.
Moreover, given a tree $T\in\Lambda^{*}$, we may project the boundary
of $T$ to the Euclidean space using $\gamma_{\Phi}$, where 
\begin{equation}
\gamma_{\Phi}\left(\partial T\right)=\bigcup_{j\in\partial T}\bigcap\limits _{n=1}^{\infty}\varphi_{j_{1}...j_{n}}\left(K\right)=\bigcap\limits _{n=1}^{\infty}\bigcup\limits _{i\in T_{n}}\varphi_{i}K\label{eq:projection of tree}
\end{equation}

Note that for every compact set $F\subset\mathbb{R}^{d}$ s.t. $\forall i\in\Lambda,\,\varphi_{i}F\subseteq F$,
we have $K=\bigcap\limits _{n=1}^{\infty}\bigcup\limits _{i\in\Lambda^{n}}\varphi_{i}F$,
hence $K\subseteq F$ and the decreasing sequence of sets $\left(\bigcup\limits _{i\in\Lambda^{n}}\varphi_{i}F\right)_{n=1}^{\infty}$
may be thought of as approximating $K$. Since $K\subseteq F$, we
may replace $K$ with $F$ in equations (\ref{eq:coding map}), (\ref{eq:projection of tree})
and the equations will remain true.

An IFS $\left\{ \varphi_{i}\right\} _{i\in\Lambda}$ satisfies the
\emph{open set condition (OSC) }if there exists some nonempty open
set $U\subset\mathbb{R}^{d}$ s.t. $\varphi_{i}U\subseteq U$ for
every $i\in\Lambda$, and $\varphi_{i}U\cap\varphi_{j}U=\emptyset$
for distinct $i,\,j\in\Lambda$. A set \emph{U} satisfying these conditions
will be called an \emph{OSC set for $\Phi$}. In case an IFS $\Phi=\left\{ \varphi_{i}\right\} _{i\in\Lambda}$
with contraction ratios $\left\{ r_{i}\right\} _{i\in\Lambda}$ satisfies
the open set condition, it is well known\footnote{This was first proved by Moran for self-similar sets without overlaps
in 1946 (see \cite{moran_1946}). The form stated here assuming the
OSC was first proved by Hutchinson in 1981 (see \cite{Hutchinson1981}).} that the Hausdorff dimension of the attractor of $\Phi$ is the unique
number $\delta$ which satisfies the equation $\sum\limits _{i\in\Lambda}r_{i}^{\delta}=1$
. For convenience, $\forall i=i_{1}...i_{n}\in\Lambda^{*}$ we denote
$r_{i}=r_{i_{1}}\cdots r_{i_{n}}$ which is the contraction ratio
of the map $\varphi_{i}$. We also denote $r_{min}=\min\left\{ r_{i}:\,i\in\Lambda\right\} $
and $r_{max}=\max\left\{ r_{i}:\,i\in\Lambda\right\} $.

We shall now turn to the probabilistic setup. 
\begin{defn}
Let $\Phi=\left\{ \varphi_{i}\right\} _{i\in\Lambda}$ be a similarity
IFS, and let $W$ be some random variable with values in $2^{\Lambda}$.
Let $T$ be a GWT with alphabet $\Lambda$ and offspring distribution
$W$, and finally let $E$ be the random set $E=\gamma_{\Phi}\left(\partial T\right)$.
The random set $E$ will be called a \emph{Galton-Watson fractal (GWF)}
w.r.t. the IFS $\Phi$ and offspring distribution $W$. We shall always
assume that $\mathbb{E}\left(\left|W\right|\right)>1$ so that the
Galton-Watson process is supercritical.
\end{defn}

The following theorem is due to Falconer \cite{FalconerK.J1986Rf}
and Mauldin and Williams \cite{MauldinR.Daniel1986RRCA}. See also
\cite[Theorem 15.10]{Lyons2016} for another elegant proof.
\begin{thm}
\label{thm:dimension of GW fractals}Let E be a Galton-Watson fractal
w.r.t. a similarity IFS $\Phi=\left\{ \varphi_{i}\right\} _{i\in\Lambda}$
satisfying the OSC, with contraction ratios $\left\{ r_{i}\right\} _{i\in\Lambda}$
and offspring distribution W. Then a.s. conditioned on nonextinction,
$\dim_{H}E=\delta$ where $\delta$ is the unique number satisfying
\[
\mathbb{E}\left(\sum_{i\in W}r_{i}^{\delta}\right)=1.
\]
\end{thm}

\subsection{Fractal percolation}

Fractal percolation is an important special case of GWFs. In general,
it is easier to analyze since it has more independence, the maps of
the corresponding IFS have a trivial orthogonal component (i.e. they
are only composed of scaling and translation transformations), and
all the contraction ratios are equal. In our case, it is mainly the
last property which makes things less complicated.

To describe fractal percolation in the framework defined above, fix
an integer $b\geq2$ and $p\in\left(0,1\right)$. Denote $\Lambda_{b}=\left\{ 0,\,1,\,...,\,b-1\right\} ^{d}$
(since in this section $b$ is fixed we will not carry the subscript
$b$ and just write $\Lambda$ instead of $\Lambda_{b}$). Construct
a GWT with alphabet $\Lambda$ and binomial offspring distribution,
i.e., $W\sim\text{Bin}\left(\Lambda,\,p\right)$. 

The coding map is the $b$-adic coding map $\gamma_{b}:\Lambda^{\mathbb{N}}\rightarrow\left[0,1\right]^{d}$
given by 

\[
\gamma_{b}\left(i\right)=\sum\limits _{k=1}^{\infty}b^{-k}i_{k},\,\,\,\forall i\in\Lambda^{\mathbb{N}}.
\]
Note that each $i_{k}$ in the above formula is a vector in $\mathbb{R}^{d}$. 

This coding map corresponds to the similarity IFS consisted of the
homotheties mapping the unit cube to each $b$-adic cube, i.e., $\varphi_{i}\left(x\right)=\frac{1}{b}\cdot x+\frac{i}{b}$
for every $i\in\Lambda,\,x\in\mathbb{R}^{d}$. Note that there is
a correspondence between elements of $\Lambda^{n}$ and closed b-adic
subcubes of the $n^{th}$ level, where for each element $j\in\Lambda^{n}$
the corresponding closed b-adic cube is given by $\gamma\left(\left[j\right]\right)$.

Set for each $n$, $E_{n}=\bigcup\limits _{j\in T_{n}}\gamma\left(\left[j\right]\right)$,
then the limit set $E$ is given by
\[
E=\bigcap_{n\in\mathbb{N}}E_{n}=\gamma_{b}\left(\partial T\right).
\]

Note that in the supercritical case, by Theorem \ref{thm:dimension of GW fractals},
a.s. conditioned on nonextinction, \emph{
\[
\dim_{H}E=\log_{b}\left(m\right).
\]
}Recall that $m=\mathbb{E}\left(Z_{1}\right)$, and in case of fractal
percolation in $\mathbb{R}^{d}$ with parameters $b,p$, $m=pb^{d}$.

\section{Hyperplane diffuse sets\label{sec:Hyperplane-absolute-game}}

\subsection{The hyperplane absolute game}

The \emph{hyperplane absolute game, }developed in \cite{Broderick2012319},
is a useful variant of Schmidt's game which was invented by W. Schmidt
in \cite{Schmidt1966} and became a main tool for the study of $\text{BA}_{d}$. 

The hyperplane absolute game is played between two players, Bob and
Alice and has one fixed parameter $\beta\in\left(0,\frac{1}{3}\right)$.
Bob starts by defining a closed ball $B_{1}=B_{\rho_{1}}\left(x_{1}\right)\subset\mathbb{R}^{d}$.
Then, for every $i\in\mathbb{N}$, after Bob has chosen a ball $B_{i}=B_{\rho_{i}}\left(x_{i}\right)$,
Alice chooses an affine hyperplane $\mathcal{L}_{i}\subset\mathbb{R}^{d}$
and an $\varepsilon_{i}\in\left(0,\beta\rho_{i}\right)$, and removes
the $\varepsilon_{i}$-neighborhood of $\mathcal{L}_{i}$ denoted
by $A_{i}=\mathcal{L}_{i}^{\left(\varepsilon_{i}\right)}$ from $B_{i}$.
Then Bob chooses his next ball $B_{i+1}=B_{\rho_{i+1}}\left(x_{i+1}\right)\subset B_{i}\setminus A_{i}$
with the restriction on the radius $\rho_{i+1}\geq\beta\rho_{i}$.
The game continues ad infinitum. A set $S\subset\mathbb{R}^{d}$ is
called\textit{ hyperplane absolute winning} (HAW) if for every $\beta\in\left(0,\frac{1}{3}\right)$,
Alice has a strategy guaranteeing that $\bigcap\limits _{n=1}^{\infty}B_{n}$
intersects $S$. Note that existence of such a strategy for some $\beta\in\left(0,\frac{1}{3}\right)$,
implies the existence of a strategy for every $\beta^{\prime}\in\left(\beta,\frac{1}{3}\right)$. 

Many interesting sets are known to be HAW (see e.g. \cite{AnJinpeng2015BOoD,FishmanL.2018Idao,Nesharim2014145}),
including the set $\text{BA}_{d}$ \cite[Theorem 2.5]{Broderick2012319}.
Note that HAW sets in $\mathbb{R}^{d}$ are always dense and have
Hausdorff dimension $d$. Also, as stated in Theorem \ref{thm:HAW is robust},
the HAW property is preserved under countable intersections and $C^{1}$
diffeomorphisms, which make these sets ``large''. The following
observation may be found useful.
\begin{rem}
\label{rem:G inv. HAW sets}Let $S\subseteq\mathbb{R}^{d}$ be HAW.
Then for every countable group $G$ of $C^{1}$ diffeomorphisms of
$\mathbb{R}^{d}$, $S$ contains a $G$ - invariant set $\tilde{S}$
which is also HAW. Indeed, we may take $\tilde{S}=\bigcap\limits _{g\in G}gS$,
and by Theorem \ref{thm:HAW is robust}, $\tilde{S}$ is itself HAW.
This property may be useful in some cases, and as an example we already
saw a use for this property in Theorem \ref{thm:fractal percolation intersects BA - Haweks}.
\end{rem}

Although HAW sets are ``large'' in the senses mentioned above, as
in the case of $\text{BA}_{d}$, HAW sets may have Lebesgue measure
0.

\emph{}%

A key feature of HAW sets is given in Theorem \ref{thm:HAW intersects diffuse sets},
which states, generally speaking, that HAW sets intersect hyperplane
diffuse sets. While the definition of the hyperplane diffuse property
given in the Introduction (Definition \ref{def:diffuse}) may seem
a bit technical and maybe tailored for the hyperplane absolute game,
an equivalent definition of this property using the notion of \emph{microsets}
(which was coined by H. Furstenberg in \cite{Furstenberg2008405})
indicates that it is actually quite natural. 
\begin{defn}
Given a compact set $C\subset\mathbb{R}^{d}$, we denote 
\[
\Omega_{C}=\left\{ A\subseteq C:\,A\text{ is compact and nonempty}\right\} .
\]
We equip $\Omega_{C}$ with the Hausdorff metric which we denote by
$d_{H}$, and is given by 
\[
\forall A,B\in\Omega_{C},\,d_{H}\left(A,B\right)=\inf\left\{ \varepsilon>0:\,A\subseteq B^{\left(\varepsilon\right)}\,\wedge\,B\subseteq A^{\left(\varepsilon\right)}\right\} .
\]
\end{defn}

It is a well known fact that as a metric space $\left(\Omega_{C},d_{H}\right)$
is compact. 

Given a (closed or open) ball with radius $r>0$ and center point
$x\in\mathbb{R}^{d}$, we define $F_{B}:\mathbb{R}^{d}\rightarrow\mathbb{R}^{d}$
by $F_{B}\left(t\right)=\frac{1}{r}\left(t-x\right)$, so that $F_{B}$
is the unique homothety mapping $B$ to the unit ball.
\begin{defn}
Let $K\subseteq\mathbb{R}^{d}$ be a compact set. A set of the form
$F_{B}\left(K\cap B\right)$ where $B$ is a closed ball centered
in $K$ is called a \emph{miniset} of $K$. Every limit of minisets
of $K$ in the Hausdorff metric is called a \emph{microset}\footnote{We follow the definition given in \cite{Broderick2012319} which is
slightly different than the original one given by Furstenberg in \cite{Furstenberg2008405}.
For starters, in Furstenberg's definition cubes are used instead of
balls, but the most significant difference is that there is no restriction
on their center points, which in many cases enables a closed set $K$
to have minisets which are contained in lower dimensional affine spaces,
even when $K$ is hyperplane diffuse, hence making Proposition \ref{prop:diffuse iff no microsets on hyperplanes}
false.

}\emph{.}
\end{defn}

\begin{prop}
\label{prop:diffuse iff no microsets on hyperplanes}\cite[Lemma 4.4]{Broderick2012319}
A compact set $K\subset\mathbb{R}^{d}$ is hyperplane diffuse iff
no microset of $K$ is contained in an affine hyperplane.
\end{prop}

The following theorem was proved in \cite{Kleinbock2005}.
\begin{thm}
\cite[Theorem 2.3]{Kleinbock2005} Let $\Phi$ be a similarity IFS
satisfying the OSC whose attractor $K$ is not contained in an affine
hyperplane, then $K$ is hyperplane diffuse and Ahlfors-regular.
\end{thm}

\begin{rem}
\label{rem:irreducible <=00003D> not on a hyperplane }Note that instead
of the condition that $K$ is not contained in a single hyperplane,
the original condition in \cite[Theorem 2.3]{Kleinbock2005} is that
no finite collection of affine hyperplanes is preserved by $\Phi$
(such an IFS is referred to as \emph{irreducible}), but it turns out
that these two conditions are in fact equivalent (regardless of the
OSC). This fact is proved in \cite[Proposition 3.1]{BRODERICK20132186}.
\end{rem}

Ahlfors-regularity is important if one wants to get a full Hausdorff
dimension of the intersection of $K$ with HAW sets. But in order
to get a positive lower bound for the Hausdorff dimension of intersections
of $K$ with HAW sets (which depends only on $K$) it is enough for
$K$ to be hyperplane diffuse. In this case the OSC may be dropped
(see Theorem \ref{thm:When attractors are hyperplane diffuse}).

\subsection{Diffuseness in IFSs\label{subsec:Diffuseness-revisited}}
\begin{defn}
\label{def:diffuseness for IFS}Let $\Phi=\left\{ \varphi_{i}\right\} _{i\in\Lambda}$
be a similarity IFS, and let $F\subset\mathbb{R}^{d}$ be a non-empty
compact set s.t. $\forall i\in\Lambda,\,\varphi_{i}F\subseteq F$.
For any $c>0$, we say that $\Phi$ is $\left(F,c\right)$-diffuse
if $\forall\mathcal{L}\subseteq\mathbb{R}^{d}$ affine hyperplane
$\exists i\in\Lambda$ s.t. $\varphi_{i}F\cap\mathcal{L}^{\left(c\right)}=\emptyset$.
Moreover, we say that $\Phi$ is $F$-diffuse if it is $\left(F,c\right)$-diffuse
for some $c>0$.
\end{defn}

Obviously, the attractor of $\Phi$ is a natural candidate for $F$
in the definition above. 
\begin{lem}
\label{lem:diffuse implies (K,c)-diffuse}Let $\Phi=\left\{ \varphi_{i}\right\} _{i\in\Lambda}$
be a similarity IFS whose attractor is denoted by $K$. Let $F\subset\mathbb{R}^{d}$
be as above, and $c>0$.
\end{lem}

\begin{enumerate}
\item If $\Phi$ is $\left(F,c\right)$-diffuse, then $\Phi$ is also $\left(K,c\right)$-diffuse. 
\item If $\Phi$ is $\left(K,c\right)$-diffuse, then $\forall c^{\prime}\in\left(0,c\right)$,
for a large enough $n$, the IFS $\left\{ \varphi_{i}\right\} _{i\in\Lambda^{n}}$
is $\left(F,c^{\prime}\right)$-diffuse.
\end{enumerate}
\begin{proof}
(1) is trivial since $K\subseteq F$. (2) follows directly from the
fact that the decreasing sequence of sets $\bigcup\limits _{i\in\Lambda^{n}}\varphi_{i}F$
converges to $K$ as $n\to\infty$.%
\end{proof}
In view of the above, we say that $\Phi$ is $c$-diffuse if it is
$\left(K,c\right)$-diffuse where $K$ is the attractor of $\Phi$,
and we say that $\Phi$ is diffuse if it is $c$-diffuse for some
$c>0$. 

Given an IFS $\left\{ \varphi_{i}\right\} _{i\in\Lambda}$ in the
background, we shall call a finite subset $A\subset\Lambda^{*}$ diffuse
(resp. $c$-diffuse and $\left(F,c\right)$-diffuse) if the IFS $\left\{ \varphi_{i}\right\} _{i\in A}$
is diffuse (resp. $c$-diffuse and $\left(F,c\right)$-diffuse). Moreover,
Given a tree $T\in\mathscr{T}_{\Lambda}$, we say that $T$ is diffuse
(resp. $c$-diffuse and $\left(F,c\right)$-diffuse) if for each $i\in T$,
$W_{T}\left(i\right)$ is diffuse (resp. $c$-diffuse and $\left(F,c\right)$-diffuse).
Note that a tree $T\in\mathscr{T}_{\Lambda}$ is diffuse iff it is
$c$-diffuse for some $c>0$. 

The following lemma will be useful in what follows.
\begin{lem}
\label{lem:sets contained in hyperplanes iff 0 width} $\forall A\subseteq\mathbb{R}^{d}$,
A is contained in an affine hyperplane $\Longleftrightarrow$ 
\[
\inf\left\{ \varepsilon>0:\,\exists\mathcal{L}\subset\mathbb{R}^{d}\text{ affine hyperplane s.t. }A\subseteq\text{\ensuremath{\mathcal{L}}}^{\left(\varepsilon\right)}\right\} =0.
\]
\end{lem}

\begin{proof}
The implication $\left(\implies\right)$ is trivial. For the other
direction, assume that $A$ is not contained in an affine hyperplane.
Then there exist $x_{1},...,x_{d+1}\in A$ which are not contained
in a single affine hyperplane, that is to say that the vectors $v_{1}=x_{2}-x_{1},...,v_{d}=x_{d+1}-x_{1}$
are linearly independent, so the matrix $M=\left(\begin{array}{c}
-\,v_{1}\,-\\
...\\
-\,v_{d}\,-
\end{array}\right)$ is nonsingular. Since $\det(\cdot)$ is a continuous function, small
perturbations of $M$ are still nonsingular, so for $\varepsilon>0$
small enough, no affine hyperplane intersects all the balls $B_{\varepsilon}\left(x_{i}\right)$
for $i=1,...,d+1$, and there is no affine hyperplane $\mathcal{L}\subset\mathbb{R}^{d}$,
s.t. $A\subseteq\text{\ensuremath{\mathcal{L}}}^{\left(\varepsilon\right)}$.
\end{proof}

\begin{prop}
\label{prop:diffuse implies 0-diffuse}Let $\Phi=$$\left\{ \varphi_{i}\right\} _{i\in\Lambda}$
be a similarity IFS and let $F\subset\mathbb{R}^{d}$ be a nonempty
compact set s.t. $\varphi_{i}F\subseteq F$ for every $i\in\Lambda$.
Then $\Phi$ is $\left(F,c\right)$-diffuse for some $c>0$ $\iff$
for every affine hyperplane $\mathcal{L}\subseteq\mathbb{R}^{d}$,
there exists some $i\in\Lambda$ s.t. $\varphi_{i}F\cap\mathcal{L}=\emptyset$.
\end{prop}

\begin{proof}
The implication $\left(\implies\right)$ is true by definition. For
the other direction, assume that $\Phi$ is not diffuse, i.e., $\Phi$
is not $\left(F,c\right)$-diffuse for any $c$. Take any decreasing
sequence $\left(c_{n}\right)$ s.t. $c_{n}\searrow0$. Then there
exists a sequence of affine hyperplanes $\left(\mathcal{L}_{n}\right)_{n\in\mathbb{N}}$
s.t. $\forall n\in\mathbb{N},\,\forall i\in\Lambda,\,\varphi_{i}F\cap\mathcal{L}_{n}^{\left(c_{n}\right)}\neq\emptyset$.
Let $C\subset\mathbb{R}^{d}$ be a closed ball containing $F^{\left(c_{0}\right)}$
(hence intersecting all the affine hyperplanes $\mathcal{L}_{n}$),
and denote $\mathcal{L}_{n}^{\prime}=\mathcal{L}_{n}\cap C$ for every
$n$. By compactness of $\Omega_{C}$, taking a subsequence we may
assume that $\mathcal{L}_{n}^{\prime}\rightarrow A$ for some $A\in\Omega_{C}$.
By Lemma \ref{lem:sets contained in hyperplanes iff 0 width}, $A$
is contained in some affine hyperplane $\mathcal{L}$. Now, given
$\varepsilon>0$, taking $n$ large enough s.t. $d_{H}\left(\mathcal{L}_{n}^{\prime},A\right)<\varepsilon/2$
and $c_{n}<\varepsilon/2$, we obtain $\mathcal{L}^{\left(\varepsilon\right)}\cap\varphi_{i}F\neq\emptyset$
for every $i\in\Lambda$%
. Since this is true for every $\varepsilon>0$ and each $\varphi_{i}F$
is closed, this implies that $\mathcal{L}\cap\varphi_{i}F\neq\emptyset$
for every $i\in\Lambda$.
\end{proof}
The following Proposition relates the concept of diffuseness of IFSs
with that of diffuseness of subsets of $\mathbb{R}^{d}$ as defined
in Definition \ref{def:diffuse}. 
\begin{prop}
\label{prop:diffuse tree implies diffuse set}Let $\Phi=\left\{ \varphi_{i}\right\} _{i\in\Lambda}$
be a similarity IFS with contracting ratios $\left\{ r_{i}\right\} _{i\in\Lambda}$
and attractor $K$. Let $T\in\mathscr{T}_{\Lambda}$ be a $\left(K,c\right)$-diffuse
tree, then $\gamma_{\Phi}\left(\partial T\right)$ is hyperplane $c\frac{r_{min}}{\text{diam}\left(K\right)}$-diffuse. 
\end{prop}

\begin{proof}
Denote $E=\gamma_{\Phi}\left(\partial T\right)$ and $\Delta=\text{diam}\left(K\right)$.
Assume we are given some $\xi\in\left(0,\,\Delta\cdot r_{min}\right)$,
$x\in E$, and an affine hyperplane $\mathcal{L}\subset\mathbb{R}^{d}$.
Let $i=i_{1}...i_{n}\in\Lambda^{n}$ be a finite word s.t. $x\in\varphi_{i}K$
and $\frac{\xi}{\Delta}r_{min}<r_{i}\leq\frac{\xi}{\Delta}$ (in order
to find such $i$, let $k=k_{1}k_{2}...\in\Lambda^{\mathbb{N}}$ be
s.t. $\gamma_{\Phi}\left(k\right)=x$, and let $n\in\mathbb{N}$ be
the unique integer s.t. $r_{k_{1}}\cdots r_{k_{n}}\leq\frac{\xi}{\Delta}<r_{k_{1}}\cdots r_{k_{n-1}}$,
then take $i=k_{1}...k_{n}\in\Lambda^{n}$). Since $\varphi_{i}^{-1}\left(\mathcal{L}\right)$
is still an affine hyperplane, there exists some $j\in W_{T}\left(i\right)$
s.t. $\varphi_{j}K\cap\left(\varphi_{i}^{-1}\left(\mathcal{L}\right)\right)^{\left(c\right)}=\emptyset$.
Applying $\varphi_{i}$ we get that $\varphi_{ij}K\cap\mathcal{L}^{\left(r_{i}c\right)}=\emptyset$.
Since $r_{i}\leq\frac{\xi}{\Delta}$, $\text{diam}\left(\varphi_{i}K\right)=r_{i}\Delta\leq\xi$
and therefore $\varphi_{ij}K\subseteq\varphi_{i}K\subseteq B_{\xi}\left(x\right)$.
Noting that $\varphi_{ij}K\cap K\neq\emptyset$ we have shown that
$K\cap B_{\xi}\left(x\right)\setminus\mathcal{L}^{\left(r_{i}c\right)}\neq\emptyset$,
and since $r_{i}>\frac{\xi}{\Delta}r_{min}$ we are done. 
\end{proof}
Proposition \ref{prop:diffuse tree implies diffuse set} implies in
particular that whenever a similarity IFS is diffuse, its attractor
is hyperplane diffuse.
\begin{thm}
\label{thm:When attractors are hyperplane diffuse}Let $\Phi=\left\{ \varphi_{i}\right\} _{i\in\Lambda}$
be a similarity IFS in $\mathbb{R}^{d}$ with attractor $K$. The
following are equivalent.
\begin{enumerate}
\item \label{enu:not contained in a hyperplane}$K$ is not contained in
an affine hyperplane.
\item \label{enu:exists a diffuse section}There exists a diffuse section
$\Pi\subset\Lambda^{*}$ (i.e., such that the IFS $\left\{ \varphi_{i}\right\} _{i\in\Pi}$
is diffuse).
\item \label{enu:-K is hyperplane diffuse}$K$ is hyperplane diffuse.
\end{enumerate}
\end{thm}

\begin{proof}
$\eqref{enu:not contained in a hyperplane}\implies\eqref{enu:exists a diffuse section}:$
Assume that (\ref{enu:exists a diffuse section}) does not hold, so
that every section $\Pi$ is not diffuse. So given some $\varepsilon>0$,
let $\Pi$ be a section s.t. $\forall i\in\Pi,\,\text{diam}\left(\varphi_{i}K\right)<\varepsilon/2$.
Since $\Pi$ is not $\varepsilon/2$-diffuse, there is some affine
hyperplane $\mathcal{L}$, s.t. $\mathcal{L}^{\left(\varepsilon/2\right)}\cap\varphi_{i}K\neq\emptyset$
for every $i\in\Pi$. Since $\text{diam}\left(\varphi_{i}K\right)<\varepsilon/2$
this implies that $\forall i\in\Pi,\,\varphi_{i}K\subset\mathcal{L}^{\left(\varepsilon\right)}$,
hence $K\subset\mathcal{L}^{\left(\varepsilon\right)}$. Taking $\varepsilon$
to 0 implies that $K$ is a.s. contained in an affine hyperplane by
Lemma \ref{lem:sets contained in hyperplanes iff 0 width}.\\
$\eqref{enu:exists a diffuse section}\implies\eqref{enu:-K is hyperplane diffuse}:$
This follows immediately from Proposition \ref{prop:diffuse tree implies diffuse set}.\\
$\eqref{enu:-K is hyperplane diffuse}\implies\eqref{enu:not contained in a hyperplane}$:
Follows from the definition of the hyperplane diffuse property.
\end{proof}
Note that the OSC is not needed for Theorem \ref{thm:When attractors are hyperplane diffuse}.
It is also worth mentioning that the above conditions are equivalent
to irreducibility of $\Phi$ as defined in \cite{Kleinbock2005} (see
Remark \ref{rem:irreducible <=00003D> not on a hyperplane }).

One of the conditions of Theorem \ref{thm: main generalized} is that
the GWF is non-planar. We now list a few equivalent conditions to
non-planarity of supercritical GWFs.
\begin{prop}
\label{prop: equivalent conditions for not in a hyperplane}Let E
be a supercritical GWF w.r.t. a similarity IFS $\Phi=\left\{ \varphi_{i}\right\} _{i\in\Lambda}$
and offspring distribution $W$, and let $T$ be the corresponding
GWT. Denote by $K$ the attractor of $\Phi$. The following conditions
are equivalent.
\begin{enumerate}
\item \label{enu:non planar}E is non-planar.
\item \label{enu:E not contained in affine hyperplane}$\mathbb{P}\left(E\text{ is not contained in an affine hyperplane\ensuremath{|} nonextinction}\right)=1.$
\item \label{enu:K not contained in affine hyperplane}$K$ is not contained
in an affine hyperplane.
\item \label{enu:exists diffuse section}$\exists\Pi\subseteq\Lambda^{*}$
section, and $c>0$ s.t. $\mathbb{P}\left(T_{\Pi}\text{ is \emph{\ensuremath{\left(K,c\right)}}-diffuse}\right)>0$.
\end{enumerate}

\end{prop}

\begin{proof}
First note that since there are only countably many sections (for
every $n$ there is only a finite number of sections of size $n$),
(\ref{enu:exists diffuse section}) is equivalent to the following
statement:
\[
\mathbb{P}\left(\text{\ensuremath{\exists\Pi\subseteq\Lambda^{*}} section, and \emph{c>0} s.t. }T_{\Pi}\text{ is \emph{\ensuremath{\left(K,c\right)}}-diffuse}\right)>0.
\]
\\
$\eqref{enu:non planar}\implies\eqref{enu:E not contained in affine hyperplane}$:
Follows from Proposition \ref{prop:0-1 law}, namely, since 
\[
\mathbb{P}\left(E\text{ is not contained in an affine hyperplane}\right)>0,
\]
almost surely given nonextinction $\exists v\in T$ s.t. $\gamma_{\Phi}\left(\partial T^{v}\right)$
is not contained in an affine hyperplane, which implies that $E$
is not contained in an affine hyperplane.\\
$\eqref{enu:E not contained in affine hyperplane}\implies\eqref{enu:K not contained in affine hyperplane}$:
Trivial\\
$\eqref{enu:K not contained in affine hyperplane}\implies\eqref{enu:exists diffuse section}$:
We first prove the following claim by induction:
\begin{claim*}
For every integer $0\leq k\leq d-1$ there exists $\left\{ i^{1},...,i^{k+2}\right\} \subseteq\Lambda^{*}$
s.t. the following hold:
\end{claim*}
\begin{enumerate}
\item[(a)] $\forall i,j\in\left\{ i^{1},...,i^{k+2}\right\} $, $\varphi_{i}K\cap\varphi_{j}K=\emptyset$
(hence $i\ngeq j$ and $j\ngeq i$).
\item[(b)] $\mathbb{P}\left(\left\{ i^{1},...,i^{k+2}\right\} \subseteq T\right)>0$
\item[(c)] For every integer $0\leq n\leq k$, for every $n$-dimensional affine
subspace $\mathcal{L}\subseteq\mathbb{R}^{d}$ that intersects $\varphi_{i^{1}}K,...,\varphi_{i^{n+1}}K$,
we have $\varphi_{i^{n+2}}K\cap\mathcal{L}=\emptyset$.
\end{enumerate}
\begin{proof}[Proof of claim]
For $k=0$: since the process is supercritical, there exist $i,j\in\Lambda$
s.t. $\mathbb{P}\left(\left\{ i,j\right\} \subseteq W\right)>0$.
Therefore, for some $i^{\prime}>i$ and $j^{\prime}>j$, $\varphi_{i^{\prime}}K\cap\varphi_{j^{\prime}}K=\emptyset$%
{} and all 3 conditions are fulfilled by $\left\{ i^{\prime},j^{\prime}\right\} $.

Assume $\left\{ i^{1},...,i^{k+1}\right\} \subseteq\Lambda^{*}$ satisfies
(a), (b), (c) for $k-1$. By the case $k=0$, there are $a,\,b>i^{k+1}$
s.t. $\varphi_{a}K\cap\varphi_{b}K=\emptyset$, and $\mathbb{P}\left(\left\{ a,b\right\} \subseteq T\right)>0$.
Pick any $x_{1}\in\varphi_{i^{1}}K,...,\,x_{k}\in\varphi_{i^{k}}K,\,x_{k+1}\in\varphi_{a}K$.
By assumption $x_{1},...,x_{k+1}$ are affinely independent, and thus
span a unique $k$-dimensional affine subspace $\mathcal{L}\subset\mathbb{R}^{d}$.
By Theorem \ref{thm:When attractors are hyperplane diffuse}, $\Phi$
has some diffuse section, so there is some $j\in\Lambda^{*}$ s.t.
$\varphi_{bj}K\cap\mathcal{L}=\emptyset$. Therefore, $\varphi_{bj}K$
does not intersect small enough perturbations of $\mathcal{L}$ as
well. So there are $j^{1}>i^{1},...,\,j^{k}>i^{k},\,j^{k+1}>a$ s.t.
$\varphi_{bj}K\cap\mathcal{L}^{\prime}=\emptyset$ for every $k$-dimensional
affine subspace $\mathcal{L}^{\prime}$ which intersects the sets
$\varphi_{j^{1}}K,...,\varphi_{j^{k+1}}K$. Denoting $j^{k+2}=bj$,
the set $\left\{ j^{1},...,j^{k+2}\right\} $ satisfies conditions
(a), (b), (c) for $k$.
\end{proof}
Let $i^{1},...,i^{d+1}\in\Lambda^{*}$ be the elements whose existence
is guaranteed by the claim for $k=d-1$. By property (a), there is
some section $\Pi\subset\Lambda^{*}$ s.t. $i^{1},...,i^{d+1}\in\Pi$.
By property (b), $\mathbb{P}\left(\left\{ i^{1},...,i^{d+1}\right\} \subseteq T_{\Pi}\right)>0$.
And by property (c) combined with Proposition \ref{prop:diffuse implies 0-diffuse},
$\left\{ i^{1},...,i^{d+1}\right\} $ is $\left(K,c\right)$-diffuse
for some $c>0$. Since $\mathbb{P}\left(\left\{ i^{1},...,i^{d+1}\right\} \subseteq T_{\Pi}\right)>0$,
then $\mathbb{P}\left(\left\{ i^{1},...,i^{d+1}\right\} ^{n}\subseteq T_{\Pi^{n}}\right)>0$.
To summarize, we have found a section $\Pi^{n}\subset\Lambda^{*}$
and $c>0$ s.t. $\mathbb{P}\left(T_{\Pi^{n}}\text{ is \emph{\ensuremath{\left(K,c\right)}}-diffuse}\right)>0$.\\
$\eqref{enu:exists diffuse section}\implies\eqref{enu:non planar}$:
By Lemma \ref{lem:diffuse implies (K,c)-diffuse} and Theorem \ref{thm:When attractors are hyperplane diffuse},
Assuming (\ref{enu:exists diffuse section}) implies that there exists
some section $\Pi$ s.t. 
\[
\mathbb{P}\left(\text{the attractor of \ensuremath{\left\{ \varphi_{i}\right\} _{i\in T_{\Pi}}} is not contained in an affine hyperplane}\right)>0.
\]
Consider the GWF with IFS $\left\{ \varphi_{i}\right\} _{i\in T_{\Pi}}$
and offspring distribution $\sim T_{\Pi}$. Obviously it has the same
law as \emph{E. }So without loss of generality we may assume that
\[
\mathbb{P}\left(\text{the attractor of \ensuremath{\left\{ \varphi_{i}\right\} _{i\in W}} is not contained in an affine hyperplane}\right)>0
\]
 and take $\Lambda$ instead of $\Pi$ for convenience of notations.
Let $A\subseteq\Lambda$ be s.t. the attractor of $\left\{ \varphi_{i}\right\} _{i\in A}$,
which we denote by $K_{A}$, is not contained in an affine hyperplane
and $\mathbb{P}\left(W=A\right)>0$. Since $K_{A}$ is not contained
in an affine hyperplane, by Lemma \ref{lem:sets contained in hyperplanes iff 0 width}
there exists some $\varepsilon>0$ s.t. $\forall\mathcal{L}\subset\mathbb{R}^{d}$
affine hyperplane, $K_{A}\nsubseteq\mathcal{L}^{\left(\varepsilon\right)}$.
Therefore, taking $n\in\mathbb{N}$ large enough, for every affine
hyperplane $\mathcal{L}$ there exists some $i\in A^{n}$ s.t. $\varphi_{i}K\cap\mathcal{L}^{\left(\varepsilon/2\right)}=\emptyset$.
Now, since $\mathbb{P}\left(W=A\right)>0$, there exists a positive
probability that $A^{n}\subset T_{n}$ and for every $i\in A^{n}$,
$T^{i}$ is infinite. Obviously, in this case $E$ is not contained
in an affine hyperplane.
\end{proof}
A discussion about hyperplane diffuseness of GWFs is postponed to
the Appendix for the sake of a more fluent reading of the paper. The
main point of this discussion is that in many cases, GWFs are a.s.
not hyperplane diffuse. In particular, fractal percolation sets are
almost surely not hyperplane diffuse.

\section{The fractal percolation case\label{sec:the-fractal-percolation}}

\subsection{Strategy}

In this section we are going to prove Theorem \ref{thm: main generalized}
for the special case of fractal percolation. In order to do that,
we want to use Theorem \ref{thm:HAW intersects diffuse sets}, but
as mentioned earlier, fractal percolation sets are a.s. not hyperplane
diffuse (see Appendix \ref{appendix:Microsets-of-Galton-Watson},
specifically Corollary \ref{cor:Fractal percolation is not hyperplane diffuse}).
So the next best thing we can do is to show that fractal percolation
sets have diffuse subsets. Therefore, what we are going to prove is
the following theorem (which is just Theorem \ref{thm:Main theorem diffuse subses}
stated for the case of fractal percolation):
\begin{thm}
\label{thm:subsets exist}Let $E$ be the limit set of a supercritical
fractal percolation process in $\left[0,1\right]^{d}$ with parameters
$b,p$. Then a.s. conditioned on nonextinction, there exists a sequence
of subsets $D_{n}\subseteq E$ which are hyperplane diffuse and Ahlfors-regular
with $\dim_{H}\left(D_{n}\right)\nearrow\log_{b}m$.
\end{thm}

In order to find the subsets $D_{n}$, we find appropriate subsets
of $\partial T$ (where $T$ is the corresponding GWT) and project
them using the coding map $\gamma_{b}$. Our way of doing so goes
through the following definition.
\begin{defn}
Let $T\in\mathscr{T}_{\mathbb{A}}$ be a   tree with alphabet $\mathbb{A}$.
For any integer $1\leq k$ we define the \emph{$k$-compressed }%
\begin{lyxgreyedout}
Find a better name than k-compressed tree%
\end{lyxgreyedout}
\emph{tree} w.r.t. to $T$ as the tree with alphabet $\mathbb{A}^{k}$
given by $T_{\left(k\right)}=\bigcup\limits _{n\geq0}T_{k\cdot n}\in\mathscr{T}_{\mathbb{A}^{k}}$.
\end{defn}

Note that in the above definition we identify $\mathbb{A}^{k\cdot n}\thickapprox\left(\mathbb{A}^{k}\right)^{n}$
by the correspondence $\left(a_{1},...,a_{nk}\right)\longleftrightarrow\left(\left(a_{1},...,a_{k}\right),...,\left(a_{\left(n-1\right)k+1},...,a_{nk}\right)\right)$.
We shall make this identification as well as $\left(\mathbb{A}^{k}\right)^{\mathbb{N}}\thickapprox\mathbb{A}^{\mathbb{N}}$
throughout the paper without farther mention.

It is important to note that when $T$ is a GWT with alphabet $\mathbb{A}$,
$T_{\left(k\right)}$ is itself a GWT with alphabet $\mathbb{A}^{k}$
and offspring distribution $\sim T_{k}$. However, one needs to be
careful and notice that if $T$ is a GWT with a binomial offspring
distribution, then $T_{\left(k\right)}$, although being a GWT, no
longer has a binomial offspring distribution due to the dependencies
between the generations. Also note that if we identify $\left(\mathbb{A}^{k}\right)^{\mathbb{N}}$
with $\mathbb{A^{N}}$ as mentioned above, then $T$ and $T_{\left(k\right)}$
have the same boundaries. 

The strategy of the proof is going to be as follows: Choose some $c\in\left(1,m\right)$
s.t. for some $k_{0}\in\mathbb{N}$, $c^{k_{0}}\in\mathbb{Z}$ where
$m=pb^{d}$%
. We are going to show that taking a very large $k$ (this may be
thought of as dividing each cube into many many small subcubes of
the same sidelength at each step), the tree $T_{\left(k\right)}$
will almost surely have a vertex $v\in T_{\left(k\right)}$ s.t. the
tree $\left(T_{\left(k\right)}\right)^{v}$ has a subtree $S$ with
the following 2 properties:
\begin{enumerate}
\item Each element of $S$ has exactly $c^{k}$ children (assuming $k=tk_{0}$
for some $t\in\mathbb{N}$, $c^{k}$ is an integer).
\item $S$ is diffuse.
\end{enumerate}
We then take $D=\varphi_{v}\left(\gamma_{b}\left(\partial S\right)\right)$.

The first property listed above implies that the natural measure constructed
on $\gamma_{b}\left(\partial S\right)$ is $\log_{b}c$-Ahlfors-regular,
hence $D$ is $\log_{b}c$-Ahlfors-regular. The second property ensures
that $\gamma_{b}\left(\partial S\right)$ is hyperplane diffuse, and
therefore so is $D$. Finally, we take $c\nearrow m$ through any
sequence.

\subsection{a-ary subtrees of compressed GWTs}

In this subsection we deal with the existence of \emph{a}-ary subtrees
in \emph{k}-compressed supercritical GWTs, where we are actually going
to be interested in the asymptotics as $k\rightarrow\infty$. We will
use Theorem \ref{thm:fixed point}, which in this case reduces to
Pakes-Dekking theorem in its original formulation since we are only
interested in \emph{a-}ary subtrees. The reader should keep in mind
that when $T$ is a GWT with alphabet $\mathbb{A}$, $T_{\left(k\right)}$
may also be considered a GWT with alphabet $\mathbb{A}^{k}$ and offspring
distribution $\sim T_{k}$. Hence, applying Theorem \ref{thm:fixed point}
to the tree $T_{\left(k\right)}$ requires the investigation of the
random set $T_{k}^{\left(s\right)}$. Since in this subsection we
are only interested in the size of the set $T_{k}^{\left(s\right)}$,
we denote $Z_{k}^{\left(s\right)}=\left|T_{k}^{\left(s\right)}\right|$
(this is a number valued random variable).

Let $T$ be a GWT with offspring distribution $W$, s.t. $\mathbb{E}\left(\left|W\right|\right)=m>1$.
Given any integers $k,\,a>0$ , we define the function $g_{k,a}:\left[0,1\right]\rightarrow\left[0,1\right]$
by 
\[
g_{k,a}\left(s\right)=\mathbb{P}\left(Z_{k}^{\left(s\right)}<a\right).
\]
The following lemma is one of the main ingredients for the proof of
Theorem \ref{thm:subsets exist}.
\begin{lem}
\label{lem:p>gamma}Let $\left(a_{k}\right)_{k\in\mathbb{N}}$ be
a sequence of positive integers s.t. $\limsup\limits _{k\rightarrow\infty}\sqrt[{\scriptstyle k}]{a_{k}}<m$,
then $g_{k,a_{k}}\left(s\right)\underset{k\rightarrow\infty}{\longrightarrow}\mathbb{P}\left(\text{extinction}\right)$
for every $s\in\left(0,1\right)$.
\end{lem}

\begin{proof}
Let $\alpha\in\left(0,1\right)$ be s.t. $\limsup\limits _{k\rightarrow\infty}\sqrt[{\scriptstyle k}]{a_{k}}<\alpha m$.
This implies that $a_{k}<\alpha^{k}m^{k}$ whenever $k$ is large
enough . Given some $\varepsilon>0$, by Corollary \ref{cor:kesten - stigum}
there is some $c>0$ s.t.
\[
\mathbb{P}\left(\dfrac{Z_{k}}{m^{k}}>c\mid\,\text{nonextinction}\right)>1-\varepsilon
\]
whenever $k$ is large enough. 

Fix some $s\in\left(0,1\right)$. 
\[
\begin{array}{l}
g_{k,a_{k}}\left(s\right)=\mathbb{P}\left(Z_{k}^{\left(s\right)}<a_{k}\right)=\\
\mathbb{P}\left(Z_{k}^{\left(s\right)}<a_{k}\mid\text{extinction}\right)\cdot\mathbb{P}\left(\text{extinction}\right)+\\
\mathbb{P}\left(Z_{k}^{\left(s\right)}<a_{k}\mid\text{nonextinction}\right)\cdot\mathbb{P}\left(\text{nonextinction}\right)
\end{array}
\]
Clearly $\mathbb{P}\left(Z_{k}^{\left(s\right)}<a_{k}\mid\text{extinction}\right)\underset{k\rightarrow\infty}{\longrightarrow}1$. 

On the other hand, 
\[
\begin{array}{l}
\mathbb{P}\left(Z_{k}^{\left(s\right)}<a_{k}\mid\text{nonextinction}\right)=\\
\mathbb{P}\left(Z_{k}^{\left(s\right)}<a_{k}\mid\text{nonextinction},\,Z_{k}>cm^{k}\right)\cdot\mathbb{P}\left(Z_{k}>cm^{k}\mid\text{nonextinction}\right)+\\
\mathbb{P}\left(Z_{k}^{\left(s\right)}<a_{k}\mid\text{nonextinction},\,Z_{k}\leq cm^{k}\right)\cdot\mathbb{P}\left(Z_{k}\leq cm^{k}\mid\text{nonextinction}\right)
\end{array}
\]
Now, for $k$ large enough, $\mathbb{P}\left(Z_{k}\leq cm^{k}\mid\text{nonextinction}\right)<\varepsilon$.
Finally, by Chebyshev's inequality we get:
\[
\begin{array}{l}
{\displaystyle \mathbb{P}\left(Z_{k}^{\left(s\right)}<a_{k}\mid\text{nonextinction},\,Z_{k}>cm^{k}\right)\leq}\\
{\displaystyle \mathbb{P}\left(\text{Bin}\left(cm^{k},\,\left(1-s\right)\right)<\alpha^{k}m^{k}\right)\leq}\\
{\displaystyle \dfrac{s\left(1-s\right)cm^{k}}{\left[\left(1-s\right)cm^{k}-\alpha^{k}m^{k}\right]^{2}}=}\\
{\displaystyle \dfrac{1}{m^{k}}\cdot\dfrac{s\left(1-s\right)c}{\left[\left(1-s\right)c-\alpha^{k}\right]^{2}}}
\end{array}
\]
Since $\alpha<1$, the last term goes to $0$ as $k\rightarrow\infty$,
and overall we get 
\[
g_{k,a_{k}}\left(s\right)\underset{k\rightarrow\infty}{\longrightarrow}\mathbb{P}\left(\text{extinction}\right)
\]
 as claimed.
\end{proof}
Since $\mathbb{P}\left(\text{extinction}\right)<1$ and for every
$k$, $g_{k,a_{k}}$ is continuous, the result of Lemma \ref{lem:p>gamma}
implies that the graph of $g_{k,a_{k}}$ has to intersect the graph
of $y=x$ in the interval $\left(0,1\right)$ when $k$ gets large
enough, which means that it has a fixed point smaller than 1. In fact,
it shows that the smallest fixed point of $g_{k,a_{k}}$ converges
to $\mathbb{P}\left(\text{extinction}\right)$ as $k\rightarrow\infty$,
which implies by Theorem \ref{thm:fixed point} that 
\[
\mathbb{P}\left(T_{\left(k\right)}\text{ contains an }a_{k}\text{-ary subtree}|\text{ nonextinction}\right)\underset{k\rightarrow\infty}{\longrightarrow}1.
\]

Moreover, by Corollary \ref{cor:Pakes-Dekking result} , choosing
$k$ large enough, almost surely conditioned on nonextinction, there
exists some vertex $v\in T_{\left(k\right)}$ (actually infinitely
many vertices) s.t. the descendants tree $\left(T_{\left(k\right)}\right)^{v}$
has an $a_{k}$-ary subtree. %

Since this reasoning will recur in what follows, we now state a general
lemma which fits the situation described above.
\begin{lem}
\label{lem:Graph fixed point}Let $T$ be a supercritical GWT with
alphabet $\mathbb{A}$. Given a sequence $\left(\mathscr{A}_{k}\right)_{k\in\mathbb{N}}$
s.t. for each $k$, $\mathscr{A}_{k}$ is a nonempty collection of
nonempty subsets of $\mathbb{A}^{k}$, define $g_{\mathscr{A}_{k}}:\left[0,1\right]\to\left[0,1\right]$
by $g_{\mathscr{A}_{k}}\left(s\right)=\mathbb{P}\left(T_{k}^{\left(s\right)}\notin\overline{\mathscr{A}_{k}}\right)$.
Assume that $\forall s\in\left(0,1\right)$, $g_{\mathscr{A}_{k}}\left(s\right)\underset{k\to\infty}{\longrightarrow}\mathbb{P}\left(\text{extinction}\right)$.
Then $\exists K\in\mathbb{N}$ s.t. $\forall k>K$, almost surely
conditioned on nonextinction, there exist infinitely many vertices
$v\in T_{\left(k\right)}$ s.t. the descendants tree $\left(T_{\left(k\right)}\right)^{v}$
has an $\mathscr{A}_{k}$-subtree.
\end{lem}

\begin{rem}
Conversely, for any $k\in\mathbb{N}$, assume that $a>m^{k}$, then
there exists some $\alpha>1$ s.t. $a>\alpha m^{k}$. In this case,
if $T_{\left(k\right)}$ has an $a$-ary subtree, then $Z_{nk}>\alpha^{n}m^{kn}$
for every $n\in\mathbb{N}$, which implies that $\frac{Z_{l}}{m^{l}}$
can't converge to a finite number as $l\to\infty$, and by Kesten-Stigum
theorem this has probability 0 of occurring. 
\end{rem}

The following lemma is the reason we are interested in \emph{a}-ary
subtrees.
\begin{lem}
\label{lem:a-ary implies dimension}Let $S\in\mathscr{T}_{\Lambda_{b}}$
be an a-ary tree, then $\gamma_{b}\left(\partial S\right)$ is $\log_{b}a$-Ahlfors-regular. 
\end{lem}

To prove this lemma one should construct the natural measure on $\gamma_{b}\left(\partial S\right)$
by equally distributing mass at each level of the tree. It is not
hard to show that this measure is Ahlfors regular with dimension $\log_{b}a$.
For more details see the proof of the more general Lemma \ref{lem:rho^alpha-ary implies alpha Ahlfors regular}
in Section \ref{sec:The-general-case}.

\subsection{Diffuse subtrees}

In this subsection we show the existence of diffuse subtrees of $T_{\left(k\right)}$
when $k$ is large enough. In fact we are going to show the existence
of $I^{d}$-diffuse subtrees (where $I^{d}=\left[0,1\right]^{d}\subset\mathbb{R}^{d}$
is the unite cube), so that for each vertex of the tree, the \emph{b}-adic
cubes corresponding to its children do not all intersect the same
affine hyperplane. As before, we denote $\Lambda_{b}=\left\{ 0,1,...,b-1\right\} ^{d}$,
and $\gamma_{b}$ the corresponding coding map given by 
\[
\gamma_{b}\left(a\right)=\sum\limits _{k=1}^{\infty}b^{-k}a_{k},\,\,\,\forall a\in\left(\Lambda_{b}\right)^{\mathbb{N}}.
\]
Given any positive integer $k$, we shall identify $\Lambda_{b^{k}}$
with $\left(\Lambda_{b}\right)^{k}$ in the obvious way. %

We denote

\[
\mathscr{D}_{b}=\left\{ X\subseteq\Lambda_{b}:\,X\text{ is \ensuremath{I^{d}}-diffuse }\right\} 
\]
By definition $\mathscr{D}_{b}$-trees are diffuse, therefore by Proposition
\ref{prop:diffuse tree implies diffuse set} they are projected by
$\gamma_{b}$ to hyperplane diffuse sets.

We now show the existence of $\mathscr{D}_{b^{k}}$-subtrees in $k$-compressed
GWTs whenever $k$ gets large enough. In fact, we are going to show
the existence of $\mathscr{D}_{b,k}$-subtrees where 
\[
\mathscr{D}_{b,k}=\left\{ X\subseteq\left(\Lambda_{b}\right)^{k}:\,\exists a\in\left(\Lambda_{b}\right)^{k-2},\,a\left(\Lambda_{b}\right)^{2}\subseteq X\right\} .
\]
Since $\forall b\geq2$, $\left(\Lambda_{b}\right)^{2}$ is $I^{d}$-diffuse,
$\mathscr{D}_{b,k}\subset\mathscr{D}_{b^{k}}$ and therefore every
$\mathscr{D}_{b,k}$-subtree is also a $\mathscr{D}_{b^{k}}$-subtree.
Note that in the definition of $\mathscr{D}_{b,k}$ we go 2 steps
backwards instead of 1 only to include the case $b=2$ for which $\Lambda_{2}$
itself is not $I^{d}$-diffuse, but $\left(\Lambda_{2}\right)^{2}$
is.

For the following lemma let $T$ be a supercritical GWT with alphabet
$\Lambda_{b}$ and binomial offspring distribution with parameter
\emph{p}. Let $g_{\mathscr{D}_{b,k}}:\left[0,1\right]\rightarrow\left[0,1\right]$
be defined by $g_{\mathscr{D}_{b,k}}\left(s\right)=\mathbb{P}\left(T_{k}^{\left(s\right)}\notin\mathscr{D}_{b,k}\right)$.
Note that $\mathscr{D}_{b,k}=\overline{\mathscr{D}_{b,k}}$, so $g_{\mathscr{D}_{b,k}}$
is the function defined in the paragraph preceding Theorem \ref{thm:fixed point}
for the GWT $T_{\left(k\right)}$ and the collection $\mathscr{D}_{b,k}$.
\begin{lem}
\label{lem:diffuse subtree}For every $s\in\left(0,1\right)$, $g_{\mathscr{D}_{b,k}}\left(s\right)\underset{{\scriptstyle k\rightarrow\infty}}{\longrightarrow}\mathbb{P}\left(\text{extinction}\right)$.
\end{lem}

\begin{proof}
By Corollary \ref{cor:kesten - stigum}, for every $\varepsilon>0$
there exists some $c>0$ s.t. whenever $k$ is large enough 
\begin{equation}
\mathbb{P}\left(Z_{k-2}>c\cdot m^{k-2}\mid\,\text{nonextinction}\right)>1-\varepsilon.\label{eq:kesten - stigum}
\end{equation}
Fix any $s\in\left(0,1\right)$. Note that each element in $T_{k-2}$
has some positive probability $p_{0}$, that all its level 2 descendants
will be contained in $T_{k}^{\left(s\right)}$, i.e., 
\[
\forall v\in\left(\Lambda_{b}\right)^{k-2},\,\mathbb{P}\left(v\left(\Lambda_{b}\right)^{2}\subseteq T_{k}^{\left(s\right)}|\,v\in T\right)=p_{0}>0.
\]
In case this event occurs for at least one element of $T_{k-2}$ we
immediately have $T_{k}^{\left(s\right)}\in\mathscr{D}_{b,k}$. This
fact, combined with (\ref{eq:kesten - stigum}) implies that 
\[
\mathbb{P}\left(T_{k}^{\left(s\right)}\in\mathscr{D}_{b,k}\vert\text{nonextinction}\right)\underset{{\scriptstyle k\rightarrow\infty}}{\longrightarrow}1
\]
and therefore

\[
g_{\mathscr{D}_{b,k}}\left(s\right)\underset{{\scriptstyle k\rightarrow\infty}}{\longrightarrow}\mathbb{P}\left(\text{extinction}\right)
\]
as claimed.
\end{proof}

\subsection{Final step\label{subsec:Final-step - fractal percolation case}}

We now combine the two lemmas \ref{lem:p>gamma}, \ref{lem:diffuse subtree}
in order to prove Theorem \ref{thm:subsets exist}.

We will need the following notation: $\mathscr{Z}=\left\{ x\in\left[1,\infty\right):\,\exists k_{0}\in\mathbb{N},\,x^{k_{0}}\in\mathbb{Z}\right\} $.
Notice that $\mathscr{Z}$ is dense in $\left[1,\infty\right)$%
, and that if $x^{k_{0}}\in\mathbb{Z}$, then $\forall t\in\mathbb{N}$,
$x^{t\cdot k_{0}}\in\mathbb{Z}$.
\begin{proof}[Proof of Theorem \ref{thm:subsets exist}]
 Let $T$ be a GWT which corresponds to a supercritical fractal percolation
process with parameters $b,p$. Given some number $c\in\left(1,m\right)\cap\mathscr{Z}$
(where $m=pb^{d}$), denote $\mathscr{C}_{k}=\left\{ B\subseteq\left(\Lambda_{b}\right)^{k}:\,\left|B\right|=\left\lfloor c^{k}\right\rfloor \right\} $.
For every $s\in\left(0,1\right)$, by Lemma \ref{lem:p>gamma}, 
\[
\mathbb{P}\left(T_{k}^{\left(s\right)}\notin\overline{\mathscr{C}_{k}}\vert\,\text{nonextinction}\right)\underset{{\scriptstyle k\rightarrow\infty}}{\longrightarrow}0
\]
and by Lemma \ref{lem:diffuse subtree}, 
\[
\mathbb{P}\left(T_{k}^{\left(s\right)}\notin\mathscr{D}_{b,k}\vert\,\text{nonextinction}\right)\underset{{\scriptstyle k\rightarrow\infty}}{\longrightarrow}0.
\]
Combining these two together we get that for every $s\in\left(0,1\right)$,

\[
\mathbb{P}\left(T_{k}^{\left(s\right)}\notin\mathscr{D}_{b,k}\cap\overline{\mathscr{C}_{k}}\vert\,\text{nonextinction}\right)\underset{{\scriptstyle k\rightarrow\infty}}{\longrightarrow}0
\]
and therefore

\[
\mathbb{P}\left(T_{k}^{\left(s\right)}\notin\mathscr{D}_{b,k}\cap\overline{\mathscr{C}_{k}}\right)\underset{{\scriptstyle k\rightarrow\infty}}{\longrightarrow}\mathbb{P}\left(\text{extinction}\right).
\]

Observe that 
\begin{equation}
\mathscr{D}_{b,k}\cap\overline{\mathscr{C}_{k}}=\overline{\mathscr{D}_{b,k}\cap\mathscr{C}_{k}}=\overline{\mathscr{D}_{b,k}\cap\overline{\mathscr{C}_{k}}}.\label{eq:monotone closure}
\end{equation}
The only part of (\ref{eq:monotone closure}) which does not follow
immediately from the definition of $\mathscr{A}\mapsto\overline{\mathscr{A}}$
and the fact that $\mathscr{D}_{b,k}=\overline{\mathscr{D}_{b,k}}$,
is $\mathscr{D}_{b,k}\cap\overline{\mathscr{C}_{k}}\subseteq\overline{\mathscr{D}_{b,k}\cap\mathscr{C}_{k}}$.
To show that this is correct, let $X\subseteq\left(\Lambda_{b}\right)^{k}$
s.t. $X\in\mathscr{D}_{b,k}\cap\overline{\mathscr{C}_{k}}$, so we
know that $\left|X\right|\geq\left\lfloor c^{k}\right\rfloor $ and
there exists some $a\in\left(\Lambda_{b}\right)^{k-2}$ s.t. $\left\{ aj:\,j\in\left(\Lambda_{b}\right)^{2}\right\} \subseteq X$.
As long as $\left\lfloor c^{k}\right\rfloor \geq b^{2d}$, we may
obviously find a subset $X^{\prime}\subseteq X$ s.t. $\left\{ aj:\,j\in\left(\Lambda_{b}\right)^{2}\right\} \subseteq X^{\prime}$
and $\left|X^{\prime}\right|=\left\lfloor c^{k}\right\rfloor $. Since
$X^{\prime}\in\mathscr{D}_{b,k}\cap\mathscr{C}_{k}$ it follows that
$X\in\overline{\mathscr{D}_{b,k}\cap\mathscr{C}_{k}}$ as claimed. 

By (\ref{eq:monotone closure}), we have shown that $g_{\mathscr{D}_{b,k}\cap\mathscr{C}_{k}}\left(s\right)\underset{k\to\infty}{\longrightarrow}\mathbb{P}\left(\text{extinction}\right)$
for every $s\in\left(0,1\right)$ where $g_{\mathscr{D}_{b,k}\cap\mathscr{C}_{k}}$
is defined as in Theorem \ref{thm:fixed point}. Hence, by Lemma \ref{lem:Graph fixed point},
a.s. conditioned on nonextinction, there exists some $v\in T_{\left(k\right)}$
s.t. $\left(T_{\left(k\right)}\right)^{v}$ has a $\mathscr{D}_{b,k}\cap\mathscr{C}_{k}-\text{subtree}$
when $k$ is large enough. Since $c\in\mathscr{Z}$, we may assume
that $c^{k}\in\mathbb{Z}$. If $S$ is such a subtree, then letting
$D_{c}^{\prime}$ be the projection of $\partial S$ to $\mathbb{R}^{d}$,
by Lemma \ref{lem:a-ary implies dimension}, and Proposition \ref{prop:diffuse tree implies diffuse set},
$D_{c}^{\prime}$ is hyperplane diffuse and $\log_{b}c$-Ahlfors-regular.
The set $D_{c}=\varphi_{v}\left(D_{c}^{\prime}\right)\subseteq E$
has the same properties. Thus, we have shown that for every $c\in\left(1,m\right)\cap\mathscr{Z}$,
a.s. conditioned on nonextinction there exists a subset $D_{c}\subseteq E$
which is $\log_{b}c$-Ahlfors-regular and hyperplane diffuse. 

Now, taking a sequence $c_{n}\nearrow m$ where for each $n$, $c_{n}\in\mathscr{Z}$,
a.s. conditioned on nonextinction there exists a sequence of hyperplane
diffuse subsets $D_{n}:=D_{c_{n}}\subseteq E$ which are Ahlfors-regular
with $\dim_{H}\left(D_{n}\right)\nearrow\log_{b}m$.
\end{proof}
\begin{rem}
In the proof of Theorem \ref{thm:subsets exist} we do not use the
fact that a.s. $\dim_{H}\left(E\right)=\log_{b}\left(pb^{d}\right)$
whenever $E\neq\emptyset$, so in particular we get another proof
of this fact (Theorem \ref{thm:subsets exist} only implies the lower
bound, but the other inequality is straightforward, see e.g. \cite[Lemma 3.7.3]{bishop2016fractals}).
\end{rem}

\section{The general case\label{sec:The-general-case}}

In this section we prove Theorem \ref{thm:Main theorem diffuse subses}
for the general case of arbitrary similarity IFSs. The main difficulty
in this setup arises from allowing different maps in the IFS to have
different contraction ratios, in which case the $k$-compressed trees
have very different weights assigned to the vertices at each level.
In order to deal with this issue, we compress the GWT along sections
where each element of the section has the same weight up to some constant
factor. 

$k$-compression of GWTs are easy to analyze since they may be considered
themselves as GWTs, but compressing a GWT along arbitrary sections
this is no longer the case. While this issue introduces some technical
difficulties and cumbersome notations, the main ideas of the proof
remain the same as in the case of fractal percolation.

\subsection{Sections\label{subsec:Sections}}

We first note that for an IFS $\Phi=\left\{ \varphi_{i}\right\} _{i\in\Lambda}$,
given a section $\Pi\subset\Lambda^{*}$, we may think of the IFS
$\left\{ \varphi_{i}\right\} _{i\in\Pi}$, which obviously has the
same attractor as $\Phi$. 
\begin{lem}
Let T be a GWT with alphabet $\mathbb{A}$ and offspring distribution
$W$, and with weights $\left\{ r_{i}\right\} _{i\in\mathbb{A}}$.
Assume that $\mathbb{E}\left(\sum\limits _{i\in W}r_{i}^{\delta}\right)=1$.
Then for every section $\Pi\subset\mathbb{A}^{*}$, $\mathbb{E}\left(\sum\limits _{i\in T_{\Pi}}r_{i}^{\delta}\right)=1$.%
\end{lem}

{} The proof of the lemma may be carried out by induction on the size
of $\Pi$ and is left as an exercise for the reader.
\begin{defn}
Given an alphabet $\mathbb{A}$ with weights $\left\{ r_{i}\right\} _{i\in\mathbb{A}}$
and a positive number $\rho\in\left(0,r_{min}\right)$, we denote
by $\Pi_{\rho}$ the section given by 
\[
\Pi_{\rho}=\left\{ i\in\mathbb{A}^{*}:\,r_{i}\leq\rho<r_{i_{1}}\cdot...\cdot r_{i_{\left|i\right|-1}}\right\} .
\]
Note that $\forall i\in\Pi_{\rho},\,\rho\cdot r_{min}<r_{i}\leq\rho$
(recall that $r_{min}=\min\left\{ r_{i}:\,i\in\mathbb{A}\right\} $)
\end{defn}

\begin{lem}
\label{lem:size of cutsets}Let $T$ be a supercritical GWT with alphabet
$\mathbb{A}$, weights $\left\{ r_{i}\right\} _{i\in\mathbb{A}}$,
and some offspring distribution $W$, and let $\delta$ \textup{\emph{satisfy}}\textup{
$\mathbb{E}\left(\sum\limits _{i\in W}r_{i}^{\delta}\right)=1$.}
Then $\forall\alpha<\delta$, a.s. conditioned on nonextinction there
exists some $\rho_{0}>0$ s.t. $\forall\rho<\rho_{0}$, $\left|T_{\Pi_{\rho}}\right|>\frac{1}{\rho^{\alpha}}$.
\end{lem}

In order to prove the lemma we need the following theorem by K. Falconer. 
\begin{thm}
\label{thm:Falconer's Theorem}Let $T$ be a GWT with alphabet $\mathbb{A}$,
weights $\left\{ r_{i}\right\} _{i\in\mathbb{A}}$, and offspring
distribution W. Given $\nu>0$, the following statements hold:
\begin{enumerate}
\item $\mathbb{E}\left(\sum\limits _{i\in W}r_{i}^{\nu}\right)\leq1\,\implies\text{either \ensuremath{\sum\limits _{i\in W}r_{i}^{\nu}}=1 a.s. or }\inf\limits _{section\,\Pi\subset\mathbb{A}^{*}}\sum\limits _{i\in T_{\Pi}}r_{i}^{\nu}=0\text{ a.s.}$ 
\item $\mathbb{E}\left(\sum\limits _{i\in W}r_{i}^{\nu}\right)>1\,\implies\inf\limits _{section\,\Pi\subset\mathbb{A}^{*}}\sum\limits _{i\in T_{\Pi}}r_{i}^{\nu}>0\,\text{a.s.}$
conditioned on nonextinction.
\end{enumerate}
\end{thm}

We note that Falconer's Theorem is in fact more general than stated
above and may be applied in cases were the weights themselves are
random variables (cf. \cite[Theorem 5.35]{Lyons2016}). Also note
that Falconer's theorem is the main ingredient in the proof given
in \cite{Lyons2016} of Theorem \ref{thm:dimension of GW fractals}.
\begin{proof}[Proof of Lemma \ref{lem:size of cutsets}]
 Given $\alpha<\delta$, fix some $\alpha^{\prime}\in\left(\alpha,\delta\right)$.
If the lemma is false, then with positive probability there exists
a decreasing sequence $\rho_{n}\searrow0$ s.t. $\left|T_{\Pi_{\rho_{n}}}\right|\leq\frac{1}{\rho_{n}^{\alpha}}$
for every $n$. In this case, for each $n$, we have 
\[
\sum\limits _{i\in T_{\Pi_{\rho_{n}}}}r_{i}^{\alpha^{\prime}}\leq\frac{1}{\rho_{n}^{\alpha}}\cdot\rho_{n}^{\alpha^{\prime}}=\rho_{n}^{\alpha^{\prime}-\alpha}\underset{n\rightarrow\infty}{\longrightarrow}0
\]
which contradicts (2) of Theorem \ref{thm:Falconer's Theorem} since
$\alpha^{\prime}<\delta$ implies that $\mathbb{E}\left(\sum\limits _{i\in W}r_{i}^{\alpha^{\prime}}\right)>1$.
\end{proof}
\begin{cor}
\label{cor:size of cutsets}Let $T$ and $\delta$ be as in the previous
lemma. Then for every $\alpha<\delta$, $\forall\varepsilon>0,\,\exists\rho_{0}>0$
s.t. $\forall\rho<\rho_{0}$
\[
\mathbb{P}\left(\left|T_{\Pi_{\rho}}\right|>\frac{1}{\rho^{\alpha}}|\text{ nonextinction}\right)>1-\varepsilon.
\]
\end{cor}

\begin{proof}
Denote by $A_{n}$ the event: $\forall\rho<n^{-1},\,\left|T_{\Pi_{\rho}}\right|>\rho^{-\alpha}$.
By Lemma \ref{lem:size of cutsets} 
\[
\mathbb{P}\left(\bigcup\limits _{n=1}^{\infty}A_{n}|\:\text{nonextinction}\right)=1.
\]
Since $\left(A_{n}\right)_{n=1}^{\infty}$ is an increasing sequence
of events, $\mathbb{P}\left(A_{n}|\:\text{nonextinction}\right)\nearrow1$.
Taking $\rho_{0}=n_{0}^{-1}$ with $n_{0}$ large enough we finish
the proof. 
\end{proof}
Note that Corollary \ref{cor:size of cutsets} replaces Corollary
\ref{cor:kesten - stigum} in the present, more flexible setup. Unlike
Corollary \ref{cor:kesten - stigum} which we proved using Kesten
- Stigum theorem and did not need to use the result about the a.s.
dimension of fractal percolation sets, in the setup of arbitrary similarity
IFSs things are more complicated and we did use Theorem \ref{thm:dimension of GW fractals}
in order to bound from below the size of sections.

We conclude this subsection with the following lemma which is a standard
application of the open set condition.%
\begin{lyxgreyedout}
This Lemma is from Mike Hochman's lecture notes (in the following
link: http://math.huji.ac.il/\textasciitilde mhochman/courses/fractals-2012/course-notes.june-26.pdf)
claim 5.17. The proof was not clear to me, so I changed it. What reference
should I provide here?%
\end{lyxgreyedout}

\begin{lem}
\label{lem:OSC use}Let $\Phi=\left\{ \varphi_{i}\right\} _{i\in\Lambda}$
be a similarity IFS with contraction ratios $\left\{ r_{i}\right\} _{i\in\Lambda}$,
which satisfies the open set condition. Let $K$ be the attractor
of $\Phi$. Then there exists some constant $C>0$, s.t. $\forall x\in\mathbb{R}^{d}$,
for any $\rho\in\left(0,r_{min}\right)$, $\left|\left\{ a\in\Pi_{\rho}:\,\varphi_{a}K\cap B_{\rho}\left(x\right)\neq\emptyset\right\} \right|<C$.
\end{lem}

\begin{proof}
Let $U\subset\mathbb{R}^{d}$ be an OSC set for $\Phi$. Denote $\delta=\text{diam}\left(U\right)$
and $\Delta=\text{diam}\left(K\right)$. Note that since $\varphi_{i}\overline{U}\subseteq\overline{U}$
for every $i\in\Lambda$, $K\subseteq\overline{U}$. Hence, since
the ball $B_{\rho}\left(x\right)$ is open, it is enough to show that
$\left|\left\{ a\in\Pi_{\rho}:\,\varphi_{a}U\cap B_{\rho}\left(x\right)\neq\emptyset\right\} \right|<C$
for some constant $C$.

Fix some open ball $B\subseteq U$ of radius $\varepsilon>0$. For
any $a\in\Pi_{\rho}$, $\varphi_{a}B$ is a ball of radius $r_{a}\varepsilon$.
Since $r_{a}\varepsilon>\rho r_{min}\varepsilon$, we have 
\[
\text{Vol}\left(\varphi_{a}U\right)\geq\text{Vol}\left(\varphi_{a}B\right)>C_{1}\rho^{d}
\]
 Where $C_{1}>0$ is some constant which depends only on $\varepsilon,\,r_{min},\,d$.

On the other hand, for any $a\in\Pi_{\rho}$, $\text{diam}\left(\varphi_{a}U\right)=r_{a}\delta\leq\rho\delta$.
Hence, if $\varphi_{a}U\cap B_{\rho}\left(x\right)\neq\emptyset$,
then $\varphi_{a}U\subseteq B_{\rho\left(1+\delta\right)}\left(x\right)$.
Note that $\text{Vol}\left(B_{\rho\left(1+\delta\right)}\left(x\right)\right)=C_{2}\rho^{d}$
where $C_{2}>0$ is some constant which depends only on $\delta$
and $d$.

Since all the sets $\left\{ \varphi_{a}U\right\} _{a\in\Pi_{\rho}}$
are disjoint, we have $\left|\left\{ a\in\Pi_{\rho}:\,\varphi_{a}U\cap B_{\rho}\left(x\right)\neq\emptyset\right\} \right|<\dfrac{C_{2}}{C_{1}}$
.
\end{proof}

\subsection{$*$-trees}

We wish to adjust the idea of $k$-compression of trees to the present
setup, where instead of transforming each $k$ levels into one, we
compress the tree along sections of the form $\Pi_{\rho^{k}}$. The
object we receive by performing this action is no longer a tree as
defined in section \ref{sec:Labeled-GW Trees}, as elements of each
level may be strings of various lengths. Therefore we need to make
the definition of trees a bit more flexible, allowing the building
blocks of the tree to be strings in the alphabet $\mathbb{A}$ instead
of just letters. 

But first, we introduce the following notation: given a subset $S\subseteq\mathbb{A}^{*}$,
we define the function $h_{S}:S\to\mathbb{N}\cup\left\{ 0\right\} $
by 
\[
\forall i\in S,\,h\left(i\right)=\left|\left\{ j\in S:\,j<i\right\} \right|.
\]
The value $h\left(i\right)$ is referred to as the \emph{height of
}$i$. The subscript after $h$ may be omitted whenever the context
is believed to be clear.
\begin{defn}
Let $\mathbb{A}$ be a finite alphabet. A subset $S\subset\mathbb{A}^{*}$
will be called a \emph{$*$-tree }with alphabet $\mathbb{A}$ if the
following conditions hold:
\begin{enumerate}
\item $\emptyset\in S$
\item $\forall\emptyset\neq i\in S$, there exists a unique $j\in S$, s.t.
$j<i$ and $h\left(j\right)=h\left(i\right)-1$
\item $\forall n\in\mathbb{N},\,\left|h^{-1}\left(n\right)\right|<\infty$.
\end{enumerate}
We denote $S_{n}=h^{-1}\left(n\right)$. For each $i\in S$ we denote
by $W_{S}\left(i\right)=\left\{ j\in\mathbb{A}^{*}:\,ij\in S_{h\left(i\right)+1}\right\} $
the set of children of $i$. The boundary of $S$ is defined by $\partial S=\left\{ i\in\mathbb{A}^{\mathbb{N}}:\,\forall n\in\mathbb{N},\,\exists j\in S_{n},\,j<i\right\} $.
Given a $*$-tree $S$ and some vertex $i\in S,$ the descendants
tree of $i$ is defined to be $S^{i}=\left\{ j\in\mathbb{A}^{*}:\,ij\in S\right\} $.

Obviously, every tree $T\in\mathscr{T}_{\mathbb{A}}$ is a $*$-tree
with alphabet $\mathbb{A}$, with $h\left(i\right)=\left|i\right|$
for every $i\in T$. %
\end{defn}

\begin{defn}
Let $S$ be a $*$-tree with alphabet $\mathbb{A}$. A \emph{$*$-subtree
}of \emph{S} is any $*$-tree $Q$ with alphabet $\mathbb{A}$ s.t.
$Q\subseteq S$ and for every $i\in Q,\,\left\{ j\in S:\,j<i\right\} \subset Q$
(this condition ensures that $h_{Q}=h_{S}\restriction_{Q}$).
\end{defn}

We now define the compression of trees along sections.
\begin{defn}
\label{def:compression along cutsets}Let $T$ be a tree with alphabet
$\mathbb{A}$. Let $\left(\Pi_{n}\right)_{n=1}^{\infty}$ be a sequence
of sections s.t. for every $n$, $\forall i\in\Pi_{n+1},\,\exists j\in\Pi_{n},\,\text{s.t. }j<i$.
Then \emph{the compression of $T$ along the sections} $\left(\Pi_{n}\right)_{n=1}^{\infty}$
is defined to be the $*$-tree $S$ with alphabet $\mathbb{A}$ given
by $S=\bigcup_{n=0}^{\infty}T_{\Pi_{n}}$, where we define $\Pi_{0}=\left\{ \emptyset\right\} $.
\end{defn}

Note that $S_{n}=T_{\Pi_{n}}$ for every $n$, and that $\partial S=\partial T$.
Now, suppose that $T$ is a  tree with alphabet $\mathbb{A}$ and
weights $\left\{ r_{i}\right\} _{i\in\mathbb{A}}$, and let $\rho\in\left(0,r_{min}\right)$
be some positive number. The compression of $T$ along the sections
$\left(\Pi_{\rho^{n}}\right)_{n=1}^{\infty}$ will be denoted by $T_{\left(\rho\right)}$.
\begin{defn}
Given an alphabet $\mathbb{A}$ with weights $\left\{ r_{i}\right\} _{i\in\mathbb{A}}$,
$\rho\in\left(0,r_{min}\right)$ and $i\in\bigcup\limits _{n=1}^{\infty}\Pi_{\rho^{n}}$,
we denote $a_{\rho}\left(i\right):=\frac{r_{i}}{\rho^{n_{\rho}\left(i\right)}}$,
where $n_{\rho}\left(i\right)$ is the unique $n\in\mathbb{N}$ s.t.
$i\in\Pi_{\rho^{n}}$. We also denote $n_{\rho}\left(\emptyset\right)=0$
and $a_{\rho}\left(\emptyset\right)=1$.
\end{defn}

\begin{prop}
\label{prop:Next elem in compressed tree}Let $\mathbb{A}$ be an
alphabet with weights $\left\{ r_{i}\right\} _{i\in\mathbb{A}}$,
and let $\rho\in\left(0,r_{min}\right)$. Then for every $i\in\bigcup\limits _{n=1}^{\infty}\Pi_{\rho^{n}}\cup\left\{ \emptyset\right\} $,
$m\in\mathbb{N}$, and $j\in\mathbb{A}^{*}$, 
\[
ij\in\Pi_{\rho^{n_{\rho}\left(i\right)+m}}\iff j\in\Pi_{\frac{\rho^{m}}{a_{\rho}\left(i\right)}}.
\]
\end{prop}

\begin{proof}
For $i=\emptyset$ the claim is trivial. Given $i\in\Pi_{\rho^{n}}$
for some $n\geq1$, for every $j\in\mathbb{A}^{*}$,

\[
\begin{array}{l}
ij\in\Pi_{\rho^{n+m}}\\
\iff r_{i}r_{j}\leq\rho^{n+m}<r_{i}r_{j_{1}}\cdots r_{j_{\left|j\right|-1}}\\
\iff r_{j}\leq\frac{\rho^{m}}{a_{\rho}\left(i\right)}<r_{j_{1}}\cdots r_{j_{\left|j\right|-1}}\\
\iff j\in\Pi_{\frac{\rho}{a_{\rho}\left(i\right)}}
\end{array}
\]
\end{proof}
The following Proposition is an immediate consequence of Proposition
\ref{prop:Next elem in compressed tree}.
\begin{prop}
\label{prop:almost iterative construction}Let $T$ be a GWT with
alphabet $\mathbb{A}$ and weights $\left\{ r_{i}\right\} _{i\in\mathbb{A}}$,
and fix any $\rho\in\left(0,r_{min}\right)$. Then $\forall i\in\bigcup\limits _{n=1}^{\infty}\Pi_{\rho^{n}}\cup\left\{ \emptyset\right\} $,
conditioned on $i\in T_{\left(\rho\right)}$, $W_{T_{\left(\rho\right)}}\left(i\right)$
has the same law as $T\cap\Pi_{\frac{\rho}{a_{\rho}\left(i\right)}}$.%
\end{prop}

Next, we define random $*$-trees. Let $\mathbb{A}$ be a finite alphabet
and let $B\subseteq\mathbb{A}^{*}$ be a $*$-tree. Let $\left\{ M_{x}\right\} _{x\in B}$
be a collection of independent random variables s.t. for every $x\in B$,
$M_{x}$ takes values in the finite set $2^{W_{B}\left(x\right)}$.
Define 
\begin{itemize}
\item $S_{0}=\emptyset$ 
\item For $n\geq1,$ $S_{n}=\bigcup\limits _{i\in S_{n-1}}\left\{ ij:\,j\in M_{i}\right\} ,$
\end{itemize}
and finally take $S=\bigcup_{n=0}^{\infty}S_{n}$. \emph{S} is then
a random $*$-tree on $B$ with offspring distributions $\left\{ M_{x}\right\} _{x\in B}$.
Note that every realization of $S$ is a \emph{$*$-}subtree of $B$%
. In the graph theoretic perspective, this process may be thought
of in the following way: Let $G$ be the graph with vertices $B$
and edges $\left\{ \left(i,j\right):\,j\in W_{B}\left(i\right)\right\} $.
The graph $G$ is a directed rooted tree in the graph theoretic sense
with $\emptyset$ serving as the root. Now, given a realization of
the random variables $\left\{ M_{x}\right\} _{x\in B}$, we take $S$
to be the connected component of the subgraph of $G$, with vertices
$B$ and edges $\left\{ \left(i,j\right):\,j\in M_{i}\right\} $ which
contains $\emptyset$.

There may be elements $i\in B$ s.t. $\mathbb{P}\left(i\in S\right)=0$.
These elements add no information to the construction. Therefore,
given the setup above we denote $B^{\prime}=\left\{ i\in B:\,\mathbb{P}\left(i\in S\right)>0\right\} $.
Note that by the construction of $S$, $B^{\prime}$ is a \emph{$*$-}subtree
of $B$.

Let $T$ be a GWT with alphabet $\mathbb{A}$, and let $\left(\Pi_{n}\right)_{n=1}^{\infty}$
be a sequence of sections as in Definition \ref{def:compression along cutsets}.
The compression of $T$ along the sections $\left(\Pi_{n}\right)_{n=1}^{\infty}$
has the law of a random $*$-tree on $B=\bigcup_{n=0}^{\infty}\Pi_{n}$,
where for each $i\in B$, $M_{i}$ has the law of $T\cap W_{B}\left(i\right)$.
In particular, the following is an immediate consequence of Proposition
\ref{prop:almost iterative construction}.
\begin{prop}
\label{prop:law of compressed tree}Let $T$ be a GWT with alphabet
$\mathbb{A}$, weights $\left\{ r_{i}\right\} _{i\in\mathbb{A}}$
and offspring distribution $W$. Then $\forall\rho\in\left(0,r_{min}\right)$,
the compressed tree $T_{\left(\rho\right)}$ has the law of a random
$*$-tree on $B=\bigcup_{n=1}^{\infty}\Pi_{\rho^{n}}\cup\left\{ \emptyset\right\} $,
with offspring distributions $M_{i}\sim T_{\Pi_{\frac{\rho}{a_{\rho}\left(i\right)}}}$,
$\forall i\in B$. 
\end{prop}

Note that since we assume that $\forall i\in\mathbb{A},\,\mathbb{P}\left(i\in W\right)>0$,
we have $B^{\prime}=B$. It is important to notice that for every
$i\in T_{\left(\rho\right)}$, $r_{min}<a_{\rho}\left(i\right)\leq1$.
This means that although $T_{\left(\rho\right)}$ need not have the
structure of a GWT where all the offspring distributions are the same,
the offspring distributions of $T_{\left(\rho\right)}$ can not vary
too much. 
\begin{defn}
In the setup of random {*}-trees as described above, the offspring
distributions $\left\{ M_{x}\right\} _{x\in B}$ will be called \emph{bounded}
if there exists some constant $C>0$ s.t. $\forall x\in B^{\prime},\,\mathbb{P}\left(\left|M_{x}\right|<C\right)=1$.
\end{defn}

The following definition should be compared with definition \ref{def:A-tree}.
\begin{defn}
Let $\mathscr{A}$ be a mapping $\mathscr{A}:\mathbb{A}^{*}\rightarrow2^{\left(2^{\mathbb{A}^{*}}\right)}$
with the notation $\mathscr{A}\left(i\right)=\mathscr{A}_{i}$ so
that $\forall i\in\mathbb{A}^{*},\,\emptyset\neq\mathscr{A}_{i}\subseteq2^{\mathbb{A}^{*}}$.
A $*$-tree $S$ is called an $\mathscr{A}$\emph{-$*$-tree} if $\forall i\in S,\,W_{S}\left(i\right)\in\mathscr{A}_{i}$.
$S$ will be called an $\mathscr{A}$\emph{-$*$-tree of level $n$}
if for every element $i\in S$ of height $<n$, $W_{S}\left(i\right)\in\mathscr{A}_{i}$.
\end{defn}

Given a mapping $\mathscr{A}$ as above, and any $x\in\mathbb{A}^{*}$,
we denote by $\mathscr{A}^{x}$ the mapping $\mathscr{A}^{x}:\mathbb{A}^{*}\to2^{\left(2^{\mathbb{A}^{*}}\right)}$
given by $\mathscr{A}^{x}\left(i\right)=\mathscr{A}_{xi}$.

The following Lemma is a generalization of Lemma \ref{lem:A-subtree of every length implies infinite A-subtree}
to $*$-trees.
\begin{lem}
\label{lem:A-*-subtrees of every length implies A-*-subtrees}Let
$S$ be a $*$-tree with alphabet $\mathbb{A}$, and let $\mathscr{A}:\mathbb{A}^{*}\rightarrow2^{\left(2^{\mathbb{A}^{*}}\right)}$
be as above. Then $S$ has an infinite $\mathscr{A}$-$*$-subtree$\iff$
$\forall n\in\mathbb{N}$, $T$ has an $\mathscr{A}$-$*$-subtree
of level $n$.%
\end{lem}

Proving the lemma only requires some minor adjustments of the proof
of Lemma \ref{lem:A-subtree of every length implies infinite A-subtree}
to $*$-trees. The details are left to the reader.

Next, we prove a version of Theorem \ref{thm:fixed point} for $*$-trees.
\begin{thm}
\label{thm:fixed point for GRLT}Let S be a random $*$-tree on the
$*$-tree $B\subseteq\mathbb{A}^{*}$ with bounded offspring distributions
$\left\{ M_{x}\right\} _{x\in B}$. Let $\mathscr{A}:B^{\prime}\rightarrow2^{\left(2^{\mathbb{A}^{*}}\right)}$
be s.t. $\forall x\in B^{\prime},\,\mathscr{A}_{x}\subseteq2^{W_{B^{\prime}}\left(x\right)}$
is monotonic. Define $g_{\mathscr{A}}:\left[0,1\right]\to\left[0,1\right]$
by $g_{\mathscr{A}}\left(s\right)=\sup\limits _{x\in B^{\prime}}\mathbb{P}\left(M_{x}^{\left(s\right)}\notin\mathscr{A}_{x}\right)$.
Let $s_{0}$ be the smallest fixed point of $g_{\mathscr{A}}$ in
$\left[0,1\right]$. Then 
\begin{equation}
\sup_{x\in B^{\prime}}\mathbb{P}\left(S^{x}\text{ has no \ensuremath{\mathscr{A}^{x}}-\ensuremath{*}-subtree}|\,x\in S\right)\leq s_{0}.\label{eq:fixed point for *-trees}
\end{equation}
\end{thm}

\begin{proof}
Since the collection $\left\{ M_{x}\right\} _{x\in B^{\prime}}$ is
bounded, the collection of functions $\left\{ \mathbb{P}\left(M_{x}^{\left(s\right)}\notin\mathscr{A}_{x}\right)\right\} _{x\in B^{\prime}}$
is equicontinuous %
and therefore $g_{\mathscr{A}}$ is continuous%
. Also, monotonicity of $\mathscr{A}_{x}$ for every $x\in B^{\prime}$
implies that $g_{\mathscr{A}}$ is monotonically increasing%
. These 2 properties of $g_{\mathscr{A}}$ imply that $\lim_{n\rightarrow\infty}g_{\mathscr{A}}^{n}\left(0\right)$
is the smallest fixed point of $g_{\mathscr{\mathscr{A}}}$, where
$g_{\mathscr{A}}^{n}$ denotes the composition of $g_{\mathscr{A}}$
with itself $n$ times. So $s_{0}=\lim_{n\rightarrow\infty}g_{\mathscr{A}}^{n}\left(0\right)$.
Define 
\[
q_{n}=\sup_{x\in B^{\prime}}\mathbb{P}\left(S^{x}\text{ has no \ensuremath{\mathscr{A}^{x}}-\ensuremath{*}-subtree of level \emph{n}}|\,x\in S\right).
\]
\begin{claim*}
$\forall n\geq1,\,g_{\mathscr{A}}^{n}\left(0\right)\geq q_{n}$.
\end{claim*}
\begin{proof}[Proof of claim]
First, notice that $g_{\mathscr{A}}\left(0\right)=q_{1}$. Now, assume
the claim is true for $n$. So $g_{\mathscr{A}}^{n}\left(0\right)\geq q_{n}$,
which implies by monotonicity of $g_{\mathscr{A}}$ that $g_{\mathscr{A}}^{n+1}\left(0\right)\geq g_{\mathscr{A}}\left(q_{n}\right)=\sup\limits _{x\in B^{\prime}}\mathbb{P}\left(M_{x}^{\left(q_{n}\right)}\notin\mathscr{A}_{x}\right)$.
For every $x\in B^{\prime}$,
\[
\begin{array}{l}
\mathbb{P}\left(M_{x}^{\left(q_{n}\right)}\notin\mathscr{A}_{x}\right)\geq\\
\mathbb{P}\left(\left\{ a\in W_{S}\left(x\right):\,\text{\ensuremath{S^{xa}} has an \ensuremath{\mathscr{A}^{xa}}-\ensuremath{*}-subtree of level \emph{n}}\right\} \notin\mathscr{A}_{x}|\,x\in S\right)=\\
\mathbb{P}\left(S^{x}\text{ has no \ensuremath{\mathscr{A}^{x}}-\ensuremath{*}-subtree of level \emph{n+1}}|\,x\in S\right)
\end{array}
\]
(the inequality uses the fact that $q_{n}$ is defined as a supremum
and the monotonicity of $\mathscr{A}_{x}$). By taking supremums we
obtain $g_{\mathscr{A}}\left(q_{n}\right)\geq q_{n+1}$ which finishes
the proof of the claim.
\end{proof}
The sequence $\left(q_{n}\right)$ is monotonically increasing and
bounded by $s_{0}$, so $q=\lim_{n\rightarrow\infty}q_{n}\leq s_{0}$.
We now need the following elementary lemma whose proof is left to
the reader:
\begin{lem*}
Let $f_{n}:A\rightarrow\mathbb{R}$ be a sequence of functions s.t.
$\forall a\in A$, the sequence $f_{n}\left(a\right)$ is monotonically
increasing, and the functions $f_{n}$ are uniformly bounded. Then
\[
\lim\limits _{n\rightarrow\infty}\sup\limits _{a\in A}f_{n}\left(a\right)=\sup\limits _{a\in A}\lim\limits _{n\rightarrow\infty}f_{n}\left(a\right).
\]
\end{lem*}
Combining the above lemma with Lemma \ref{lem:A-*-subtrees of every length implies A-*-subtrees},
we obtain that %
\[
q=\sup\limits _{x\in B^{\prime}}\mathbb{P}\left(\text{\ensuremath{S^{x}} has no \ensuremath{\mathscr{A}^{x}}-\ensuremath{*}-subtree}|\,x\in S\right)
\]
{} which concludes the proof of the theorem.
\end{proof}
\begin{rem}
Equality in equation (\ref{eq:fixed point for *-trees}) need not
hold. See the Appendix for a counter-example.
\end{rem}

\begin{rem}
Although Theorems \ref{thm:fixed point for GRLT}, \ref{thm:fixed point}
are stated using symbolic spaces, these theorems are purely graph
theoretic and discuss random subtrees of a given locally finite rooted
tree with uniformly bounded offspring distributions.

\end{rem}

\subsection{Proof of main theorem}

\subsubsection{Ahlfors-regularity}
\begin{lem}
\label{lem: rho^alpha subtrees exist}Let T be a supercritical GWT
with alphabet $\mathbb{A}$, weights $\left\{ r_{i}\right\} _{i\in\mathbb{A}}$,
and offspring distribution W, and let $\delta$ satisfy $\mathbb{E}\left(\sum\limits _{i\in W}r_{i}^{\delta}\right)=1$.
Then $\forall\alpha\in\left(0,\delta\right),\,\forall s\in\left(0,1\right),$
\[
\sup_{x\in B_{\rho}}\mathbb{P}\left(\left|\left(T_{\Pi_{\frac{\rho}{a_{\rho}\left(x\right)}}}\right)^{\left(s\right)}\right|<\rho^{-\alpha}|\text{ nonextinction}\right)\underset{\rho\rightarrow0}{\longrightarrow}0
\]
where $B_{\rho}=\left\{ \emptyset\right\} \cup\left(\bigcup_{n=1}^{\infty}\Pi_{\rho^{n}}\right)$.
\end{lem}

\begin{proof}
First, we recall that $a_{\rho}\left(x\right)\in\left[r_{min},1\right]$
for every $x\in B_{\rho}$. Since for any positive constant $C$,
$\mathbb{P}\left(\left|T_{\Pi_{\varepsilon}}\right|<C|\text{ nonextinction}\right)$
decreases as $\varepsilon$ decreases, we have for any $x\in B_{\rho}$,
\[
\mathbb{P}\left(\left|\left(T_{\Pi_{\frac{\rho}{a_{\rho}\left(x\right)}}}\right)\right|<\frac{1}{\rho^{\alpha}}|\text{ nonextinction}\right)\leq\mathbb{P}\left(\left|T_{\Pi_{\frac{\rho}{r_{min}}}}\right|<\frac{1}{\rho^{\alpha}}|\text{ nonextinction}\right).
\]
Hence, it is enough to show that $\mathbb{P}\left(\left|\left(T_{\Pi_{\frac{\rho}{r_{min}}}}\right)^{\left(s\right)}\right|<\frac{1}{\rho^{\alpha}}|\text{ nonextinction}\right)\underset{\rho\rightarrow0}{\longrightarrow}0$.

Fix some $\beta\in\left(\alpha,\delta\right)$. Given any $\varepsilon>0$,
by Corollary \ref{cor:size of cutsets} 
\[
\mathbb{P}\left(\left|T_{\Pi_{\frac{\rho}{r_{min}}}}\right|>\left(\frac{r_{min}}{\rho}\right)^{\beta}|\text{ nonextinction}\right)>1-\varepsilon
\]
whenever $\rho$ is small enough. Given some $s\in\left(0,1\right)$,
for a small enough $\rho$, $\left(\frac{r_{min}}{\rho}\right)^{\beta}>\frac{2}{1-s}\cdot\frac{1}{\rho^{\alpha}}$,
so that
\[
\mathbb{P}\left(\left|T_{\Pi_{\frac{\rho}{r_{min}}}}\right|>\frac{2}{\left(1-s\right)\rho^{\alpha}}|\text{ nonextinction}\right)>1-\varepsilon.
\]
By Chebyshev's inequality
\[
\mathbb{P}\left(\text{Bin}\left(\frac{2}{\left(1-s\right)\rho^{\alpha}},1-s\right)<\frac{1}{\rho^{\alpha}}\right)\leq\frac{\frac{2}{\left(1-s\right)\rho^{\alpha}}s\left(1-s\right)}{\left(\frac{2}{\left(1-s\right)\rho^{\alpha}}\left(1-s\right)-\frac{1}{\rho^{\alpha}}\right)^{2}}=2s\rho^{\alpha}\underset{\rho\rightarrow0}{\longrightarrow}0.
\]
Therefore, $\mathbb{P}\left(\left|\left(T_{\Pi_{\frac{\rho}{r_{min}}}}\right)^{\left(s\right)}\right|<\frac{1}{\rho^{\alpha}}|\text{ nonextinction}\right)\underset{\rho\rightarrow0}{\longrightarrow}0$
as required.
\end{proof}
\begin{lem}
\label{lem:rho^alpha-ary implies alpha Ahlfors regular}Let $\Phi=$$\left\{ \varphi_{i}\right\} _{i\in\Lambda}$
be a similarity IFS satisfying the OSC, with contraction ratios $\left\{ r_{i}\right\} _{i\in\Lambda}$,
and let $T\in\mathscr{T}_{\Lambda}$ be an infinite tree with alphabet
$\Lambda$. Let $\rho\in\left(0,r_{min}\right)$ and $\alpha>0$ be
s.t. $\rho^{-\alpha}$ is an integer, and let S be a $\rho^{-\alpha}$-ary
\emph{$*$-}subtree of $\left(T_{\left(\rho\right)}\right)^{\omega}$
for some $\omega\in T_{\left(\rho\right)}$. Then $E=\gamma_{\Phi}\left(\partial S\right)$
is $\alpha$-Ahlfors regular.
\end{lem}

\begin{proof}
Construct a probability measure $\mu$ on $\Lambda^{\mathbb{N}}$
by equally distributing mass at each level of the tree, i.e., $\forall i\in S,\,\mu\left(\left[i\right]\right)=\rho^{h_{S}\left(i\right)\alpha}$.
Let $\nu$ be the projection of $\mu$ to $\mathbb{R}^{d}$, i.e.,
$\nu=\left(\gamma_{\Phi}\right)_{*}\mu$. Obviously $\text{supp}\left(\nu\right)=E$.
We show that $\nu$ is $\alpha$-Ahlfors regular.

Fix any $r>0$ and $x\in E$. Let $n$ be the unique integer s.t.
$\frac{\rho^{n+1}}{a_{\rho}\left(\omega\right)}<r\leq\frac{\rho^{n}}{a_{\rho}\left(\omega\right)}$.
By Lemma \ref{lem:OSC use} and Proposition \ref{prop:Next elem in compressed tree},
the ball $B_{\frac{\rho^{n}}{a_{\rho}\left(\omega\right)}}\left(x\right)$
intersects at most $C$ of the sets $\left\{ \varphi_{i}\left(K\right)\right\} _{i\in S_{n}}$,
where $K$ is the attractor of $\Phi$ and $C>0$ is some constant
not depending on $x$ and $n$. Therefore, 
\[
\nu\left(B_{r}\left(x\right)\right)\leq\sum_{i\in S_{n},\,\varphi_{i}\left(K\right)\cap B_{r}\left(x\right)\neq\emptyset}\mu\left(\left[i\right]\right)\leq C\cdot\rho^{n\alpha}\leq C\rho^{-\alpha}\cdot r^{\alpha}
\]

On the other hand, let $j\in S_{m+1}$ be s.t. $x\in\varphi_{j}\left(K\right)$
and $\frac{\rho^{m+1}}{a_{\rho}\left(\omega\right)}<\frac{r}{\Delta}\leq\frac{\rho^{m}}{a_{\rho}\left(\omega\right)}$,
where $\Delta=\text{diam}\left(K\right)$. Then $\text{diam}\left(\varphi_{j}\left(K\right)\right)=r_{j}\Delta\leq\frac{\rho^{m+1}}{a_{\rho}\left(\omega\right)}\Delta<r$,
and therefore $\varphi_{j}\left(K\right)\subseteq B_{r}\left(x\right)$.
Hence 
\[
\nu\left(B_{r}\left(x\right)\right)\geq\nu\left(\varphi_{j}\left(K\right)\right)=\rho^{\left(m+1\right)\alpha}\geq\left(\frac{\rho a_{\rho}\left(\omega\right)}{\Delta}\right)^{\alpha}\cdot r^{\alpha}
\]
\end{proof}

\subsubsection{Diffuseness}
\begin{prop}
\label{prop: diffuse compressed tree implies diffuse limit set}Let
$\Phi=\left\{ \varphi_{i}\right\} _{i\in\Lambda}$ be a similarity
IFS with an attractor $K$ and contraction ratios $\left\{ r_{i}\right\} _{i\in\Lambda}$.
Let $T\subset\Lambda^{*}$ be an infinite tree. Let $T_{\left(\rho\right)}$
be the compression of $T$ for some $\rho\in\left(0,r_{min}\right)$.
Fix some $\omega\in T_{\left(\rho\right)}$. Assume that there exists
some $c>0$ s.t. $\forall i\in\left(T_{\left(\rho\right)}\right)^{\omega}$,
$W_{T_{\left(\rho\right)}}\left(i\right)$ is $\left(K,c\right)$-diffuse.
Then the limit set $E=\gamma_{\Phi}\left(\partial T^{\omega}\right)$
is $\frac{\rho cr_{min}}{\Delta}$-diffuse, where $\Delta=\text{diam}\left(K\right)$.
\end{prop}

\begin{proof}
Fix $\xi\in\left(0,\dfrac{\rho\Delta}{a_{\rho}\left(\omega\right)}\right)$,
$x\in E$ and an affine hyperplane $\mathcal{L}\subset\mathbb{R}^{d}$
. Let $n\in\mathbb{N}$ be the unique integer s.t. $\rho^{n}\leq\frac{\xi a_{\rho}\left(\omega\right)}{\Delta}<\rho^{n-1}$,
and let $i\in T^{\omega}\cap\Pi_{\frac{\rho^{n}}{a_{\rho}\left(\omega\right)}}$
be s.t. $x\in\varphi_{i}K$ (by Proposition \ref{prop:Next elem in compressed tree}
$i\in\left(\text{\ensuremath{T_{\left(\rho\right)}}}\right)^{\omega}$).
Note that 
\[
\frac{r_{min}\rho}{\Delta}\xi<\dfrac{r_{min}\rho^{n}}{a_{\rho}\left(\omega\right)}<r_{i}\leq\dfrac{\rho^{n}}{a_{\rho}\left(\omega\right)}\leq\frac{\xi}{\Delta}.
\]
By assumption $\exists j\in W_{T_{\left(\rho\right)}}\left(i\right)$
s.t. $\varphi_{j}K\cap\left(\varphi_{i}^{-1}\mathcal{L}\right)^{\left(c\right)}=\emptyset$.
Applying $\varphi_{i}$ we get that $\varphi_{ij}K\cap\mathcal{L}^{\left(r_{i}c\right)}=\emptyset$,
hence $\varphi_{ij}K\cap\mathcal{L}^{\left(\frac{r_{min}\rho c}{\Delta}\xi\right)}=\emptyset$.
Notice that $\text{diam}\left(\varphi_{i}K\right)=r_{i}\Delta\leq\xi$,
therefore $\varphi_{ij}K\subset\varphi_{i}K\subseteq B_{\xi}\left(x\right)$.
By assumption $\forall v\in T_{\left(\rho\right)}$, $W_{T_{\left(\rho\right)}}\left(v\right)$
is $\left(K,c\right)$-diffuse and in particular non-empty, so every
descendants tree of $T_{\left(\rho\right)}$ is infinite. Hence, $\varphi_{ij}K\cap E\neq\emptyset$
and therefore $B_{\xi}\left(x\right)\cap E\setminus\mathcal{L}^{\left(\frac{r_{min}\rho c}{\Delta}\xi\right)}\neq\emptyset$.
\end{proof}
\begin{lem}
\label{lem:sections are diffuse}Let $T$ be a GWT corresponding to
a similarity IFS $\Phi=\left\{ \varphi_{i}\right\} _{i\in\Lambda}$
with offspring distribution $W$. Assume that $\mathbb{P}\left(\text{\ensuremath{W} is \ensuremath{\left(F,c\right)}-diffuse}\right)=\nu>0$
for some $c>0$ and $F$ as in Definition \ref{def:diffuseness for IFS}.
Then for every section $\Pi\subset\Lambda^{*}$, $\mathbb{P}\left(\text{\ensuremath{T_{\Pi}} is \ensuremath{\left(F,c\right)}-diffuse}\right)\geq\nu\cdot\mathbb{P}\left(T_{n}\neq\emptyset\right)^{\left|\Lambda\right|}$
where $n=\max\left\{ \left|i\right|:\,i\in\Pi\right\} $ is the maximal
depth of the section $\Pi$.
\end{lem}

\begin{proof}
First we note that if $A\subseteq\Lambda$ is $\left(F,c\right)$-diffuse,
then for every finite set $P\subseteq\Lambda^{*}$ s.t. $\forall i\in A,\,\exists j\in P$,
s.t. $j\geq i$, $P$ is also $\left(F,c\right)$-diffuse. Indeed,
given an affine hyperplane $\mathcal{L}\subset\mathbb{R}^{d}$, $\varphi_{i}F\cap\mathcal{L}^{\left(c\right)}=\emptyset$
for some $i\in A$. For some $j\in P$, $j\geq i$ and therefore $\varphi_{j}F\subseteq\varphi_{i}F$
which implies $\varphi_{j}F\cap\mathcal{L}^{\left(c\right)}=\emptyset$.
Hence, given some section $\Pi\subset\Lambda^{*}$, if $T_{1}$ is
$\left(F,c\right)$ - diffuse and $\forall i\in T_{1},\,\exists j\in T_{\Pi}$
s.t. $j\geq i$, then $T_{\Pi}$ is also $\left(F,c\right)$ - diffuse.
Since for every $i\in\Lambda$, $\mathbb{P}\left(\exists j\in T_{\Pi}\text{ s.t. }j\geq i|\,i\in T_{1}\right)\geq\mathbb{P}\left(T_{n}\neq\emptyset\right)$
and these events are independent for different elements of $\Lambda$
and also independent of the event: $T_{1}$ is $\left(F,c\right)$
- diffuse, the claim follows. 
\end{proof}
For the next lemma we need the following notation. Let $\Phi=\left\{ \varphi_{i}\right\} _{i\in\Lambda}$
be a similarity IFS with contraction ratios $\left\{ r_{i}\right\} _{i\in\Lambda}$
and attractor $K$. Given $\rho\in\left(0,r_{min}\right)$, denote
$B_{\rho}=\left\{ \emptyset\right\} \cup\left(\bigcup_{n=1}^{\infty}\Pi_{\rho^{n}}\right)$.
For every $c>0$ and $x\in B_{\rho}$, denote 
\[
^{\rho,c}\mathscr{D}_{x}=\left\{ A\subseteq\Pi_{\frac{\rho}{a_{\rho}\left(x\right)}}:\,\exists i\in\Pi_{\frac{\rho}{a_{\rho}\left(x\right)r_{min}}},\,\left\{ \varphi_{v}:\,iv\in A\right\} \text{ is \emph{\ensuremath{\left(K,c\right)}}-diffuse}\right\} .
\]
This notation (and the following lemma) should be compared with the
notation $\mathscr{D}_{b,k}$ which was defined before Lemma \ref{lem:diffuse subtree}.
Note that every $A\in{}^{\rho,c}\mathscr{D}_{x}$ is $\left(K,\rho c\right)$-diffuse.
We denote $^{\rho,c}\mathscr{D}:B_{\rho}\to2^{\left(2^{\Lambda^{*}}\right)}$
where $^{\rho,c}\mathscr{D}\left(x\right)={}^{\rho,c}\mathscr{D}_{x}$%

\begin{lem}
\label{lem: diffuse subtrees exist}Let $T$ be a GWT corresponding
to a similarity IFS $\Phi=\left\{ \varphi_{i}\right\} _{i\in\Lambda}$
with contraction ratios $\left\{ r_{i}\right\} _{i\in\Lambda}$ and
offspring distribution $W$. Denote by $K$ the attractor of $\Phi$.
Assume that $\exists c>0$ s.t. $\mathbb{P}\left(\text{\ensuremath{\left\{  \varphi_{i}\right\}  _{i\in W}} is \emph{\ensuremath{\left(K,c\right)}}-diffuse}\right)=\nu>0$.
Then for every $s\in\left(0,1\right)$,
\[
\sup_{x\in B_{\rho}}\mathbb{P}\left(\left(T_{\Pi_{\frac{\rho}{a_{\rho}\left(x\right)}}}\right)^{\left(s\right)}\notin{}^{\rho,c}\mathscr{D}_{x}|\text{ nonextinction}\right)\underset{\rho\rightarrow0}{\longrightarrow}0
\]
where $B_{\rho}=\left\{ \emptyset\right\} \cup\left(\bigcup_{n=1}^{\infty}\Pi_{\rho^{n}}\right)$.
\end{lem}

\begin{proof}
Recall that $a_{\rho}\left(x\right)\in\left[r_{min},1\right]$ for
every $x\in B_{\rho}$. Fix some $\alpha\in\left(0,\delta\right)$
where $\delta$ satisfies the equation $\mathbb{E}\left(\sum\limits _{i\in\Lambda}r_{i}^{\delta}\right)=1$.
Then by Corollary \ref{cor:size of cutsets}, given $\varepsilon>0$
there exists some $\rho_{0}>0$ s.t. whenever $\rho<\rho_{0}$, 
\[
\mathbb{P}\left(\left|T\cap\Pi_{\frac{\rho}{\left(r_{min}\right)^{2}}}\right|>\left(r_{min}\right)^{2\alpha}\cdot\frac{1}{\rho^{\alpha}}\right)>1-\varepsilon.
\]
Fix some $x\in B_{\rho}.$ Since $a_{\rho}\left(x\right)\geq r_{min}$,
for every $\rho<\rho_{0}$,
\begin{equation}
\mathbb{P}\left(\left|T\cap\Pi_{\frac{\rho}{a_{\rho}\left(x\right)\cdot r_{min}}}\right|>\frac{1}{\rho^{\alpha}}\right)>1-\varepsilon\label{eq: lower bound on sections}
\end{equation}
where the constant $\left(r_{min}\right)^{2\alpha}$ was removed as
it may be absorbed by taking a slightly smaller $\alpha$ and assuming
that $\rho_{0}$ is small enough.

For every $i\in\Pi_{\frac{\rho}{a_{\rho}\left(x\right)\cdot r_{min}}}$,
\[
\mathbb{P}\left(\left\{ \varphi_{j}\right\} _{j\in W_{T}\left(i\right)}\text{ is \emph{\ensuremath{\left(K,c\right)}}-diffuse}|\,i\in T\right)=\nu>0.
\]
Denote $V_{i}=\left\{ j\in\Lambda^{*}:\,ij\in\Pi_{\frac{\rho}{a_{\rho}\left(x\right)}}\right\} $.
$V_{i}$ is a section, and $\max\limits _{j\in V_{i}}\left|j\right|\leq\left\lceil \log_{r_{max}}\left(r_{min}\right)^{2}\right\rceil \eqqcolon n_{0}$.
Note that $n_{0}$ is independent of $i$ and of $x$%
. By Lemma \ref{lem:sections are diffuse}, 
\[
\mathbb{P}\left(\text{\ensuremath{\left\{ \varphi_{j}\right\} _{j\in T_{V_{i}}}} is \emph{\ensuremath{\left(K,c\right)}}-diffuse}\right)\geq\nu\cdot\mathbb{P}\left(T_{n_{0}}\neq\emptyset\right)^{\left|\Lambda\right|},
\]
and given $s\in\left(0,1\right)$, 
\begin{equation}
\mathbb{P}\left(\text{\ensuremath{\left\{ \varphi_{j}\right\} _{j\in T_{V_{i}}^{\left(s\right)}}} is \emph{\ensuremath{\left(K,c\right)}}-diffuse}\right)\geq\nu\cdot\mathbb{P}\left(T_{n_{0}}\neq\emptyset\right)^{\left|\Lambda\right|}\cdot\left(1-s\right)^{\left(\left|\Lambda\right|^{n_{0}}\right)}.\label{eq:Prob for c-diffuse}
\end{equation}
We shall denote the right hand side of inequality (\ref{eq:Prob for c-diffuse})
by $\nu^{\prime}$. Notice that conditioned on $i\in T$, $\left\{ v\in\Lambda^{*}:\,iv\in T\cap\Pi_{\frac{\rho}{a_{x}}}\right\} \sim T_{V_{i}}$.
Therefore,
\[
\forall i\in\Pi_{\frac{\rho}{a_{\rho}\left(x\right)\cdot r_{min}}},\,\mathbb{P}\left(\left\{ v\in\Lambda^{*}:\,iv\in T\cap\Pi_{\frac{\rho}{a_{\rho}\left(x\right)}}\right\} \,\text{is \emph{\ensuremath{\left(K,c\right)}}-diffuse}|\,i\in T\right)\geq\nu^{\prime}.
\]

Using inequality (\ref{eq: lower bound on sections}) we conclude
the proof of the lemma.
\end{proof}

\subsubsection{Final step}
\begin{proof}[Proof of Theorem \ref{thm:Main theorem diffuse subses} ]
Let $E$ be a non-planar GWF and $T$ the corresponding GWT, w.r.t.
a similarity IFS $\Phi=\left\{ \varphi_{i}\right\} _{i\in\Lambda}$
whose attractor is denoted by $K$, and offspring distribution $W$.
First, note that by Proposition \ref{prop: equivalent conditions for not in a hyperplane},
there exist a section $\Pi\subseteq\Lambda^{*}$, and $c>0$ s.t.
$\mathbb{P}\left(T_{\Pi}\text{ is \emph{\ensuremath{\left(K,c\right)}}-diffuse}\right)>0$.
Since we may consider the IFS $\left\{ \varphi_{i}\right\} _{i\in\Pi}$
which has the same attractor as $\Phi$, and the GWF corresponding
to the offspring distribution $\sim T_{\Pi}$ which has the same law
as $E$, there is no loss of generality in assuming that $\mathbb{P}\left(W\text{ is \emph{\ensuremath{\left(K,c\right)}}-diffuse}\right)>0$,
hence we proceed assuming the latter holds. 

Given $\rho>0$, consider the compressed $*$-tree $T_{\left(\rho\right)}$.
Recall that by Proposition \ref{prop:law of compressed tree}, $T_{\left(\rho\right)}$
has the law of a random $*$-tree on $B_{\rho}=\bigcup\limits _{n=1}^{\infty}\Pi_{\rho^{n}}\cup\left\{ \emptyset\right\} $
with offspring distributions $M_{i}\sim T_{\Pi_{\frac{\rho}{a_{\rho}\left(i\right)}}}$,
for every $i\in B_{\rho}$. By Lemma \ref{lem: diffuse subtrees exist},
for every $s\in\left(0,1\right)$,
\[
\sup_{x\in B_{\rho}}\mathbb{P}\left(\left(T_{\Pi_{\frac{\rho}{a_{\rho}\left(x\right)}}}\right)^{\left(s\right)}\notin{}^{\rho,c}\mathscr{D}_{x}|\,\text{nonextinction}\right)\underset{\rho\rightarrow0}{\longrightarrow}0.
\]
Fix any $\alpha\in\left(0,\delta\right)$. By Lemma \ref{lem: rho^alpha subtrees exist},
\[
\sup_{x\in B_{\rho}}\mathbb{P}\left(\left|\left(T_{\Pi_{\frac{\rho}{a_{\rho}\left(x\right)}}}\right)^{\left(s\right)}\right|<\rho^{-\alpha}|\text{ nonextinction}\right)\underset{\rho\rightarrow0}{\longrightarrow}0.
\]
Denoting for every $x\in B_{\rho}$, $^{\rho,\alpha}\mathscr{C}_{x}=\left\{ A\subseteq\Pi_{\frac{\rho}{a_{\rho}\left(x\right)}}:\,\left|A\right|\geq\rho^{-\alpha}\right\} $,
we have 
\[
\sup_{x\in B_{\rho}}\mathbb{P}\left(\left(T_{\Pi_{\frac{\rho}{a_{\rho}\left(x\right)}}}\right)^{\left(s\right)}\notin{}^{\rho,c}\mathscr{D}_{x}\cap{}^{\rho,\alpha}\mathscr{C}_{x}|\text{ nonextinction}\right)\underset{\rho\rightarrow0}{\longrightarrow}0
\]
 which implies that 
\[
\sup_{x\in B_{\rho}}\mathbb{P}\left(\left(T_{\Pi_{\frac{\rho}{a_{\rho}\left(x\right)}}}\right)^{\left(s\right)}\notin{}^{\rho,c}\mathscr{D}_{x}\cap{}^{\rho,\alpha}\mathscr{C}_{x}\right)\underset{\rho\rightarrow0}{\longrightarrow}\mathbb{P}\left(\text{extinction}\right).
\]

Therefore, we have shown that for every $s\in\left(0,1\right)$, 
\[
\sup_{x\in B_{\rho}}\mathbb{P}\left(M_{x}^{\left(s\right)}\notin{}^{\rho,c}\mathscr{D}_{x}\cap{}^{\rho,\alpha}\mathscr{C}_{x}\right)\underset{\rho\rightarrow0}{\longrightarrow}\mathbb{P}\left(\text{extinction}\right).
\]
This implies that as a function of $s$, $\sup_{x\in B_{\rho}}\mathbb{P}\left(M_{x}^{\left(s\right)}\notin{}^{\rho,c}\mathscr{D}_{x}\cap{}^{\rho,\alpha}\mathscr{C}_{x}\right)$
has a fixed point <1 whenever $\rho$ is small enough. Denote $^{\rho,c}\mathscr{A}:B_{\rho}\to2^{\left(2^{\Lambda^{*}}\right)}$
given by $^{\rho,c}\mathscr{A}\left(x\right)={}^{\rho,c}\mathscr{D}_{x}\cap{}^{\rho,\alpha}\mathscr{C}_{x}$.
Since $^{\rho,c}\mathscr{A}\left(x\right)$ is monotonic for every
$x\in B_{\rho}$, by Theorem \ref{thm:fixed point for GRLT} we obtain
\[
\sup_{x\in B_{\rho}}\mathbb{P}\left(\left(T_{\left(\rho\right)}\right)^{x}\text{ has no \ensuremath{\left(^{\rho,c}\mathscr{A}\right)^{x}}-\ensuremath{*}-subtree}|\,x\in T_{\left(\rho\right)}\right)<1
\]
whenever $\rho$ is small enough, and in this case,
\begin{equation}
\mathbb{P}\left(\exists x\in T_{\left(\rho\right)}\text{ s.t. }\left(T_{\left(\rho\right)}\right)^{x}\text{ has a \ensuremath{\left(^{\rho,c}\mathscr{A}\right)^{x}}-\ensuremath{*}-subtree}|\text{ nonextinction}\right)=1.\label{eq:*-subtrees exists a.s.}
\end{equation}

Fix $\rho$ small enough s.t. (\ref{eq:*-subtrees exists a.s.}) holds,
and s.t. $\rho^{-\alpha}$ is an integer larger than $\left|\Lambda\right|^{n_{0}}$,
where $n_{0}=\left\lceil \log_{r_{max}}\left(r_{min}\right)^{2}\right\rceil $
as in the proof of Lemma \ref{lem: diffuse subtrees exist}. Then
every $\left(^{\rho,c}\mathscr{A}\right)^{x}\text{-}\ensuremath{*}\text{-subtree}$
contains a $\rho^{-\alpha}$-ary $\left(^{\rho,c}\mathscr{D}\right)^{x}$\emph{-$*$-}subtree.
Indeed, if $A\in\left(^{\rho,c}\mathscr{D}\right)_{v}^{x}$ for some
$v\in B_{\rho}^{x}$, we may remove elements of $A$ except for a
subset of size at most $\left|\Lambda\right|^{n_{0}}$ and obtain
a smaller set which is still in $\left(^{\rho,c}\mathscr{D}\right)_{v}^{x}$.

Now, assume that for some $x\in T_{\left(\rho\right)}$, $\left(T_{\left(\rho\right)}\right)^{x}\text{ has a \ensuremath{\left(^{\rho,c}\mathscr{A}\right)^{x}}-\ensuremath{*}-subtree}$,
then by the above, it also contains a $\rho^{-\alpha}$-ary $\left(^{\rho,c}\mathscr{D}\right)^{x}$\emph{-$*$-}subtree
$S$. Denote $D_{\alpha}^{\prime}=\gamma_{\Phi}\left(\partial S\right)$.
Since $\forall v\in B_{\rho}$, every set in $^{\rho,c}\mathscr{D}_{v}$
is $\left(K,\rho c\right)$ - diffuse, by Proposition \ref{prop: diffuse compressed tree implies diffuse limit set},
$D_{\alpha}^{\prime}$ is hyperplane diffuse, and so is $D_{\alpha}=\varphi_{x}\left(D_{\alpha}^{\prime}\right)\subseteq E$.

In case $\Phi$ satisfies the OSC, Lemma \ref{lem:rho^alpha-ary implies alpha Ahlfors regular}
implies that $D_{\alpha}^{\prime}$ is also $\alpha$-Ahlfors regular
(and so is $D_{\alpha}$). Thus, in this case we have shown that for
every $\alpha\in\left(0,\delta\right)$, a.s. conditioned on nonextinction,
there exists a subset $D_{\alpha}\subseteq E$ which is hyperplane
diffuse and $\alpha$-Ahlfors regular. Taking a sequence $\alpha_{n}\nearrow\delta$
concludes the proof.
\end{proof}
\begin{rem}
The proof of Theorem \ref{thm:Main theorem diffuse subses} for the
case without the OSC could obviously be much shorter since the existence
of a $\mathscr{D}^{\rho,c}$-$*$-subtree suffices. 
\end{rem}

\appendix

\section{Microsets of Galton-Watson fractals\label{appendix:Microsets-of-Galton-Watson}}

\begin{lyxgreyedout}
Should this section be included?%
\end{lyxgreyedout}
We now study the microsets of GWFs where our goal is to show that
in many cases, GWFs are a.s. not hyperplane diffuse (see Corollaries
\ref{cor:GW fractals are not hyperplane diffuse}, \ref{cor:Fractal percolation is not hyperplane diffuse}).
The reader should note that in this paper, we follow the definition
given in \cite{Broderick2012319} for microsets. This definition is
slightly different than the original definition given by Furstenberg
in \cite{Furstenberg2008405}. If one wishes to translate the results
to the latter definition, some minor adjustments need to be made.

A direct consequence of Proposition \ref{prop:0-1 law} is the following.
\begin{prop}
\label{prop:descendant trees are a.s. dense}Let $T$ be a supercritical
GWT on the alphabet $\mathbb{A}$, and let $\mathscr{GW}$ be the
corresponding measure on $\mathscr{T}_{\mathbb{A}}$. Then a.s. conditioned
on nonextinction, 
\[
\overline{\left\{ T^{v}:\,v\in T\right\} }=\text{supp}\left(\mathscr{GW}\right).
\]
\end{prop}

Now, assume that an IFS $\Phi=\left\{ \varphi_{i}\right\} _{i\in\Lambda}$
with attractor $K$ is given. Proposition \ref{prop:descendant trees are a.s. dense}
provides information about convergence in $\mathscr{T}_{\Lambda}$.
But this alone does not provide information about convergence in the
Hausdorff metric of the corresponding sets in $\mathbb{R}^{d}$, since
except for trivial cases, the map $\mathscr{T}_{\Lambda}\to\Omega_{K}$
is not continuous. A more useful observation for this purpose is the
next Proposition which involves the following notation: For every
finite tree $Y\in\mathscr{T}_{\mathbb{A}}$ where $\mathbb{A}$ is
some alphabet, we denote 
\[
\left[Y\right]_{\infty}=\left\{ L\in\mathscr{T}_{\mathbb{A}}:\,L\in\left[Y\right]\text{ and }\forall v\in Y_{\text{length}\left(Y\right)-1},\,L^{v}\text{ is infinite}\right\} .
\]
In other words, for every tree $L\in\mathscr{T}_{\mathbb{A}}$, $L\in\left[Y\right]_{\infty}$
iff $L$ and $Y$ coincide up to the last level of $Y$ (i.e., $L\in\left[Y\right]$),
and every element of $L$ of this level has an infinite line of descendants.
Note that for a supercritical GWT $T$ on the alphabet $\mathbb{A}$,
for every finite tree $F\in\mathscr{T}_{\mathbb{A}}$, $\mathbb{P}\left(T\in\left[F\right]\right)>0\iff\mathbb{P}\left(T\in\left[F\right]_{\infty}\right)>0$.
\begin{defn}
Let $\Phi=\left\{ \varphi_{i}\right\} _{i\in\Lambda}$ be a similarity
IFS. Each $\varphi_{i}$ is a composition of a scaling transformation,
a translation, and an orthogonal transformation which we denote by
$O_{i}$. The closure of the group generated by all the orthogonal
transformations $\left\{ O_{i}\right\} _{i\in\Lambda}$ will be denoted
by $\mathcal{O}_{\Phi}$.
\end{defn}

\begin{prop}
\label{prop:convergence of descendant trees}Let T be a supercritical
GWT with alphabet $\Lambda$, and let $\Phi=\left\{ \varphi_{i}\right\} _{i\in\Lambda}$
be a similarity IFS. Then a.s. conditioned on nonextinction, for every
finite tree $F\in\mathscr{T}_{\Lambda}$ s.t. $\mathbb{P}\left(T\in\left[F\right]\right)>0$,
the following hold:
\end{prop}

\begin{enumerate}
\item $\exists v\in T$ s.t. $T^{v}\in\left[F\right]_{\infty}$
\item If $\mathcal{O}_{\Phi}$ is finite, then $\forall O\in\mathcal{O}_{\Phi}$,
$\exists v\in T$ s.t. $T^{v}\in\left[F\right]_{\infty}$ and $O_{v}=O$.
\end{enumerate}
\begin{proof}
(1) is immediate. For (2), denote $l=\left|O_{\Phi}\right|$, and
notice that $O_{\Phi}=\left\{ O_{i}:\,\left|i\right|\leq l\right\} $%
. Denote $p=\min\left\{ \mathbb{P}\left(i\in T\right):\,i\in\bigcup\limits _{n=0}^{l}\Lambda^{n}\right\} $,
and since we assume that $\forall i\in\Lambda,\,\mathbb{P}\left(i\in T\right)>0$,
we have $p>0$. For every $v\in\Lambda^{*}$ and $O\in\mathcal{O}_{\Phi}$,
\[
\mathbb{P}\left(\exists w\in T^{v}\text{ s.t. }O_{vw}=O|\,v\in T\right)\geq p.
\]
Let $F\in\mathscr{T}_{\Lambda}$ and $O\in\mathcal{O}_{\Phi}$ be
as above. Then for every $v\in\Lambda^{*}$, 
\[
\mathbb{P}\left(\exists w\in T^{v}\text{ s.t. }O_{vw}=O\text{ and }T^{vw}\in\left[F\right]_{\infty}|\,v\in T\right)\geq p\cdot\mathbb{P}\left(T\in\left[F\right]_{\infty}\right).
\]
Using Corollary \ref{cor:kesten - stigum}, we get that a.s. conditioned
on nonextinction, $\exists m\in T$, s.t. $O_{m}=O$ and $T^{m}\in\left[F\right]_{\infty}$.
Since the set of finite trees in $\mathscr{T}_{\Lambda}$ is countable,
and $\mathcal{O}_{\Phi}$ is finite, the order of the quantifiers
may be reversed.
\end{proof}
The advantage of Proposition \ref{prop:convergence of descendant trees}
over Proposition \ref{prop:descendant trees are a.s. dense} will
become clear in Lemma \ref{lem:Hausdorff convergence of ball intersections},
but first we need the following compactness argument.

\begin{lem}
\label{lem:sequential compactness of trees}Let $T$ be some infinite
tree, and let $v^{n}\in T$ be some sequence of nodes of $T$ with
$\left|v^{n}\right|\to\infty$, then there exist $w\in\partial T$
and a sequence $n_{k}\to\infty$, s.t. 
\[
\max\left\{ \left|u\right|:\,u\leq v^{n_{k}}\text{ and }u\leq w\right\} \underset{k\to\infty}{\longrightarrow}\infty.
\]
\end{lem}

\begin{proof}
Let the set $\mathbb{A}^{\prime}=\mathbb{A}\cup\left\{ \star\right\} $
be equipped with the discrete topology. Denote for each $v\in\mathbb{A}^{*}$,
$\widetilde{v}=v\star\star...\in\left(\mathbb{A}^{\prime}\right)^{\mathbb{N}}$.
Since $\left(\mathbb{A}^{\prime}\right)^{\mathbb{N}}$ with the product
topology is compact, the sequence $\widetilde{v^{n}}$ has some convergent
subsequence $\widetilde{v^{n_{k}}}\to w\in\left(\mathbb{A}^{\prime}\right)^{\mathbb{N}}$.
Since $\forall k\in\mathbb{N}$, $\left(\widetilde{v^{n_{k}}}\right)_{1},...,\left(\widetilde{v^{n_{k}}}\right)_{\left|v^{n_{k}}\right|}\neq\star$
and since $\left|v^{n_{k}}\right|\to\infty$, $w\in\text{\ensuremath{\mathbb{A^{N}}}}$.
\end{proof}
\begin{lem}
\label{lem:Hausdorff convergence of ball intersections}Let $\Phi=\left\{ \varphi_{i}\right\} _{i\in\Lambda}$
be a similarity IFS, and $S\in\mathscr{T}_{\Lambda}$ some infinite
tree. Let $\left(T_{\left[n\right]}\right)_{n\in\mathbb{N}}$ be a
sequence of trees s.t. $T_{\left[n\right]}\in\left[\bigcup\limits _{i=0}^{n}S_{i}\right]_{\infty}$
for every $n\in\mathbb{N}$. Let $B=\overline{B_{\rho}\left(x\right)}$
be some closed ball s.t. $x\in\text{\ensuremath{\gamma}}_{\Phi}\left(\partial S\right)$.
Then there exists a sequence of closed balls $B_{n}=\overline{B_{\rho_{n}}\left(x_{n}\right)}$,
s.t. $x_{n}\in\gamma_{\Phi}\left(\partial T_{\left[n\right]}\right)$
for every $n\in\mathbb{N}$, $x_{n}\to x$, $\rho_{n}\to\rho$, and
\[
d_{H}\left(B_{n}\cap\gamma_{\Phi}\left(\partial T_{\left[n\right]}\right),\,B\cap\gamma_{\Phi}\left(\partial S\right)\right)\underset{n\to\infty}{\longrightarrow}0.
\]
\end{lem}

\begin{proof}
Let $\rho>0$ and $x\in\text{\ensuremath{\gamma}}_{\Phi}\left(\partial S\right)$
be the radius and the center point of $B$. %
Let $x_{n}\in\gamma_{\Phi}\left(\partial T_{\left[n\right]}\right)$
satisfy $\left\Vert x_{n}-x\right\Vert \leq\Delta r_{max}^{n}$ for
every $n$, where $\Delta$ is the diameter of the attractor of $\Phi$
(such $x_{n}$ exist since $T_{\left[n\right]}\in\left[\bigcup\limits _{i=0}^{n}S_{i}\right]_{\infty}$).
Define $\rho_{n}=\rho+2\Delta r_{max}^{n}$ and let $B_{n}=\overline{B_{\rho_{n}}\left(x_{n}\right)}$.
Fix some $\varepsilon>0$.

Assume we are given a point $y\in B\cap\gamma_{\Phi}\left(\partial S\right)$.
As before, for every $n\in\mathbb{N}$ there exists $y_{n}\in\gamma_{\Phi}\left(\partial T_{\left[n\right]}\right)$
s.t. $\left\Vert y_{n}-y\right\Vert \leq\Delta r_{max}^{n}$.%
{} $\left\Vert y_{n}-x_{n}\right\Vert \leq\left\Vert y_{n}-y\right\Vert +\left\Vert y-x\right\Vert +\left\Vert x-x_{n}\right\Vert \leq\rho+2\Delta r_{max}^{n}$,
hence $y_{n}\in B_{n}\cap\gamma_{\Phi}\left(\partial T_{\left[n\right]}\right)$.
Choosing $N\in\mathbb{N}$ large enough (does not depend on $y$),
for every $n>N$, $\left\Vert y_{n}-y\right\Vert <\varepsilon$.

On the other hand, we need to show that for some large enough $N\in\mathbb{N}$,
for every $n>N$, for every point $y_{n}\in B_{n}\cap\gamma_{\Phi}\left(\partial T_{\left[n\right]}\right)$,
there exists a point $y\in B\cap\gamma_{\Phi}\left(\partial S\right)$
s.t. $\left\Vert y_{n}-y\right\Vert <\varepsilon$. Assume this is
false, so there is a sequence $y_{n_{k}}\in B_{n_{k}}\cap\gamma_{\Phi}\left(\partial T_{\left[n_{k}\right]}\right)$
for some sequence $n_{k}\to\infty$ s.t. for every $y\in B\cap\gamma_{\Phi}\left(\partial S\right)$,
$\left\Vert y_{n_{k}}-y\right\Vert \geq\varepsilon$ for every $k$.
Since $T_{\left[n_{k}\right]}\in\left[\bigcup\limits _{i=0}^{n_{k}}S_{i}\right]_{\infty}$for
every $k$, there exists $w^{n_{k}}\in S_{n_{k}}$ s.t. $y_{n_{k}}\in\gamma_{\Phi}\left(\left[w^{n_{k}}\right]\right)$
for every $k$%
. According to Lemma \ref{lem:sequential compactness of trees}, by
taking a subsequence we may assume that there exists some $u\in\partial S$
s.t. $\max\left\{ \left|t\right|:\,t\leq w^{n_{k}}\text{ and }t\leq u\right\} \underset{k\to\infty}{\longrightarrow}\infty$
which implies that $y_{n_{k}}\underset{k\to\infty}{\longrightarrow}\gamma_{\Phi}\left(u\right)$.
Also, since $\left\Vert y_{n_{k}}-x\right\Vert \leq\rho+3\Delta r_{max}^{n}$%
, $\left\Vert \gamma_{\Phi}\left(u\right)-x\right\Vert \leq\rho$,
hence $\gamma_{\Phi}\left(u\right)\in B\cap\gamma_{\Phi}\left(\partial S\right)$
which contradicts our assumption.
\end{proof}
\begin{rem}
Note that if in the statement of Lemma \ref{lem:Hausdorff convergence of ball intersections}
we remove the restriction that the centers of the balls must be elements
of the corresponding sets, i.e., we allow arbitrary $x$ and $\left(x_{n}\right)_{n\in\mathbb{N}}$,
we still need to take shrinking balls converging to $B$, since although
$d_{H}\left(\gamma_{\Phi}\left(\partial T_{\left[n\right]}\right),\,\gamma_{\Phi}\left(\partial S\right)\right)\underset{n\to\infty}{\longrightarrow}0$,
it is not true in general that $d_{H}\left(B\cap\gamma_{\Phi}\left(\partial T_{\left[n\right]}\right),\,B\cap\gamma_{\Phi}\left(\partial S\right)\right)\underset{n\to\infty}{\longrightarrow}0.$
\end{rem}

The proof of the following Lemma is an easy exercise.
\begin{lem}
\label{lem:Ball from OSC }Let $\Phi=\left\{ \varphi_{i}\right\} _{i\in\Lambda}$
be a similarity IFS, and let $T\in\mathscr{T}_{\Lambda}$ be an infinite
tree. Let $U\subset\mathbb{R}^{d}$ be an OSC set for $\Phi$, and
$B\subseteq U$ be any subset of $U$. Then for every $v\in T$, 
\[
\varphi_{v}\left(B\right)\cap\gamma_{\Phi}\left(\partial T\right)=\varphi_{v}\left(B\cap\gamma_{\Phi}\left(\partial T^{v}\right)\right).
\]
\end{lem}

\begin{prop}
\label{prop:microsets up to rotation}Let $E$ be a Galton-Watson
fractal with respect to a similarity IFS $\Phi=\left\{ \varphi_{i}\right\} _{i\in\Lambda}$
which satisfies the OSC, and let $\mathscr{GW}$ be the corresponding
measure on $\mathscr{T}_{\Lambda}$. Let $U\subset\mathbb{R}^{d}$
be an OSC set for $\Phi$. Then a.s. conditioned on nonextinction,
for every infinite tree $S\in\text{supp}\left(\mathscr{GW}\right)$%
, $\exists O\in\mathcal{O}_{\Phi}$ s.t. for every closed ball $B\subseteq U$
centered in $\gamma_{\Phi}\left(\partial S\right)$, $O\circ F_{B}\left(\gamma_{\Phi}\left(\partial S\right)\cap B\right)$
is a microset of $E$. Moreover, if in addition $\mathcal{O}_{\Phi}$
is finite, then a.s. conditioned on nonextinction, for every $S$
and $B$ as above, for every $O\in\mathcal{O}_{\Phi}$, $O\circ F_{B}\left(\gamma_{\Phi}\left(\partial S\right)\cap B\right)$
is a microset of $E$.
\end{prop}

The reader should recall that $F_{B}$ is defined as the unique homothety
mapping the closed ball $B$ to the closed unit ball $\overline{B_{1}\left(0\right)}$.
\begin{proof}
Let $T$ be the corresponding GWT. By Proposition \ref{prop:convergence of descendant trees},
a.s. conditioned on nonextinction, $\forall S\in\text{supp}\left(\mathscr{GW}\right)$
there exists a sequence $v^{n}\in T$ s.t. $T^{v^{n}}\in\left[\bigcup\limits _{i=0}^{n}S_{i}\right]_{\infty}$
for every $n\in\mathbb{N}$. Given such $S$ and sequence $v^{n}$,
for every closed ball $B\subseteq U$ centered in $\gamma_{\Phi}\left(\partial S\right)$,
by Lemma \ref{lem:Hausdorff convergence of ball intersections}, $\gamma_{\Phi}\left(\partial T^{v^{n}}\right)\cap B_{n}\underset{n\to\infty}{\longrightarrow}\gamma_{\Phi}\left(\partial S\right)\cap B$
for some sequence of closed balls $B_{n}$ centered in $\gamma_{\Phi}\left(\partial T^{v^{n}}\right)$,
whose radii and centers converge to those of $B$. Since $B\subseteq U$,
then for large enough values of $n$%
, $B_{n}\subseteq U$, so we may assume without loss of generality
that $\forall n\in\mathbb{N},\,B_{n}\subseteq U$. By Lemma \ref{lem:Ball from OSC },
$\varphi_{v^{n}}\left(\gamma_{\Phi}\left(\partial T^{v^{n}}\right)\cap B_{n}\right)=E\cap\varphi_{v^{n}}\left(B_{n}\right)$
for every $n$. Since each $\varphi_{v^{n}}$ is a similarity map,
$\varphi_{v^{n}}\left(B_{n}\right)$ is a closed ball, and $F_{\varphi_{v^{n}}\left(B_{n}\right)}=F_{O_{v^{n}}\left(B_{n}\right)}\circ O_{v^{n}}\circ\varphi_{v^{n}}^{-1}$
is the homothety%
{} mapping $\varphi_{v^{n}}\left(B_{n}\right)$ to the closed unit ball.
Thus,
\[
\begin{array}{l}
F_{\varphi_{v^{n}}\left(B_{n}\right)}\left(E\cap\varphi_{v^{n}}\left(B_{n}\right)\right)=\\
F_{O_{v^{n}}\left(B_{n}\right)}\circ O_{v^{n}}\left(\gamma_{\Phi}\left(\partial T^{v^{n}}\right)\cap B_{n}\right)=\\
O_{v^{n}}\circ F_{B_{n}}\left(\gamma_{\Phi}\left(\partial T^{v^{n}}\right)\cap B_{n}\right)
\end{array}
\]
is a miniset of $E$ for every $n$. By compactness of $\mathcal{O}_{\Phi}$,
there exists a subsequence $\left(v^{n_{k}}\right)$ s.t. $O_{v^{n_{k}}}\underset{k\to\infty}{\longrightarrow}O$
where $O\in\mathcal{O}_{\Phi}$. It follows that 
\[
O_{v^{n_{k}}}\circ F_{B_{n_{k}}}\left(\gamma_{\Phi}\left(\partial T^{v^{n_{k}}}\right)\cap B_{n_{k}}\right)\underset{k\to\infty}{\longrightarrow}O\circ F_{B}\left(\gamma_{\Phi}\left(\partial S\right)\cap B\right)
\]
in the Hausdorff metric, and hence $O\circ F_{B}\left(\gamma_{\Phi}\left(\partial S\right)\cap B\right)$
is a microset of $E$ which concludes the proof of (1).

In order to prove (2), assume that $\mathcal{O}_{\Phi}$ is finite
and fix any $\tilde{O}\in\mathcal{O}_{\Phi}$. By (2) of Proposition
\ref{prop:convergence of descendant trees}, we could ensure that
for every $n\in\mathbb{N}$, $O_{v^{n}}=\tilde{O}$.

\end{proof}
\begin{cor}
\label{cor:GW fractals are not hyperplane diffuse}Let $E$ be a Galton-Watson
fractal with respect to a similarity IFS $\Phi=\left\{ \varphi_{i}\right\} _{i\in\Lambda}$,
and let $T$ be the corresponding GWT. Assume $\Phi$ satisfies the
OSC, and let $U\subset\mathbb{R}^{d}$ be an OSC set for $\Phi$.
If $\exists S\in\text{supp}\left(\mathscr{GW}\right)$ s.t. $\gamma_{\Phi}\left(\partial S\right)\cap U$
is nonempty and contained in an affine hyperplane, then a.s. $E$
is not hyperplane diffuse.%
\begin{lyxgreyedout}
Should the discussion include examples which satisfy the assumptions
of the corollary? For example, if P(|W|=0)>0, then every singleton
of the fractal is a microset.%
\end{lyxgreyedout}
\end{cor}

\begin{cor}
\label{cor:Fractal percolation is not hyperplane diffuse}Let $E$
be the limit set of a supercritical fractal percolation process. Then
a.s. conditioned on nonextinction, every closed subset of $B_{1}\left(0\right)$
which intersects the origin is a microset of $E$ (in particular,
E is a.s. not hyperplane diffuse).
\end{cor}

\begin{proof}
This follows immediately from Proposition \ref{prop:microsets up to rotation}
since in the case of fractal percolation, $\text{supp}\left(\mathscr{GW}\right)=\mathscr{T}_{\mathbb{A}}$
and $\mathcal{O}_{\Phi}$ is trivial.
\end{proof}

\section{A counter example for equality in Theorem \ref{thm:fixed point for GRLT}}

We now show an example of a random $*$-tree for which equality in
equation \ref{eq:fixed point for *-trees} does not hold.
\begin{example}
Let $\mathbb{A}=\left\{ a_{1},a_{2},...,a_{n}\right\} $ be an alphabet
and let $B=\left\{ \emptyset\right\} \cup\left\{ a_{1},a_{2}\right\} \mathbb{A}^{*}\subset\mathbb{A}^{*}$
be a $*$-tree. Let $S\subset\mathbb{A}^{*}$ be a random $*$-tree
on $B$ with offspring distributions $\left\{ M_{i}\right\} _{i\in B}$
given by:
\begin{itemize}
\item $\mathbb{P}\left(M_{i}=\mathbb{A}\right)=1$ for $i\in\left\{ a_{1},a_{2}\right\} $.
\item $\forall i\in B,\,\left|i\right|\geq2\,\implies M_{i}\sim\text{Bin}\left(\left\{ a_{1},a_{2},a_{3}\right\} ,p\right)$
where $p\in\left(0,1\right)$ is large enough so that a GWT with alphabet
of size 3 and binomial offspring distribution with parameter $p$
has a positive probability, $\alpha>0$, of containing a binary subtree
(by Pakes-Dekking theorem there exists such $p$).
\item $\mathbb{P}\left(M_{\emptyset}=\left\{ a_{1},a_{2}\right\} \right)=$$\alpha+\varepsilon,\,\mathbb{P}\left(M_{\emptyset}=\left\{ a_{1}\right\} \right)=1-\left(\alpha+\varepsilon\right)$
for some small $\varepsilon>0$.
\end{itemize}
Now, define $\mathscr{A}_{x}=\left\{ L\subseteq\mathbb{A}^{*}:\,\left|L\right|\geq2\right\} $
for every $x\in B$, so that $\mathscr{A}$-$*$-trees are $*$-trees
which contain binary trees.

Choosing $n$ large enough, we may guarantee that the sup in equation
(\ref{eq:fixed point for *-trees}) is realized by every element $x\in B$
with $\left|x\right|\geq2$ and its value is $1-\alpha$, that is
to say that $q=1-\alpha$ where $q$ is as defined in the proof of
Theorem \ref{thm:fixed point for GRLT}. In that proof we have shown
that $q\leq s_{0}$ where $s_{0}$ is the smallest fixed point of
the function $g_{\mathscr{A}}\left(s\right)=\sup\limits _{x\in B^{\prime}}\mathbb{P}\left(M_{x}^{\left(s\right)}\notin\mathscr{A}_{x}\right)$
in $\left[0,1\right]$. 

Analyzing $g_{\mathscr{A}}\left(q\right)=\sup\limits _{x\in B^{\prime}}\mathbb{P}\left(M_{x}^{\left(q\right)}<2\right)$,
one should notice that the sup in the formula for $g_{\mathscr{A}}\left(q\right)$
is realized by $x=\emptyset$. This is because $\alpha<\mathbb{P}\left(\text{Bin}\left(3,p\right)\geq2\right)$
and $\varepsilon$ may be chosen to be arbitrarily small. So $g_{\mathscr{A}}\left(q\right)=\mathbb{P}\left(M_{\emptyset}^{\left(q\right)}<2\right)=1-\left(\left(\alpha+\varepsilon\right)\cdot\alpha^{2}\right)$
which is strictly larger than $q$ when $\varepsilon$ is small enough.
\end{example}

\bibliographystyle{abbrv}
\bibliography{all}

\end{document}